\documentclass[a4paper]{article}
\usepackage[margin=3cm]{geometry}
\usepackage{amsmath,amssymb,amsthm, amsfonts}
\usepackage[english]{babel}
\usepackage{hyperref}
\usepackage{graphicx} 
\usepackage{tikz}
\usepackage{tikz-cd}
\usetikzlibrary{math} 
\usepackage{mathrsfs}
\usepackage{mathtools}
\usepackage{enumitem}
\usepackage{cite}
\usepackage{bbm}
\usepackage{bm}
\usepackage{stmaryrd}
\usepackage{mathrsfs}
\usepackage{subcaption}
\usetikzlibrary{decorations.pathmorphing, shapes, ext.transformations.mirror}
\usepackage{tabularx}
\usepackage{array}
\usepackage{xcolor}
\usepackage{caption}
\usepackage{float}

\usepackage{titlesec}

\titleformat{\part}[block]
  {\filcenter\LARGE\bfseries}
  {Part \Roman{part}:}
  {0.6em}
  {}

\newcommand{\unit}{{\mathrm{1}\mkern-4mu{\mathchoice{}{}{\mskip-0.5mu}{\mskip-1mu}}\mathrm{l}}}
\newcommand\Nakayama{\operatorname{N}^r}
\newcommand\Radford{\delta}
\newcommand\BRF{\mathfrak{B}}
\newcommand\XX{\operatorname{X}}
\newcommand\Cb{\operatorname{Cb}}
\newcommand\Cd{\operatorname{Cd}}
\newcommand\mt{\operatorname{t}}
\newcommand{\on}{\operatorname}
\newcommand\D{\mathbb{D}}

\newcommand\BB{\mathcal{B}}
\renewcommand\SS{\mathcal{S}}
\newcommand\EE{\mathcal{E}}
\newcommand\II{\mathcal{I}}
\newcommand\CC{\mathcal{C}}
\newcommand\ZZ{\mathcal{Z}}

\newcommand\Zz{\mathbb{Z}}

\newcommand\rev{\operatorname{ev}}
\newcommand\rcoev{\operatorname{coev}}
\newcommand\lev{\widetilde{\operatorname{ev}}}
\newcommand\lcoev{\widetilde{\operatorname{coev}}}
\newcommand\apiv{\mathsf{p}}

\DeclareMathOperator{\Bord}{Bord}
\DeclareMathOperator{\SN}{SN}
\DeclareMathOperator{\Proj}{Proj}

\DeclareMathOperator{\Vect}{Vect}

\DeclareMathOperator{\Bimod}{Bimod}

\DeclareMathOperator{\Id}{Id}
\DeclareMathOperator{\Hom}{Hom}
\DeclareMathOperator{\End}{End}
\DeclareMathOperator{\Rep}{Rep}

\newcommand{\orSN}{\mathcal{SN}}

\newcommand{\cusp}{
    \tikzmath{\gap=1.4;}
    \def\sadlines{2}
    \draw[smooth, gray!70] plot coordinates {(-2,1) (0,0) (1,-\gap) (0,-\gap-1) (-1,-\gap) (0,0)(2,1)};

    \foreach \y in {0,...,\sadlines}{
    \tikzmath{\step=\y/(\sadlines+1);}
    \draw[smooth, gray!70] plot coordinates {(-2,1.5-\step/2) (0,1.5-1.5*\step) (2,1.5-\step/2)}; 
    }

    \draw[gray!70] (-2,-2.6-\gap) -- (2,-2.6-\gap);
    \draw[smooth, gray!70] plot coordinates {(-2,-1.8-\gap)(0,-2.4-\gap) (2,-1.8-\gap)};
     \draw[smooth, gray!70] plot coordinates {(-2,-1-\gap) (0,-2.2-\gap) (2,-1-\gap)};
    \draw[smooth, gray!70] plot coordinates {(-2,-\gap) (0,-2-\gap) (2,-\gap)};

    \foreach \y in {1,...,\sadlines}{
    \tikzmath{\step=\gap/(\sadlines+1);}
    \tikzmath{\x=2*\y/(\sadlines+1);}
    \draw[smooth, gray!70] plot coordinates {(-2,1-\y*\step-0.5*\step) (-\x,-0.15*\gap*\x*\x) (-1-0.45*\x,-\gap) (0,-\gap-1 - 0.5*\x) (1+0.45*\x,-\gap) (\x,-0.15*\gap*\x*\x) (2,1-\y*\step-0.5*\step)}; 
    }
   
    \foreach \y in {1,...,\sadlines}{
    \tikzmath{\step=\y/(\sadlines+1);}
    \draw[gray!70] plot[smooth cycle] coordinates {(-1+\step,-\gap) (0,-\gap*\step) (1-\step,-\gap) (0,-\gap-1+\step)};
    
    \node[star,  star point ratio=2.25, fill=black, inner sep = 1pt] at (0,0) {};
    \node[star,  star point ratio=2.25, fill=black, inner sep = 1pt] at (0,-\gap) {};
    }
    }

\newtheorem{counter}{Counter}
\newtheorem{theorem}{Theorem}
\newtheorem{theoremA}{Theorem}

\newtheorem{lemma}[counter]{Lemma}
\newtheorem{corollary}[counter]{Corollary}
\newtheorem{proposition}[counter]{Proposition}
\newtheorem{conjecture}[counter]{Conjecture}
\theoremstyle{definition}
\newtheorem{definition}[counter]{Definition}
\newtheorem{example}[counter]{Example}
\newtheorem{notation}[counter]{Notation}
\newtheorem{remark}[counter]{Remark}
\newtheorem{warning}[counter]{Warning}

\numberwithin{counter}{section}

\title{\vskip-20pt String nets for twisted pivotal categories}
\author{Benjamin Ha\"ioun, William Stewart and Filippos Sytilidis}
\date{\today}
\begin{document}
\maketitle
\begin{abstract}
    We develop a graphical calculus for monoidal categories equipped with twisted pivotal structures, which are a generalization of pivotal structures originating from the study of orientation structures in the context of the Cobordism Hypothesis. 

    This graphical calculus depends on a possibly singular foliation, and we use it to construct twisted string net modules for surfaces equipped with a Morse function or a Morse foliation. We prove that, despite the apparent dependence on this Morse function, the twisted string net modules assemble in an oriented categorified 2-TQFT.

    We study when the twisted string net module of the 2-sphere vanishes, relate it to the distinguished invertible object for finite tensor categories and exhibit examples of non-unimodular finite tensor categories with non-vanishing twisted string net module on the 2-sphere. This vanishing is expected to be the main obstruction for extending our categorified 2-TQFT to a non-compact 3-TQFT.
\end{abstract}
\setcounter{tocdepth}{2} 
\tableofcontents
\section{Introduction}
Graphical calculus for rigid monoidal categories, as developed in \cite{JoyalStreetTensorCalculus}, see \cite{Selinger} for a survey, has been a very successful tool in constructing topological invariants. Under suitable hypotheses and structure on the input category, they yield string net modules of surfaces \cite{LevinWen, KirillovStringNetTV}, and Turaev--Viro 3-manifold invariants \cite{TuraevViroStateSum}.  

The extra structure which is typically used in these constructions is a \textit{pivotal structure}, i.e. an identification of the double dual functor $(-)^{**}$ with the identity functor as monoidal functors. The original hypothesis on the input category $\BB$ were semisimplicity for the string net constructions, and finiteness along with sphericality for the 3-manifold invariants. Recent developments allow for string net modules associated to non-semisimple pivotal categories \cite{CGPAdmissbleskein, MSWY23}, which we will refer to as \textit{admissible string net modules}. The input for such admissible string net modules is a monoidal pivotal category $\BB$ together with a two-sided tensor ideal $\II \subseteq \BB$, which the reader is welcome to think of as the ideal of projective objects in a non-semisimple category. The Turaev--Viro 3-manifold invariants have also been adapted to the non-semisimple setting \cite{GPTmodified6jsymbols, CGPVnc2plus1}, and yield partially-defined TQFTs, sometimes called non-compact TQFTs. This construction is based on the theory of modified traces \cite{GeerPatureauTuraevModifiedqdim, GKPmtrace} and chromatic morphisms \cite{CGPVchromatic}. Along with finiteness, one key assumption needed in this construction is \textit{unimodularity}.

String net modules and their admissible versions transport naturally under orientation-preserving diffeomorphisms of surfaces, and yield representations of Mapping Class Groups of surfaces. It is shown in \cite{MSWY23} that admissible string net modules of finite pivotal categories recover Lyubashenko's non-semisimple Mapping Class Group representations \cite{Lyub3mfldProjMCG} associated to the Drinfeld center of $\BB$.

Without a pivotal structure, the graphical calculus of \cite{JoyalStreetTensorCalculus} needs a notion of a progressive direction, and string net modules can only be defined for framed or foliated surfaces \cite{KST24}. The point-of-view we adopt is that pivotal structures give us a way to descend from foliated surfaces to oriented ones. In this paper, we will explore all the possible ways to do so, and show that there are indeed some that do not come from pivotal structures.

A very different approach to this topic is offered by the cobordism hypothesis \cite{BaezDolan, LurieCob}. It states that fully extended framed TQFTs are classified by their value on the point, and that the dimension of the TQFT, i.e. the integer $d$ so that our TQFT is a functor out of $\Bord_d^{\operatorname{fr}}$, is controlled by the dualizability of this object. In our context, a good model for the target category $3\Vect$ is the Morita 3-category of monoidal categories \cite{JFS, Haugseng}, and the value on the point is the input category $\BB$. 

Dualizability in this 3-category has been studied in \cite{DSPS}. They prove that finite semisimple monoidal categories are 3-dualizable and that rigid non-semisimple categories are 2-dualizable, which matches the constructions above. They also observe that finite non-semisimple categories are very close to being 3-dualizable, which matches the non-compact TQFT constructions of \cite{CGPVnc2plus1}. However, the unimodularity condition does not appear in their dualizability study, which suggests that there might be a non-unimodular analogue of these topological constructions. 

An important point to note here is that the cobordism hypothesis as stated above does not classify oriented TQFTs, but framed ones. In other words, the study of \cite{DSPS} only predicts framed analogues of \cite{CGPVnc2plus1}. There is an oriented version of the cobordism hypothesis, which classifies fully extended oriented $k$-TQFTs by $k$-dualizable objects equipped with an $SO(k)$-structure, which is an $SO(k)$-homotopy-fixed-point-structure for some action of $SO(k)$ defined using the framed cobordism hypothesis. In particular, it proposes a complete answer to the question above about all the ways to descend string net modules from framed to oriented surfaces; they should correspond to $SO(2)$-structures.
It is in general very difficult to describe $SO(k)$-structures, in large part because they use the cobordism hypothesis in their definition. However, for the special case of $SO(2)$ in $3\Vect$, there is a very good conjectural description, see \cite{SchomPriesDualLowdimHighCat, DSPS}. They observe that a pivotal structure indeed induces an $SO(2)$-structure, but that there are many more.

\subsection*{Results}
In this paper, we develop a graphical calculus for monoidal categories which are not necessarily pivotal, but equipped with a more general $SO(2)$-structure. For $\alpha$ an invertible element in a monoidal category $\BB$, we introduce the notion of an $\alpha$-twisted pivotal structure in Definition \ref{def:alphaTwistedPivStr}, which follows the conjectural description of $SO(2)$-structures from \cite{SchomPriesDualLowdimHighCat, DSPS} and relates to various notions in the literature \cite{HalbigZormanPivotalityTwisted, ShimizuPivotalOnDrinfeld, KST24}. 

We develop a graphical calculus for surfaces equipped with possibly singular foliations, giving rise to the notion of twisted string net module $\SN_\II^\alpha(\Sigma)$ for such a surface $\Sigma$, see Definition \ref{def:twistedStringNetModules}. In particular, a Morse function on an oriented surface $S$ induces such a singular foliation.
We show that this graphical calculus descends to oriented surfaces. Our main result is as follows:
\begin{theoremA}[See Theorem \ref{thm:main_thm}]
    Given a tensor ideal $\II$ in an $\alpha$-twisted-pivotal category $\BB$, we construct an oriented categorified 2-TQFT, i.e. a symmetric monoidal functor 
    $$\orSN_\II^\alpha:\Bord_{12\sim}^{\on{or}} \to \Bimod^{hop}$$
    such that $\orSN_\II^\alpha(S)$ is canonically isomorphic to $\SN_\II^\alpha(\Sigma)$ where $\Sigma$ is the surface $S$ equipped with a singular foliation arising from any choice of Morse function on $S$. 
\end{theoremA}
Here $\Bord_{12\sim}^{\on{or}}$ is the $(2,1)$-category of oriented 1-manifold, oriented 2-cobordisms and isotopy classes of diffeomorphisms, and $\Bimod$ is the bicategory of linear categories, profunctors and natural transformations.

The main ingredient in the construction of this functor is the presentation of the symmetric monoidal bordism $(2,1)$-category $\Bord_{12\sim}^{\on{or}}$ in terms of Morse and Cerf theory established in \cite{Filippos}, and recalled in Section \ref{subsec:Morse_data+sur}. We need to assign a value to each generating 0, 1 and 2-morphism, which we do as follows:
    \begin{description}
    \item[Objects:] To a connected oriented 1-manifold $\Gamma$ we associate the category $\SN_\II^\alpha(\Gamma)$ introduced in \cite{KST24}. We recall their construction and adapt it to the admissible setting in Section \ref{sec:rigidGraphicalCalculus}.
    \item[1-morphisms:] Every generating 1-morphism gives rise to a $2$-dimensional bordism $\Sigma \colon \Gamma_{in} \to \Gamma_{out}$ equipped with a Morse function, and in particular with the associated Morse singular foliation, to which we associate the twisted string net bimodule
    \[
    \SN_{\II}^{\alpha}(\Sigma) : \SN_\II^\alpha(\Gamma_{in})^{op} \otimes \SN_\II^\alpha(\Gamma_{out}) \to \Vect
    \]
    from $\SN_\II^\alpha(\Gamma_{in})$ to $\SN_\II^\alpha(\Gamma_{out})$. See Section \ref{sec:twistedGraphicalCalculus}.
    \item[2-morphisms:] Most generating 2-morphisms induce foliation-preserving diffeomorphisms which naturally act on twisted string net modules. There are 3 generators, coming from the creation, cancellation or crossing of two critical points, to which we assign specific maps described in Section \ref{subsec:generators}.
    \end{description}
We then need to check a certain number of relations, which we do in Section \ref{subsec:relations+result}.

\begin{remark}
The cobordism hypothesis is only used as guidance in the present work, and not in the technical construction of the TQFT. 
\end{remark}

We investigate examples for our construction, and provide examples satisfying the following.
\begin{theoremA}[Examples \ref{ex:noPivotalStructure} and \ref{ex:NonunimodTwSpherical}]
    There exists finite tensor categories $\BB$ which:
    \begin{enumerate}
        \item admit a twisted-pivotal structure but no pivotal structure
        \item are non-unimodular, but admit twisted pivotal structures for which $\orSN_\II^\alpha(S^2) \neq 0$.
    \end{enumerate}
\end{theoremA}
Let us comment on the relevance of the second statement here. As discussed above, one limiting condition for building 3-TQFTs out of pivotal categories is unimodularity. For a finite tensor category $\BB$, with $\II = \Proj(\BB)$, the admissible string net module of the 2-sphere is dual to the space of modified traces which is non-zero if and only if $\BB$ is unimodular \cite{GKPmtrace}. If the vector space associated to $S^2$ is zero, there is no hope of extending string net modules to a 3-TQFT, even a non-compact one, as every 3-manifold factors through $S^2$. Therefore, untwisted admissible string net constructions cannot build 3-manifold invariants out of non-unimodular categories, but our construction could. 


\subsection*{A few words on how this project started}
As explained in \cite{WalkerNotes, HaiounWRTdelCY}, the Witten--Reshetikhin--Turaev 3-TQFTs are best understood as a boundary condition to their anomaly, the Crane--Yetter theory (or in the non-semisimple case, the \cite{CGHP} theory). In this setting, WRT can be understood as a fully extended theory, with values in the arrow category $\operatorname{BrTens}^\to$ of \cite{JFS, BJS}, and associates to the point the regular module $\CC_\CC$ over the input modular tensor category $\CC$, see \cite{FreedSlides, HaiounUnit}. 

In this description, the boundary condition, given by the regular module, is simply the inclusion of the empty skein in every skein module, and is very easy to describe. On the other hand, the anomaly is quite non-trivial to build, and in particular involves the Kirby color of \cite{TuraevBook} or its non-semisimple analogue \cite{Lyub3mfldProjMCG}. However, the resulting 4-TQFT is invertible, so it seems like the anomaly 4-TQFT should have a much simpler description.

When $\CC = \operatorname{Rep}_qG$ is the (non-semisimple) modular tensor category of representations of the small quantum group, this can be observed by the fact that this (quite non-trivial) braided category is Morita equivalent (so isomorphic in $\operatorname{BrTens}$) to the semisimple braided category $\Rep_qT$ of representations of the Cartan $T$. The Morita equivalence is given by the small quantum Borel $\BB=\Rep_qB$. So the 4-TQFT associated to $\Rep_qG$ can equivalently be described by the much simpler $\Rep_qT$ theory. The WRT boundary condition however becomes the composite $$\Rep_qG\underset{\Rep_qG}{\boxtimes}\Rep_qB \simeq \Rep_qB$$ as a boundary condition to $\Rep_qT$, which is a lot more interesting than the regular boundary condition. 

From this discussion, and applying the Cobordism Hypothesis quite liberally, we conclude that the WRT theory should have an equivalent description as $\Rep_qB$ string nets living at the boundary of a $\Rep_qT$-skein-theory bulk, in the spirit of \cite{BrownJordanParabolicSkein}. These categories respectively come with a pivotal and a ribbon structure, and we naively use them for our skein modules. A quick computation, well-known to non-semisimple quantum topologists, shows that this $\Rep_qB$ boundary condition cannot be extended to 3-manifolds. Indeed, $\Rep_qB$ is non-unimodular, and its value on $S^2$ (with bounding $B^3$ bulk) vanishes. 

We are stuck on some apparent contradiction: the WRT boundary condition is certainly well-defined on 3-manifolds (actually, on most 3-manifolds, in this non-semisimple setting) but this isomorphic description isn't. The only possible way where we could have made a mistake here is in the orientation structures, which we chose somewhat arbitrarily. And indeed, orientations structures will have some impact on the value on $S^2$, which doesn't admit any framing.

Our conclusion is that when trying to describe WRT as a $\Rep_qB$ boundary condition to $\Rep_qT$, one cannot use pivotal and ribbon structures as the $SO(2)$ and $SO(3)$ structures, and one needs to wander into the realm of general orientation structures.

In this project, we only study general $SO(2)$-structures on $\Rep_qB$, discarding the bulk $\Rep_qT$. We are very interested in understanding general $SO(3)$-structures on braided categories in the future, which is the first step to study general $SO(2)$-structures on their boundary conditions.


\subsection*{Future directions}
\paragraph{Mapping Class Group representations.} For a finite tensor twisted-pivotal category $\BB$, our construction provides representations of Mapping Class Groups of surfaces, which are expected to be finite dimensional. The theorem above points at such input categories $\BB$ which would yield representations of a possibly very different nature than those coming from untwisted string net modules. We are currently working on an explicit description of these representations.

\paragraph{3-manifold invariants.} For finite tensor categories $\BB$ with $\orSN_\II^\alpha(S^2) \neq 0$, we expect that the oriented categorified 2-TQFT $\orSN_\II^\alpha$ can be extended to a non-compact 3-TQFT. In particular, by the theorem above, we would obtain 3-manifold invariants from non-unimodular categories, which was not possible before. We will come back to the construction of these 3-TQFTs in the near future.

\paragraph{General orientation structures on braided categories.} There is an analogue story one dimension up, with skein theory for braided categories. The usual orientation data in this setting is a ribbon structure. We would be very interested in understanding what is a general $SO(3)$-structure on a braided category, and to have a graphical calculus for them. Note that there is announced work of Douglas, Schommer-Pries and Snyder which would give a description of these general $SO(3)$-structures.

\paragraph{Connection with the oriented cobordism hypothesis.} It is natural to expect that the oriented theory we construct (or an appropriate completion thereof) is the truncation of a corresponding fully extended theory determined by the oriented cobordism hypothesis. To prove this would require a detailed understanding of the oriented cobordism hypothesis, including a prescription of how diffeomorphisms of surfaces act in terms of duality data and $SO(2)$-fixed point data. A precise understanding of the $SO(2)$-action on the groupoid of tensor categories would also be necessary. An intermediate step would be to compare the assignments on surfaces.

\subsection*{Outline}
We begin by recalling the \cite{JoyalStreetTensorCalculus} progressive graphical calculus for monoidal categories in Section \ref{sec:rigidGraphicalCalculus}, which we view as a graphical calculus for foliated surfaces. We also incorporate the admissibility conditions of \cite{CGP}.

We define the notion of \textit{twisted pivotal structures} in Section \ref{sec:twistedPivStr} and give some of its elementary properties. This notion arises from the study of orientation structures for the cobordism hypothesis and we recall the main motivating ideas.

We introduce the main object of the paper, namely \textit{twisted string net modules}, in Section \ref{sec:twistedGraphicalCalculus}. This has been the most cumbersome definition to work out, and we prove various results showing its good behavior, in particular we compute the values on different foliated disks. We also prove an excision result which can be reformulated as saying that twisted string net modules arrange in a foliated TQFT.

We prove our main result in Section \ref{sec:TQFT}: twisted string net modules, which a priori depend on a foliation, or a Morse function, on a surface actually only depend on the underlying oriented surface. More precisely, we construct an oriented categorified 2-TQFT.

We study the resulting categorified 2-TQFT in Section \ref{sec:Properties_S2}. In particular we relate the value on the 2-sphere to the distinguished invertible object of \cite{ENOdistinguishedinvertible} and to the twisted modified traces of \cite{GKPmtrace}.

We study examples for our construction coming from finite-dimensional Hopf algebras in Section \ref{sec:Examples}. We study one family of Hopf algebras for which we can classify all twisted pivotal structures, and which provide examples with various interesting behaviors. 


\subsection*{Acknowledgments}
We would like to thank Sebastian Halbig for explaining to us the examples coming from generalized Taft algebras, Lukas M\"uller for enlightening conversations on distinguished invertible objects, and Noah Snyder, Chris Schommer-Pries, Kevin Walker, Bruce Bartlett, Nathan Geer and François Costantino for encouragements and helpful discussions.

BH and WS are funded by the Simons Foundation award 1013836 as part of the Simons Collaboration on Global Categorical Symmetry. WS also acknowledges funding by the Deutsche Forschungsgemeinschaft  (DFG, German Research Foundation) through the Collaborative Research Center SFB 1085 Higher invariants - 224262486. FS acknowledges funding by the Max Planck Institute for Mathematics in Bonn.


\section{Graphical calculus for rigid categories in foliated surfaces}\label{sec:rigidGraphicalCalculus}
Unless stated otherwise, every manifold we consider will be compact, oriented and smooth, possibly with boundary.
\begin{definition}
A \textit{framing} on an oriented surface $\Sigma$ is a trivialization of its tangent bundle, i.e. the choice of two vector fields $\vec{x}$ and $\vec{y}$ on $\Sigma$ which form an oriented basis of the tangent space at every point.
It is not very easy to draw framings, and we will use another representation. 

A \textit{foliation} on a surface $\Sigma$ is a partition of $\Sigma = \underset{\lambda\in\Lambda}\sqcup \Gamma_\lambda$ into a disjoint union of embedded oriented 1-manifolds which is locally diffeomorphic to the standard foliation of $\mathbb R^2 = \underset{\lambda\in\mathbb R}\sqcup \mathbb R \times \{\lambda\}$. A \textit{leaf} of the foliation is a connected component of one of the 1-manifolds $\Gamma_\lambda$. A framing induces a foliation by partitioning $\Sigma$ into flow lines of $\vec{x}$. A foliation determines a framing up to a contractible choice. 

An embedded oriented graph $T \hookrightarrow \Sigma$ in a framed surface is \textit{progressive} if for every point $p$ of an edge of $T$, its tangent vector in the direction of $T$ has a positive $\vec{y}$-coordinate. 
An embedded oriented graph $T \hookrightarrow \Sigma$ in a foliated surface is \textit{positively transverse} if it intersect every leaf $\Gamma_\lambda$ transversely and positively. We only ask this condition on edges of $T$ and not on vertices. 
The graph $T$ is progressive with respect to some framing if and only if it is positively transverse to the induced foliation.
\end{definition}
\begin{warning}
    A progressive graph cannot always be transported through an isotopy of the framing. Indeed, the circle $\{0\}\times S^1$ is progressive in the torus $S^1\times S^1$ with standard framing $(\partial_x,\partial_y)$. However, this framing is isotopic to $(\partial_y,-\partial_x)$ by a global rotation, and no curve isotopic to  $\{0\}\times S^1$ is progressive in this framing. We do not consider framings up to isotopy. This problem is fixed by assuming rigidity, though we will not need this here. To avoid any confusion with the usual notion of framing which is often considered up to isotopy relative to the boundary, we will only talk about foliated surfaces thereafter.
\end{warning}

\begin{warning}
    If $\gamma$ is a single curve on $\Sigma$, we will say that $T$ is transverse to $\gamma$ if every edge of $T$ intersects $\gamma$ transversely, and the vertices of $T$ do not intersect $\gamma$. This is a generic condition. However, if $(\Gamma_\lambda)_\lambda$ is a foliation, when we say that $T$ is transverse to it we do not ask this condition on vertices (as it would never be satisfied). In other words, if $T$ is transverse to a foliation, and $\gamma$ is a leaf of the foliation, it is not automatic that $T$ is transverse to $\gamma$ as we still need to ensure that no vertices of $T$ lie in $\gamma$.
\end{warning}

Note that any graph $T$ can be made progressive up to a small isotopy, introducing new vertices in some edges and changing the orientations, see Figure \ref{fig:makeCapProgressive}.
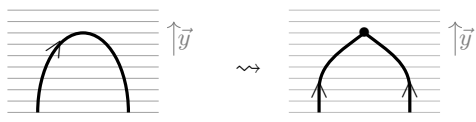
\begin{figure}
    \centering
\begin{tikzpicture}[xscale = 2, yscale = 1.5, baseline = 0.5cm]
    \def\nlines{10}
    \foreach \y in {0,...,\nlines} {
    \draw[gray!70] (0,\y/\nlines) -- ++(1,0);
    }
    \draw[gray, ->](1.1,0.5)--++(0,0.3) node[midway, right = -2pt]{\small $\vec{y}$};
    \draw[very thick] (0.2, 0) arc(180:0:0.3 and 0.7) node[pos = 0.3, sloped]{$>$};  
    \end{tikzpicture} \quad$\leadsto$\quad
    \begin{tikzpicture}[xscale = 2, yscale = 1.5, baseline = 0.5cm]
    \def\nlines{10}
    \foreach \y in {0,...,\nlines} {
    \draw[gray!50] (0,\y/\nlines) -- ++(1,0);
    }
    \draw[gray, ->](1.1,0.5)--++(0,0.3) node[midway, right = -2pt]{\small $\vec{y}$};
    \draw[very thick] (0.2, 0) -- (0.2,0.2)node[pos = 1, sloped]{$>$} to[out = 90, in = -135] (0.5,0.7) node{\small $\bullet$} to[out = -45, in = 90] (0.8,0.2) -- (0.8,0)node[pos = 0, sloped]{$<$};  
    \end{tikzpicture}    \caption{Left: A graph which is not progressive, the top of the cap is not transverse to the foliation (in gray) and half of the graph is negatively transverse to it. Right: A very similar progressive graph.}
    \label{fig:makeCapProgressive}
\end{figure}

\begin{definition}
Let $\BB$ be a monoidal category. A \textit{$\BB$-labeling} in an oriented 1-manifold $\Gamma$ is a finite set of points in $\Gamma$ each colored by an object of $\BB$. A \textit{$\BB$-string net} in a framed surface $\Sigma$ is a progressive embedded oriented graph $T \hookrightarrow \Sigma$ which intersect the boundary of $\Sigma$ transversely and only at 1-valent vertices, whose edges are colored by objects of $\BB$ and whose non-boundary vertices are colored by morphisms of $\BB$ with source and target prescribed by the adjacent edges. 
\begin{equation*}
    \begin{tikzpicture}[xscale = 2, yscale = 1.5, baseline = 0.5cm]
    \def\nlines{10}
    \foreach \y in {0,...,\nlines} {
    \draw[gray!50] (0,\y/\nlines) -- ++(1,0);
    }
    \draw[gray, ->](-0.2,0.5)--++(0,0.3) node[midway, right = -2pt]{\small $\vec{y}$};
    \draw[very thick] (0.2, 0) to[out = 90, in = -135] node[pos = 0.2, sloped]{$>$} node[pos = 0.2, left]{$x$} (0.5,0.5);
    \draw[very thick] (0.5, 0) to[out = 90, in = -90] node[pos = 0.2, sloped]{$>$} node[pos = 0.2, left]{$y$} (0.5,0.5);
    \draw[very thick] (0.8, 0) to[out = 90, in = -45] node[pos = 0.2, sloped]{$<$} node[pos = 0.2, left]{$z$} (0.5,0.5);
    \draw[very thick] (0.3, 1) to[out = -90, in = 125] node[pos = 0.2, sloped]{$<$} node[pos = 0.2, left]{$v$} (0.5,0.5)node{\small $\bullet$};
    \draw[very thick] (0.7, 1) to[out = -90, in = 55] node[pos = 0.2, sloped]{$>$} node[pos = 0.2, left]{$w$} (0.5,0.52)node[right=6pt]{$f \in \Hom_\BB(x\otimes y \otimes z, v \otimes w)$};
    \end{tikzpicture}
    \quad
    \text{also denoted}
    \quad
    \begin{tikzpicture}[xscale = 2, yscale = 1.5, baseline = 0.5cm]
    \def\nlines{10}
    \foreach \y in {0,...,\nlines} {
    \draw[gray!50] (0,\y/\nlines) -- ++(1,0);
    }
    \draw[gray, ->](-0.2,0.5)--++(0,0.3) node[midway, right = -2pt]{\small $\vec{y}$};
    \draw[very thick] (0.2, 0) to[out = 90, in = -135] node[pos = 0.2, sloped]{$>$} node[pos = 0.2, left]{$x$} (0.5,0.5);
    \draw[very thick] (0.5, 0) to[out = 90, in = -90] node[pos = 0.2, sloped]{$>$} node[pos = 0.2, left]{$y$} (0.5,0.5);
    \draw[very thick] (0.8, 0) to[out = 90, in = -45] node[pos = 0.2, sloped]{$<$} node[pos = 0.2, left]{$z$} (0.5,0.5);
    \draw[very thick] (0.3, 1) to[out = -90, in = 125] node[pos = 0.2, sloped]{$<$} node[pos = 0.2, left]{$v$} (0.5,0.5)node{\small $\bullet$};
    \draw[very thick] (0.7, 1) to[out = -90, in = 55] node[pos = 0.2, sloped]{$>$} node[pos = 0.2, left]{$w$} (0.5,0.5);
    \node[rectangle, draw, fill=white, thick, inner sep = 2pt, minimum width = 1cm]at (0.5,0.5) {$f$};
    \end{tikzpicture}
\end{equation*}
These are considered up to isotopy relative to the boundary in the space of $\BB$-string net, i.e. preserving the progressive property.
The boundary of a $\BB$-string net is naturally a $\BB$-labeling in the boundary of $\Sigma$ with color inherited from the only adjacent edge.

String nets can be transported along embeddings $\iota:\Sigma\to \Sigma'$ preserving the framings, and more generally along orientation-preserving embeddings preserving only the induced foliation, which will still preserve the progressive property. We call such embeddings \textit{progressive}. This laxer notion is handy because every point $p \in \Sigma$ has a neighborhood $\varphi :\D \hookrightarrow\Sigma$ which is progressive for the standard framing on the disk, though it does not necessarily admit one which preserves the framing without first changing the framing on $\Sigma$ through an isotopy.
\end{definition}
To a $\BB$-labeling on an interval one associates an object of $\BB$ by taking the tensor product of the colors in the order they appear. This object is well-defined up to unique isomorphism by the coherence of monoidal categories.

There is a well-defined evaluation map, denoted $Ev$ and referred to as the graphical calculus for monoidal categories, from $\BB$-string nets in $[0,1]^2$ with boundary in $[0,1]\times \{0,1\}$ to morphisms in $\BB$ from the object associated to the $\BB$-labeling in $[0,1]\times \{0\}$ to the object associated to the $\BB$-labeling in $[0,1]\times \{1\}$.
\begin{equation*}
Ev\left(
    \begin{tikzpicture}[xscale = 2, yscale = 1.5, baseline = 0.6cm]
    \def\nlines{10}
    \foreach \y in {0,...,\nlines} {
    \draw[gray!50] (0,\y/\nlines) -- ++(1,0);
    }
    \draw[very thick] (0.2, 0) to[out = 90, in = -135] node[pos = 0.2, sloped]{$>$} node[pos = 0.2, left]{$x$} (0.5,0.4);
    \draw[very thick] (0.5, 0) to[out = 90, in = -90] node[pos = 0.2, sloped]{$>$} node[pos = 0.2, left]{$y$} (0.5,0.4);
    \draw[very thick] (0.8, 0) to[out = 90, in = -45] node[pos = 0.2, sloped]{$<$} node[pos = 0.2, left]{$z$} (0.5,0.4);
    \draw[very thick] (0.3, 1) to[out = -90, in = 125] node[pos = 0.2, sloped]{$<$} node[pos = 0.2, left]{$v$} (0.5,0.4)node{\small $\bullet$};
    \draw[very thick] (0.7, 1) to[out = -90, in = 55] node[pos = 0.2, sloped]{$>$} node[pos = 0.2, left]{$w$}  node[pos = 0.8, sloped]{$>$} node[pos = 0.7, left]{$t$}  node[midway]{\small $\bullet$} node[midway, right]{$g$}(0.5,0.4)node[right]{$f$};
    \end{tikzpicture}\right)
     = (\Id_v \otimes g)\circ f \in \Hom_\BB(x\otimes y \otimes z, v \otimes w)
\end{equation*}
See \cite{JoyalStreetTensorCalculus, KST24} for details.

\begin{definition} Let $\Sigma$ be a foliated surface, $\BB$ a $\Bbbk$-linear monoidal category and $X$ a $\BB$-labeling on $\partial \Sigma$.
The \textit{relative string net module} $\SN_\BB(\Sigma; X)$ of $\Sigma$ with coefficient in $\BB$ and boundary condition $X$ is the $\Bbbk$-vector-space generated by $\BB$-string nets in $\Sigma$ with boundary $X$ modulo isotopy and the \textit{progressive skein relations}
\begin{equation}\label{eq:progressiveskeinrel}
\sum_i \lambda_i T_i \sim 0
\end{equation}
if the string nets $T_i$ coincide outside a progressive embedded square $\varphi:[0,1]^2\hookrightarrow \Sigma$, intersect it transversely only along $[0,1]\times \{0,1\}$, and such that its evaluation as a morphism in $\BB$ is zero, i.e. $\sum_i \lambda_i Ev(T_i\vert_{im\, \varphi}) = 0$.

A progressive embedding $\iota:\Sigma\to\Sigma'$ preserving the boundary labellings induces a linear map between relative string net modules by transporting string nets under $\iota$ as above. Two isotopic such embeddings (related by a path in the space of progressive embeddings preserving the boundary labellings) induce the same linear map.
\end{definition}

\subsection{Admissibility}
When $\BB$ is non-semisimple, string net modules are not quite the correct notion. In particular, even in the pivotal and unimodular cases, no extension to 3-manifolds is known (and none is expected to exist). This problem is fixed by considering admissible string net modules \cite{CGPAdmissbleskein, CGPVnc2plus1}. We now adapt their definition to our foliated context.

\begin{definition}
    Let $\II \subseteq\BB$ be a (two-sided) tensor ideal, i.e. a full subcategory closed under retracts and tensoring by objects of $\BB$ on either side. In particular, it is stable under isomorphisms and, if $\BB$ is rigid, under taking duals. The reader is encouraged to think of the ideal of projective objects in a tensor category.

    A $\BB$-labelling on a 1-manifold $\Gamma$ is \textit{admissible} if there is at least one point per connected component of $\Gamma$ whose label lies in $\II$. A $\BB$-string net in a foliated surface $\Sigma$ is \textit{admissible} if there is at least one $\II$-colored edge per connected component of $\Sigma$ (which does not imply that its boundary labelling is admissible). A progressive skein relation is \textit{admissible} if there is at least one $\II$-colored edge per connected component outside the square where the relation happens. In the case $\II= \BB$, then, up to adding some strands colored by the unit, every boundary label, string net, and skein relation is admissible.

    Let $\Sigma$ be a foliated surface and $X$ a $\BB$-labelling in its boundary (not necessarily admissible). The \textit{admissible relative string net module} $\SN_\II(\Sigma;X)$ is generated by admissible $\BB$-string nets in $\Sigma$ with boundary $X$ modulo isotopy and admissible progressive skein relations. If there is at least one $\II$-colored boundary label per connected component of $\Sigma$, admissibility of string nets and skein relations is automatic, and this agrees with the $\BB$-string net module.
\end{definition}

\begin{remark}
    It is not uncommon to ask that every edge of the $\II$-string net module $\SN_\II(\Sigma;X)$ is colored by an object of $\II$, instead of just one per connected component. If every boundary color lies in $\II$ and $\BB$ is rigid, this can be ensured using that $\II$ is a tensor ideal. However, it will be important for us to use strands that are not colored by objects of $\II$.
\end{remark}

\subsection{Functoriality}\label{sec:functoriality}
Let us now encode functoriality in the boundary condition $X$ in a convenient way. 
\begin{definition}\label{def:SN_on_1-mflds}
    Let $\Gamma$ be a closed oriented 1-manifold. The \textit{string net category} $\SN_\II(\Gamma)$ has objects admissible $\BB$-labelings $X$ on $\Gamma$ and morphisms the string net module $\SN_\II(\Gamma\times I; X\times \{0\} \sqcup Y\times \{1\})$ where $\Gamma\times I$ has cylindrical foliation with level sets $\Gamma\times \{t\}$. Composition is given by gluing string nets along their common boundary and identifying $(\Gamma\times I) \underset{\Gamma}{\cup}(\Gamma\times I)\simeq \Gamma\times I$. The identity is $X\times I$.
    \end{definition}
The string net category of the interval $\SN_\II(I)$ is equivalent to $\II$ by the evaluation functor. 

Every object of the string net category of the circle $\SN_\II(S^1)$ is isomorphic to a labeling with a single point $1\in S^1$ colored by some object $x\in \II$. Morphisms between two such objects are generated by morphisms of the form
\begin{equation}\label{eq:genMorphSNS1}
\begin{tikzpicture}[scale = 1.2, baseline = 0.5cm]
    \def\nlines{5}
    \foreach \y in {1,...,\nlines} {
    \draw[gray!50, dashed] (0,\y/\nlines) arc(180:0:1 and 0.24);
    \draw[gray!50] (0,\y/\nlines) arc(180:360:1 and 0.24);
    }
    \draw[dashed] (0,0) arc(180:0:1 and 0.24);
    \draw (0,0) arc(180:360:1 and 0.24);
    \draw (0,1) arc(180:540:1 and 0.24);
    \draw (0,0) -- ++(0,1);
    \draw (2,0) -- ++(0,1);
    \draw[gray, ->](-0.2,0.5)--++(0,0.3) node[midway, left = 0pt]{\small $\vec{y}$};
    \draw[very thick] (1,-0.24) -- ++(0,1) node[pos = 0]{\small $\bullet$} node[pos = 0.25, sloped]{\small $>$} node[pos = 0.25, left]{$x$}node[pos = 0.75, sloped]{\small $>$} node[pos = 0.75, left]{$y$} node[pos = 1]{\small $\bullet$} node[pos = 0.5]{\small $\bullet$} node[pos = 0.5, right]{$f$};  
    \end{tikzpicture} 
\text{ for } f \in \Hom_\BB(x,y) \text{ and } 
\begin{tikzpicture}[scale = 1.2, baseline = 0.5cm]
    \def\nlines{5}
    \foreach \y in {1,...,\nlines} {
    \draw[gray!50, dashed] (0,\y/\nlines) arc(180:0:1 and 0.24);
    \draw[gray!50] (0,\y/\nlines) arc(180:360:1 and 0.24);
    }
    \draw[dashed] (0,0) arc(180:0:1 and 0.24);
    \draw (0,0) arc(180:360:1 and 0.24);
    \draw (0,1) arc(180:540:1 and 0.24);
    \draw (0,0) -- ++(0,1);
    \draw (2,0) -- ++(0,1);
    \draw[gray, ->](-0.2,0.5)--++(0,0.3) node[midway, left = 0pt]{\small $\vec{y}$};
    \draw[very thick] (1.1,-0.24) -- ++(-0.2,1) node[pos = 0]{\small $\bullet$} node[pos = 0.5, sloped]{\small $<$} node[pos = 0.5, right]{$x$} node[pos = 1]{\small $\bullet$};  
    \draw[very thick] (0.9,-0.24) to[out = 110, in = -65] node[pos = 0]{\small $\bullet$} node[pos = 0.5, sloped]{\small $<$} node[pos = 0.5, above]{$y$} (0,0.25);
    \draw[very thick, dashed]   (0,0.25)  to[out = 40, in = 175] (2,0.65);
    \draw[very thick] (2,0.65) to[out = -160, in = -70] (1.1, 0.76) node{\small $\bullet$};
    \end{tikzpicture} 
\text{ from } y\otimes x \text{ to } x \otimes y
\end{equation}
and the inverses of the second ones.

    \begin{definition}\label{def:SNfunctor}
    Let $\Sigma$ be a foliated surface with foliated boundary $\Gamma\subseteq\partial \Sigma$, i.e. each connected component of $\Gamma$ lies in a single leaf of the foliation, and they all admit a foliated collar as in Definition \ref{def:Bord^Morse_fol}. We decompose $\Gamma= \Gamma_{in}\sqcup \Gamma_{out}$ depending on whether the progressive direction points inwards or outwards at the boundary. The string net modules of $\Sigma$ assemble into a functor
    \begin{equation}\label{eq:SNcobordismBimodule}
    \begin{array}{cccccc}
    \SN_\II(\Sigma;-): & \SN_\II(\Gamma_{in})^{op} &\otimes &\SN_\II(\Gamma_{out}) &\to &\Vect\\
    & X&\otimes& Y  &\mapsto & \SN_\II(\Sigma;X\sqcup Y) \ .
    \end{array}
    \end{equation}
    On morphism, it is given by gluing string nets along their common boundary and identifying $(\Gamma\times I) \underset{\Gamma}{\cup}\Sigma\simeq \Sigma$.
\end{definition}

\begin{remark}
    One can define a similar string net category associated to a 1-manifold equipped with any 2-framing \cite[Def. 6.1]{KST24}, and the definition above corresponds to the trivial, or cylindrical, 2-framing. One would need a slight adjustment and add the data of an orientation to $\BB$-labelings. This notion is useful as in general cutting a foliated surface along a curve, the 2-framing restricted to the curve is not cylindrical. However, in this paper we will be interested in foliations coming from Morse functions, and decompositions of surfaces along regular level sets of this Morse function, which all have cylindrical framing. Hence this more general definition will not be necessary and we opted to not overload the reader with definitions.
\end{remark}

\subsection{Rigidity}
\begin{definition}
    Let $x \in \BB$ be an object of a monoidal category, a \emph{right dual} for $x$ is an object $x^* \in \BB$ together with a right evaluation map $\rev_x : x^* \otimes x \to \unit$ and a right coevaluation map $\rcoev_x: \unit \to  x\otimes x^*$ satisfying the snake relations. A \emph{left dual} for $x$ is an object ${}^*x \in \BB$ together with a left evaluation map $\lev_x : x \otimes {}^*x \to \unit$ and $\lcoev_x: \unit \to  {}^*x\otimes x$ satisfying the snake relations. 
\end{definition}
\begin{remark}
    If $\BB$ is rigid, meaning that all of its object have right and left duals, then there is a natural choice for coloring the graph in Figure \ref{fig:makeCapProgressive}. However, not every graph obtained through this procedure has a natural coloring, see Figure \ref{fig:ColoringCapCircle}.
    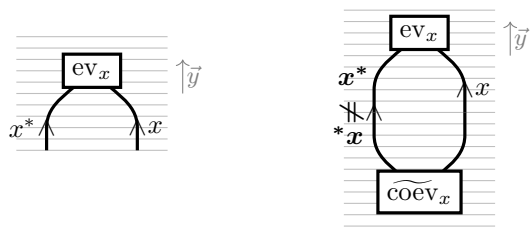
\begin{figure}
    \centering
    \begin{tikzpicture}[xscale = 2, yscale = 1.5, baseline = 0.5cm]
    \def\nlines{10}
    \foreach \y in {0,...,\nlines} {
    \draw[gray!50] (0,\y/\nlines) -- ++(1,0);
    }
    \draw[gray, ->](1.1,0.5)--++(0,0.3) node[midway, right = -2pt]{\small $\vec{y}$};
    \draw[very thick] (0.2, 0) -- (0.2,0.2)node[pos = 1, sloped]{$>$}node[pos = 1, left]{$x^*$} to[out = 90, in = -125] (0.5,0.7) node[rectangle, draw, fill = white]{$\rev_x$} to[out = -55, in = 90] (0.8,0.2) -- (0.8,0)node[pos = 0, sloped]{$<$} node[pos = 0, right]{$x$};  
    \end{tikzpicture} \quad \quad \quad \quad \begin{tikzpicture}[xscale = 2, yscale = 1.5, baseline = 0cm]
    \def\nlines{10}
    \foreach \y in {-\nlines,...,\nlines} {
    \draw[gray!50] (0,\y/\nlines) -- ++(1,0);
    }
    \draw[gray, ->](1.1,0.5)--++(0,0.3) node[midway, right = -2pt]{\small $\vec{y}$};
    \draw[very thick] (0.2, 0) -- (0.2,0.2)node[pos = 0, sloped]{$>$}node[pos = 1, above left = -2pt]{$\bm{ x^*}$} to[out = 90, in = -125] (0.5,0.7) node[rectangle, draw, fill = white]{$\rev_x$} to[out = -55, in = 90] (0.8,0.2) -- (0.8,0)node[pos = 0, sloped]{$<$} node[pos = 0, right]{$x$};  
    \draw[very thick] (0.2, 0) -- (0.2,-0.2)node[pos = 1, left]{$\bm{ {}^*x}$}node[pos = 0, above, rotate = 90]{$\bm{\neq}$} to[out = -90, in = 125] (0.5,-0.7) node[rectangle, draw, fill = white]{$\lcoev_x$} to[out = 55, in = -90] (0.8,-0.2) -- (0.8,0);  
    \end{tikzpicture}
    \caption{Left: The favorite coloring for the progressive cap. Right: There is no favorite coloring for the circle.}
    \label{fig:ColoringCapCircle}
\end{figure}
    The usual fix for this problem, see e.g. \cite{Selinger}, is to equip $\BB$ with a pivotal structure, which identifies the left and right duals in the circle of Figure \ref{fig:ColoringCapCircle}. This is \textit{not} the solution that we will use.
\end{remark}

\begin{notation}
    To make the graphical calculus lighter, we will omit labels of 2-valent vertices when they are labeled by an evaluation or a coevaluation, i.e.
    \begin{equation*}
    \begin{tikzpicture}[xscale = 2, yscale = 1.5, baseline = 0.5cm]
    \def\nlines{10}
    \foreach \y in {0,...,\nlines} {
    \draw[gray!50] (0,\y/\nlines) -- ++(1,0);
    }
    \draw[gray, ->](1.1,0.5)--++(0,0.3) node[midway, right = -2pt]{\small $\vec{y}$};
    \draw[very thick] (0.2, 0) -- (0.2,0.2)node[pos = 1, sloped]{$>$}node[pos=1,left]{$x^*$} to[out = 90, in = -135] (0.5,0.7) node{\small $\bullet$} to[out = -45, in = 90] (0.8,0.2) -- (0.8,0)node[pos = 0, sloped]{$<$}node[pos = 0, right]{$x$};  
\end{tikzpicture}:=
\begin{tikzpicture}[xscale = 2, yscale = 1.5, baseline = 0.5cm]
    \def\nlines{10}
    \foreach \y in {0,...,\nlines} {
    \draw[gray!50] (0,\y/\nlines) -- ++(1,0);
    }
    \draw[gray, ->](1.1,0.5)--++(0,0.3) node[midway, right = -2pt]{\small $\vec{y}$};
    \draw[very thick] (0.2, 0) -- (0.2,0.2)node[pos = 1, sloped]{$>$}node[pos = 1, left]{$x^*$} to[out = 90, in = -125] (0.5,0.7) node[rectangle, draw, fill = white]{$\rev_x$} to[out = -55, in = 90] (0.8,0.2) -- (0.8,0)node[pos = 0, sloped]{$<$} node[pos = 0, right]{$x$};  
    \end{tikzpicture}\hspace{7mm} \text{and}\hspace{7mm}
\begin{tikzpicture}[xscale = 2, yscale = 1.5, baseline = 0.5cm]
    \def\nlines{10}
    \foreach \y in {0,...,\nlines} {
    \draw[gray!50] (0,\y/\nlines) -- ++(1,0);
    }
    \draw[gray, ->](1.1,0.5)--++(0,0.3) node[midway, right = -2pt]{\small $\vec{y}$};
    \draw[very thick] (0.2, 1) -- (0.2,0.8)node[pos = 1, sloped]{$<$}node[pos=1,left]{$x$} to[out = -90, in = 135] (0.5,0.3) node{\small $\bullet$} to[out = 45, in = -90] (0.8,0.8) -- (0.8,1)node[pos = 0, sloped]{$>$}node[pos = 0, right]{$x^*$};  
    \end{tikzpicture}:=
    \begin{tikzpicture}[xscale = 2, yscale = 1.5, baseline = 0.5cm]
    \def\nlines{10}
    \foreach \y in {0,...,\nlines} {
    \draw[gray!50] (0,\y/\nlines) -- ++(1,0);
    }
    \draw[gray, ->](1.1,0.5)--++(0,0.3) node[midway, right = -2pt]{\small $\vec{y}$};
    \draw[very thick] (0.2, 1) -- (0.2,0.8)node[pos = 1, sloped]{$<$}node[pos = 1, left]{$x$} to[out = -90, in = 135] (0.5,0.3) node[rectangle, draw, fill = white]{$\rcoev_x$} to[out = 45, in = -90] (0.8,0.8) -- (0.8,1)node[pos = 0, sloped]{$>$} node[pos = 0, right]{$x^*$};  
    \end{tikzpicture}\ .
\end{equation*}
\end{notation}

\section{Twisted pivotal structures}\label{sec:twistedPivStr}
\begin{definition}\label{def:alphaTwistedQuasiPivStr}
    Let $\alpha \in \BB^\times$ be an invertible object in a rigid monoidal category $\BB$. An \textit{$\alpha$-twisted quasi-pivotal structure} $\apiv$ on $\BB$ is the data of a natural isomorphism
    \begin{equation}
        \apiv:(-)^{\ast\ast}  \otimes \alpha \simeq \alpha \otimes (-)
    \end{equation}
    which is compatible with the monoidal structure, in the sense that $\apiv_\unit$ is induced by the canonical isomorphism $\unit^{**}\simeq \unit$ and unitors and the twisted pivotal structure on a tensor product $\apiv_{x\otimes y}$ is given by the composite
    $$(x\otimes y)^{**}\otimes \alpha \overset{can}{\simeq} x^{**}\otimes y^{**}\otimes \alpha \overset{\Id_{x^{**}}\otimes\apiv_y}{\simeq} x^{**}\otimes \alpha \otimes y \overset{\apiv_x\otimes \Id_y}{\simeq} \alpha \otimes x\otimes y\ .$$
\end{definition}
\begin{definition}\label{def:alphaTwistedPivStr}
    An {$\alpha$-twisted quasi-pivotal structure} $\apiv$ is called an \textit{$\alpha$-twisted pivotal structure} if it satisfies the \emph{twist relation}
    \begin{equation}\label{eq:twistrelation}
        \begin{tikzpicture}[baseline = 0.5cm, scale = 1]
            \draw (0,0) -- (0,2.5) node[pos = 0.5, sloped]{$>$} node[pos = 0.5, left]{$\alpha$};
            \draw (1,0)--(1,2) node[pos = 0.25, sloped]{\small $>$} node[pos = 0.25, right]{\small $\alpha^*$} node[pos = 0.75, sloped]{\small $>$} node[pos = 0.75, right]{\small $\alpha$};
            \draw (2,-0.5) -- (2,0) -- (2,2) node[pos = 0.1, sloped]{ $>$} node[pos = 0.1, right]{$\alpha$} node[pos = 0.75, sloped]{\small $>$} node[pos = 0.75, right]{\small ${}^*\alpha$};
            \node[rectangle, draw, fill=white, minimum width=1.5cm] at (0.5,0) {$\rcoev_\alpha$};
            \node[rectangle, draw, fill=white, minimum width=1.5cm, minimum height = 0.6cm] at (1.5,1) {$\apiv_{\,{}^*\alpha}$};
            \node[rectangle, draw, fill=white, minimum width=1.5cm] at (1.5,2) {$\lev_\alpha$};
            \draw (4,-0.5) -- (4,2.5) node[pos = 0.5, sloped]{$>$} node[pos = 0.5, right]{$\alpha$} node[pos = 0.5, left = 0.7cm]{$=$};
        \end{tikzpicture}
    \end{equation}
    A rigid monoidal category $\BB$ equipped with an $\alpha$-twisted pivotal structure for some $\alpha\in\BB^\times$ is called a \textit{twisted-pivotal category}.
\end{definition}

\begin{remark}
The notion of twisted quasi-pivotal structure appears in several places in the literature. It is called a \textit{quasi-pivotal structure} in \cite{HalbigZormanPivotalityTwisted} or a \textit{$(-)^{**}$-twisted half braiding} on $\alpha$ in \cite{DSPS}. The name \textit{$h$-pivotal structure} was also suggested by Douglas--Schommer-Pries--Snyder. Objects equipped with a $(-)^{**}$-twisted half braiding form a category called the \textit{trace category} in \cite[Definition 3.1.4]{DSPS}. It is shown in \cite{KST24} that, after an appropriate completion, this is the value that the framed $\BB$-string net theory assigns to the circle with product framing, recalled in Definition \ref{def:SNfunctor}. It is also shown in \cite{ShimizuPivotalOnDrinfeld} that twisted quasi-pivotal structures on a finite tensor category $\BB$ correspond to pivotal structures on its Drinfeld center.

In the context of Hopf algebras, twisted quasi-pivotal structures correspond to pairs in involutions \cite{HalbigGenTaft}. The twist relation also appeared in this context under the name \textit{modular}.
\end{remark}


\begin{notation}
Let $\alpha \in \BB^{\times}$ be an invertible object. To fix conventions for the inverse, we set $\alpha^{-1} = \alpha^*$. The evaluation and coevaluation maps $\rev_{\alpha}: \alpha^{\ast} \otimes \alpha \to \unit$ and $\rcoev_{\alpha}:\unit \to \alpha \otimes \alpha^{\ast}$ are invertible and thus witness the invertibility of $\alpha$. Moreover, their inverses $\rcoev_{\alpha}^{-1}: \alpha \otimes \alpha^{\ast} \to \unit$ and $\rev_{\alpha}^{-1}:\unit \to \alpha^{\ast} \otimes \alpha$ express $\alpha^{\ast}$ as a left dual to $\alpha$. Without loss of generality we may assume that ${}^{*}\alpha = \alpha^{-1}=\alpha^{*}$, with left evaluation and coevaluation maps given by $\rcoev_{\alpha}^{-1}$ and $\rev^{-1}_{\alpha}$ respectively. In our graphical calculus, we get a canonical interpretation for all types of caps and cups:
    \begin{equation*}
    \begin{tikzpicture}[xscale = 2, yscale = 1.5, baseline = 0.5cm]
    \def\nlines{10}
    \foreach \y in {0,...,\nlines} {
    \draw[gray!50] (0,\y/\nlines) -- ++(1,0);
    }
    \draw[gray, ->](1.1,0.5)--++(0,0.3) node[midway, right = -2pt]{\small $\vec{y}$};
    \draw[very thick] (0.2, 0) -- (0.2,0.2)node[pos = 1, sloped]{$>$}node[pos=1,left]{$\alpha$} to[out = 90, in = -135] (0.5,0.7) node{\small $\bullet$} to[out = -45, in = 90] (0.8,0.2) -- (0.8,0)node[pos = 0, sloped]{$<$}node[pos = 0, right]{$\alpha^{-1}$}; 
\end{tikzpicture}:=
\begin{tikzpicture}[xscale = 2, yscale = 1.5, baseline = 0.5cm]
    \def\nlines{10}
    \foreach \y in {0,...,\nlines} {
    \draw[gray!50] (0,\y/\nlines) -- ++(1,0);
    }
    \draw[gray, ->](1.1,0.5)--++(0,0.3) node[midway, right = -2pt]{\small $\vec{y}$};
    \draw[very thick] (0.2, 0) -- (0.2,0.2)node[pos = 1, sloped]{$>$}node[pos = 1, left]{$\alpha$} to[out = 90, in = -125] (0.5,0.7) node[rectangle, draw, fill = white]{$\rcoev_\alpha^{-1}$} to[out = -55, in = 90] (0.8,0.2) -- (0.8,0)node[pos = 0, sloped]{$<$} node[pos = 0, right]{$\alpha^{-1}$};  
    \end{tikzpicture}\hspace{10mm}
\begin{tikzpicture}[xscale = 2, yscale = 1.5, baseline = 0.5cm]
    \def\nlines{10}
    \foreach \y in {0,...,\nlines} {
    \draw[gray!50] (0,\y/\nlines) -- ++(1,0);
    }
    \draw[gray, ->](1.1,0.5)--++(0,0.3) node[midway, right = -2pt]{\small $\vec{y}$};
    \draw[very thick] (0.2, 1) -- (0.2,0.8)node[pos = 1, sloped]{$<$}node[pos=1,left]{$\alpha^{-1}$} to[out = -90, in = 135] (0.5,0.3) node{\small $\bullet$} to[out = 45, in = -90] (0.8,0.8) -- (0.8,1)node[pos = 0, sloped]{$>$}node[pos = 0, right]{$\alpha$};  
    \end{tikzpicture}:=
    \begin{tikzpicture}[xscale = 2, yscale = 1.5, baseline = 0.5cm]
    \def\nlines{10}
    \foreach \y in {0,...,\nlines} {
    \draw[gray!50] (0,\y/\nlines) -- ++(1,0);
    }
    \draw[gray, ->](1.1,0.5)--++(0,0.3) node[midway, right = -2pt]{\small $\vec{y}$};
    \draw[very thick] (0.2, 1) -- (0.2,0.8)node[pos = 1, sloped]{$<$}node[pos = 1, left]{$\alpha^{-1}$} to[out = -90, in = 135] (0.5,0.3) node[rectangle, draw, fill = white]{$\rev_\alpha^{-1}$} to[out = 45, in = -90] (0.8,0.8) -- (0.8,1)node[pos = 0, sloped]{$>$} node[pos = 0, right]{$\alpha$};  
    \end{tikzpicture} \ .
\end{equation*}
By definition, the following relations (and their left-right mirrors) hold:
\begin{equation}\label{eq:alpha-inv-rels}
    \begin{tikzpicture}[xscale = 1.25, yscale = 0.75, baseline = 0cm]
    \def\nlines{5}
    \foreach \y in {-\nlines,...,\nlines} {
    \draw[gray!50] (0,\y/\nlines) -- ++(1,0);
    }
    \draw[gray, ->](1.1,0.5)--++(0,0.3) node[midway, right = -2pt]{\small $\vec{y}$};
    \draw[very thick] (0.2, 0) -- (0.2,0.2)node[pos = 0, sloped]{$>$} node[pos = 0, left]{$\alpha^{-1}$} to[out = 90, in = -125] (0.5,0.7) node{$\bullet$} to[out = -55, in = 90] (0.8,0.2) -- (0.8,0)node[pos = 1, sloped]{$<$} node[pos = 1, right]{$\alpha$};  
    \draw[very thick] (0.2, 0) -- (0.2,-0.2) to[out = -90, in = 125] (0.5,-0.7) node{$\bullet$} to[out = 55, in = -90] (0.8,-0.2) -- (0.8,0);  
    \end{tikzpicture}
    =
    \begin{tikzpicture}[xscale = 1.25, yscale = 0.75, baseline = 0cm]
    \def\nlines{5}
    \foreach \y in {-\nlines,...,\nlines} {
    \draw[gray!50] (0,\y/\nlines) -- ++(1,0);
    }
    \draw[gray, ->](1.1,0.5)--++(0,0.3) node[midway, right = -2pt]{\small $\vec{y}$};
    \end{tikzpicture}
    \hspace{10mm}
    \begin{tikzpicture}[xscale = 1.25, yscale = 0.75, baseline = 0cm]
    \def\nlines{5}
    \foreach \y in {-\nlines,...,\nlines} {
    \draw[gray!50] (0,\y/\nlines) -- ++(1,0);
    }
    \draw[gray, ->](1.1,0)--++(0,0.3) node[midway, right = -2pt]{\small $\vec{y}$};
    \draw[very thick] (0.2, -1) -- (0.2,-0.8)node[pos = 1, sloped]{$>$} node[pos = 1, left]{$\alpha^{-1}$} to[out = 90, in = -125] (0.5,-0.3) node{$\bullet$} to[out = -55, in = 90] (0.8,-0.8) -- (0.8,0-1)node[pos = 0, sloped]{$<$} node[pos = 0, right]{$\alpha$};  
    \draw[very thick] (0.2, 1) -- (0.2,0.8) node[pos = 1, sloped]{$<$} node[pos = 1, left]{$\alpha^{-1}$} to[out = -90, in = 125] (0.5,0.3) node{$\bullet$} to[out = 55, in = -90] (0.8,0.8) -- (0.8,1) node[pos = 0, sloped]{$>$} node[pos = 0, right]{$\alpha$};
    \end{tikzpicture}
    =
    \begin{tikzpicture}[xscale = 1.25, yscale = 0.75, baseline = 0cm]
    \def\nlines{5}
    \foreach \y in {-\nlines,...,\nlines} {
    \draw[gray!50] (0,\y/\nlines) -- ++(1,0);
    }
    \draw[gray, ->](1.1,0.5)--++(0,0.3) node[midway, right = -2pt]{\small $\vec{y}$};
    \draw[very thick] (0.2, -1) -- (0.2,1) node[pos=0.5,sloped]{$>$}node[pos=0.5,left]{$\alpha^{-1}$};
    \draw[very thick] (0.8, -1) -- (0.8,1) node[pos=0.5,sloped]{$>$}node[pos=0.5,right]{$\alpha$};
    \end{tikzpicture}
\end{equation}
We refer to the relations of \eqref{eq:alpha-inv-rels} as the \emph{eye relation} and \emph{merge relation} for the invertibility of $\alpha$ respectively.
\end{notation}

\begin{notation}
We will denote graphically the twisted pivotal structure $\apiv$ as follows, and use the same notation for its mate:
\begin{gather*}
\begin{tikzpicture}[xscale = 2, yscale = 1.5, baseline = 0.5cm]
    \def\nlines{10}
    \foreach \y in {0,...,\nlines} {
    \draw[gray!50] (0,\y/\nlines) -- ++(1,0);
    }
    \draw[gray, ->](1.1,0.5)--++(0,0.3) node[midway, right = -2pt]{\small $\vec{y}$};
    \draw[very thick] (0.2,0) to[out = 90, in=-90]node[pos = 0.15, sloped]{$>$}node[pos=0.15,left]{$x^{**}$}node[pos = 0.85, sloped]{$>$}node[pos=0.85,left]{$x$} (0.8,1);
    \draw[very thick] (0.8,0) to[out = 90, in=-90]node[pos = 0.15, sloped]{$<$}node[pos=0.15,left]{$\alpha$} node[pos = 0.85, sloped]{$<$}node[pos=0.85,left]{$\alpha$} (0.2,1);
    \node[rectangle, draw, thick, fill = white, rotate = 45, inner sep = 4pt] at (0.5, 0.5){};
\end{tikzpicture}
:=
    \begin{tikzpicture}[xscale = 2, yscale = 1.5, baseline = 0.5cm]
    \def\nlines{10}
    \foreach \y in {0,...,\nlines} {
    \draw[gray!50] (0,\y/\nlines) -- ++(1,0);
    }
    \draw[gray, ->](1.1,0.5)--++(0,0.3) node[midway, right = -2pt]{\small $\vec{y}$};
    \draw[very thick] (0.2,0) to[out = 90, in=-90]node[pos = 0.15, sloped]{$>$}node[pos=0.15,left]{$x^{**}$}node[pos = 0.85, sloped]{$>$}node[pos=0.85,left]{$x$} (0.8,1);
    \draw[very thick] (0.8,0) to[out = 90, in=-90]node[pos = 0.15, sloped]{$<$}node[pos=0.15,left]{$\alpha$} node[pos = 0.85, sloped]{$<$}node[pos=0.85,left]{$\alpha$} (0.2,1);
    \node[rectangle, draw, very thick, fill = white, minimum width = 10mm] at (0.5, 0.5){$\apiv_x$};
\end{tikzpicture}
\hspace{2cm}
\begin{tikzpicture}[xscale = 2, yscale = 1.5, baseline = 0.5cm]
    \def\nlines{10}
    \foreach \y in {0,...,\nlines} {
    \draw[gray!50] (0,\y/\nlines) -- ++(1,0);
    }
    \draw[gray, ->](1.1,0.5)--++(0,0.3) node[midway, right = -2pt]{\small $\vec{y}$};
    \draw[very thick] (0.2,0) to[out = 90, in=-90]node[pos = 0.15, sloped]{$>$}node[pos=0.15,left]{$\alpha^{-1}$}  node[pos = 0.85, sloped]{$>$}node[pos=0.85,right]{$\alpha^{-1}$} (0.8,1);
    \draw[very thick] (0.8,0) to[out = 90, in=-90]node[pos = 0.15, sloped]{$<$}node[pos=0.15,left]{$x^{**}$} node[pos = 0.85, sloped]{$<$}node[pos=0.85,left]{$x$} (0.2,1);
    \node[rectangle, draw, thick, fill = white, rotate = 45, inner sep = 4pt] at (0.5, 0.5){};
\end{tikzpicture}
:=
\begin{tikzpicture}[xscale = 2, yscale = 1.5, baseline = 0.5cm]
    \def\nlines{10}
    \foreach \y in {0,...,\nlines} {
    \draw[gray!50] (0,\y/\nlines) -- ++(1,0);
    }
    \draw[gray, ->](1.1,0.5)--++(0,0.3) node[midway, right = -2pt]{\small $\vec{y}$};
    \draw[very thick] (0.2,0) to[out = 90, in=-90]node[pos = 0.15, sloped]{$<$}node[pos=0.15,left]{$\alpha^{-1}$} (0.1,0.8)node{\small $\bullet$} to[out = -35, in = 145] (0.9,0.2) node{\small $\bullet$} to[out = 90, in = -90] node[pos = 0.85, sloped]{$<$}node[pos=0.85,right]{$\alpha^{-1}$} (0.8,1);
    \draw[very thick] (0.8,0) to[out = 110, in=-135, looseness = 2]node[pos = 0.15, sloped]{$<$}node[pos=0.05,left]{$x^{**}$} (0.5,0.5) to [out = 45, in = -70, looseness = 2] node[pos = 0.75, sloped]{$<$}node[pos=0.75,above right]{$x$} (0.2,1);
    \node[rectangle, draw, thick, fill = white, rotate = 45, inner sep = 4pt] at (0.5, 0.5){};
\end{tikzpicture}
\end{gather*}
This a-priori leads to some ambiguity when $x = \alpha$ or $x = \alpha^{-1}$, but the different interpretations agree by \eqref{eq:otherTwist}.\\
The naturality of $\apiv$ with respect to $f:x \to y$ now reads as:
\begin{equation}\label{eq:pivNatGraphically}
\begin{tikzpicture}[xscale = 1.2, yscale = 0.8, baseline = 1cm]
    \def\nlines{10}
    \foreach \y in {0,...,\nlines} {
    \draw[gray!50] (-0.7,3*\y/\nlines) -- ++(2,0);
    }
    \draw[very thick] (0,0) to[out = 135, in=-135, in distance = 10mm]node[pos = 0.15, sloped]{\small $<$}node[pos=0.15,left]{$x^{**}$}
    (0,2.5) node{\small $\bullet$} to[out = -45, in = 45] node[pos = 0.35, sloped]{\small $<$}node[pos=0.35,right=-1pt]{$x^{*}$}
    (0,0.8) node{\small $\bullet$} to[out = 110, in=-70] node[midway, rectangle, draw, fill=white]{$f$} (0,2.2) node{\small $\bullet$}
    to[out = -135, in = 135]
    (0,0.5)node{\small $\bullet$} to[out = 45, in = -90]node[pos = 0.15, sloped]{\small $>$}node[pos=0.15,right]{$y^{**}$} node[pos = 0.85, sloped]{\small $>$}node[pos=0.85,right]{$y$} (1,3);
    \draw[very thick] (1,0) to[out = 90, in=-60, out distance = 20mm, in distance = 9mm]node[pos = 0.15, sloped]{\small $<$}node[pos=0.15,right]{$\alpha$} node[pos = 0.8, sloped]{\small $<$}node[pos=0.8,right]{$\alpha$} node[pos = 0.43, rectangle, draw, thick, fill = white, rotate = 45, inner sep = 4pt]{} (0.2,3);
\end{tikzpicture} 
\ =\ 
\begin{tikzpicture}[xscale = 1.2, yscale = 0.8, baseline = 1cm]
    \def\nlines{10}
    \foreach \y in {0,...,\nlines} {
    \draw[gray!50] (-0.7,3*\y/\nlines) -- ++(2,0);
    }
    \draw[very thick] (0,0) to[out = 90, in=-90, in distance = 20mm]node[pos = 0.15, sloped]{\small $>$}node[pos=0.15,left]{$x^{**}$}node[pos = 0.8, rectangle, draw, fill=white]{$f$} node[pos = 0.63, sloped]{\small $>$}node[pos=0.6,right]{$x$} node[pos = 0.95, sloped]{\small $>$}node[pos=0.95,right]{$y$} (1,3);
    \draw[very thick] (1,0) to[out = 120, in=-90]node[pos = 0.15, sloped]{\small $<$}node[pos=0.15,right]{$\alpha$} node[pos = 0.8, sloped]{\small $<$}node[pos=0.8,right]{$\alpha$} node[pos = 0.4, rectangle, draw, thick, fill = white, sloped, inner sep = 4pt]{} (0.2,3);
\end{tikzpicture}
\end{equation}
The compatibility of $\apiv$ with the monoidal structure reads as:
\begin{equation}\label{eq:pivMonoidalGraphically}
\begin{tikzpicture}[xscale = 2, yscale = 1.5, baseline = 0.75cm]
    \foreach \y in {0,...,5} {
    \draw[gray!50] (0,\y*0.3) -- ++(1.3,0);
    }
    \draw[very thick] (0.8,0) to[out = 90, in=-90]node[pos = 0.15, sloped]{$<$}node[pos=0.15,right]{$\alpha$} node[pos = 0.85, sloped]{$<$}node[pos=0.85,left]{$\alpha$} (0.2,1) -- ++(0,0.5);
    \draw[very thick] (0.2,0) to[out = 90, in=-90]node[pos = 0.15, sloped]{$>$}node[pos=0.15,left]{$x^{**}$}node[pos = 0.85, sloped]{$>$}node[pos=0.85,left]{$x$} node[pos = 0.55, rectangle, draw, thick, fill = white, rotate = 45, inner sep = 3pt]{} (0.6,1);
    \draw[very thick] (0.4,0) to[out = 90, in=-90]node[pos = 0.15, sloped]{$>$}node[pos=0.15,right]{$y^{**}$}node[pos = 0.85, sloped]{$>$}node[pos=0.85,right]{$y$} node[pos = 0.45, rectangle, draw, thick, fill = white, rotate = 45, inner sep = 3pt]{} (0.8,1);
    \draw[very thick] (0.7, 1.1) node[rectangle, draw, thick, fill = white]{\footnotesize $\ \Id\ $} -- (0.7,1.5)node[pos = 0.7, sloped]{$>$}node[pos=0.7,right]{$x\otimes y$};
\end{tikzpicture}
\ = \
\begin{tikzpicture}[xscale = 2, yscale = 1.5, baseline = 0.75cm]
    \foreach \y in {0,...,5} {
    \draw[gray!50] (-0.4,\y*0.3) -- ++(1.6,0);
    }
    \draw[very thick] (0.8,0) -- (0.8, 0.5) to[out = 90, in=-90]node[pos = 0.15, sloped]{$<$}node[pos=0.15,right]{$\alpha$} node[pos = 0.85, sloped]{$<$}node[pos=0.85,left]{$\alpha$} (0.2,1.5);
    \draw[very thick] (0.2,0) to[out = 90, in=-90]node[pos = 0.3, sloped]{$>$}node[pos=0.15,left]{$x^{**}$}(0.2,0.5);
    \draw[very thick] (0.4,0) to[out = 90, in=-90]node[pos = 0.3, sloped]{$>$}node[pos=0.15,right]{$y^{**}$}(0.4,0.5);
    \draw[very thick] (0.3,0.5) to[out = 90, in=-90]node[pos = 0.25, sloped]{$>$}node[pos=0.25,left]{$(x\otimes y)^{**}$}node[pos = 0.85, sloped]{$>$}node[pos=0.85,right]{$x\otimes y$} node[pos = 0.5, rectangle, draw, thick, fill = white, rotate = 45, inner sep = 4pt]{} node[pos = 0, rectangle, draw, thick, fill = white]{\footnotesize $\ can\ $} (0.7,1.5);
\end{tikzpicture}
\quad \text{and} \quad 
\begin{tikzpicture}[xscale = 2, yscale = 1.5, baseline = 0.5cm]
    \foreach \y in {0,...,5} {
    \draw[gray!50] (0,\y*0.2) -- ++(1,0);
    }
    \draw[very thick] (0.2,0) to[out = 90, in=-90]node[pos = 0.15, sloped]{$>$}node[pos=0.15,left]{$\unit$}node[pos = 0.85, sloped]{$>$}node[pos=0.85,left]{$\unit$} (0.8,1);
    \draw[very thick] (0.8,0) to[out = 90, in=-90]node[pos = 0.15, sloped]{$<$}node[pos=0.15,left]{$\alpha$} node[pos = 0.85, sloped]{$<$}node[pos=0.85,left]{$\alpha$} (0.2,1);
    \node[rectangle, draw, thick, fill = white, rotate = 45, inner sep = 4pt] at (0.5, 0.5){};
\end{tikzpicture}
\ = \
\begin{tikzpicture}[xscale = 2, yscale = 1.5, baseline = 0.5cm]
    \foreach \y in {0,...,5} {
    \draw[gray!50] (0,\y*0.2) -- ++(1,0);
    }
    \draw[very thick] (0.8,0) to[out = 90, in=-90]node[pos = 0.5, sloped]{$<$}node[pos=0.5,left]{$\alpha$} (0.2,1);
\end{tikzpicture}
\end{equation}
Together, they imply that one can pass an evaluation or a coevaluation through $\alpha$:
\begin{equation}\label{eq:pivNatEval}
\begin{tikzpicture}[xscale = 2, yscale = 1.5, baseline = 0.75cm]
    \foreach \y in {0,...,5} {
    \draw[gray!50] (0,\y*0.2) -- ++(1,0);
    }
    \draw[very thick] (0.8,0) to[out = 90, in=-90]node[pos = 0.15, sloped]{$<$}node[pos=0.15,right]{$\alpha$} node[pos = 0.85, sloped]{$<$}node[pos=0.85,left]{$\alpha$} (0.2,1);
    \draw[very thick] (0.2,0) to[out = 90, in=-155]node[pos = 0.15, sloped]{$>$}node[pos=0.15,left]{$x^{***}$}node[pos = 0.85, sloped]{$>$}node[pos=0.85,above]{$x^*$} node[pos = 0.6, rectangle, draw, thick, fill = white, rotate = 45, inner sep = 3pt]{} (0.7,0.9) node{\small $\bullet$};
    \draw[very thick] (0.4,0) to[out = 90, in=-75]node[pos = 0.15, sloped]{$>$}node[pos=0.15,right]{$x^{**}$}node[pos = 0.85, sloped]{$>$}node[pos=0.85,right]{$x$} node[pos = 0.47, rectangle, draw, thick, fill = white, rotate = 45, inner sep = 3pt]{} (0.7,0.9);
\end{tikzpicture}
\ = \
\begin{tikzpicture}[xscale = 2, yscale = 1.5, baseline = 0.75cm]
    \foreach \y in {0,...,5} {
    \draw[gray!50] (0,\y*0.2) -- ++(1,0);
    }
    \draw[very thick] (0.8,0) to[out = 90, in=-90]node[pos = 0.5, sloped]{$<$}node[pos=0.5,right]{$\alpha$} (0.2,1);
    \draw[very thick] (0.2,0) to[out = 90, in=-135]node[pos = 0.5, sloped]{$>$}(0.35,0.4) node{\small $\bullet$};
    \draw[very thick] (0.4,0) to[out = 90, in=-55]node[pos = 0.5, sloped]{$<$}(0.35,0.4);
\end{tikzpicture}\ .
\end{equation}
\end{notation}
\begin{lemma}
Let $\apiv$ be an $\alpha$-twisted quasi-pivotal structure. The following are equivalent:
\begin{equation}\label{eq:otherTwist}
\left[
\begin{tikzpicture}[xscale = 1, yscale = 0.6, baseline = 0.3cm]
    \foreach \y in {-3,...,10} {
    \draw[gray!50] (-0.4,\y/5) -- ++(1.4,0);
    }
    \draw[very thick] (-0.2,2) to[out = -90, in = 125]node[pos = 0.35, sloped]{$<$}node[pos=0.35,left]{$\alpha$} (0.2,0)node{\small $\bullet$} to[out = 55, in=-90] (0.7,1) to[out = 90, in = -45] (0.5,1.5) node{\small $\bullet$};
    \draw[very thick] (0.8,-0.6) to[out = 90, in=-90, out distance = 10mm] (0.3,1) to[out = 90, in = -135] (0.5,1.5);
    \node[rectangle, draw,  thick, fill = white, rotate = 45, inner sep = 4pt] at (0.5, 0.4){};
\end{tikzpicture} = \begin{tikzpicture}[xscale = 1, yscale = 0.6, baseline = 0.3cm]
    \foreach \y in {-3,...,10} {
    \draw[gray!50] (-0.4,\y/5) -- ++(0.8,0);
    }
    \draw[very thick] (0,2)-- (0,-0.6)node[pos = 0.35, sloped]{$<$}node[pos=0.35,left]{$\alpha$};
\end{tikzpicture} 
\right]
\ \Leftrightarrow \
\left[
\begin{tikzpicture}[xscale = 1.5, yscale = 1.5, baseline = 0.6cm]
    \def\nlines{10}
    \foreach \y in {0,...,\nlines} {
    \draw[gray!50] (0,\y/\nlines) -- ++(1,0);
    }
    \draw[very thick] (0.2,0) to[out = 90, in=-90]node[pos = 0.15, sloped]{$>$}node[pos=0.25,above left = -5pt]{\small $\alpha^{-1}$}node[pos = 0.85, sloped]{$>$}node[pos=0.85,above left = -5pt]{\small $\alpha^{-1}$} (0.8,1);
    \draw[very thick] (0.8,0) to[out = 90, in=-90]node[pos = 0.15, sloped]{$<$}node[pos=0.15,left]{\small $\alpha$} node[pos = 0.85, sloped]{$<$}node[pos=0.85,left]{\small $\alpha$} (0.2,1);
    \node[rectangle, draw, thick, fill = white, rotate = 45, inner sep = 4pt] at (0.5, 0.5){};
\end{tikzpicture} = 
\begin{tikzpicture}[xscale = 1.5, yscale = 1.5, baseline = 0.6cm]
    \def\nlines{10}
    \foreach \y in {0,...,\nlines} {
    \draw[gray!50] (0,\y/\nlines) -- ++(1,0);
    }
    \draw[very thick] (0.2,0) to[out = 90, in=-135] (0.5,0.4) node{\small $\bullet$} to[out = -45, in = 90]node[pos = 0.6, sloped]{$<$}node[pos=0.6,right]{\small $\alpha$} (0.8,0);
    \draw[very thick] (0.2,1) to[out = -90, in=135]node[pos = 0.4, sloped]{$<$}node[pos=0.4,left]{\small $\alpha$} (0.5,0.6) 
    node{\small $\bullet$} to[out = 45, in = -90] (0.8,1);
\end{tikzpicture}
\right]
\ \Leftrightarrow \
\left[
\begin{tikzpicture}[xscale = 1.5, yscale = 1.5, baseline = 0.6cm]
    \def\nlines{10}
    \foreach \y in {0,...,\nlines} {
    \draw[gray!50] (0,\y/\nlines) -- ++(1,0);
    }
    \draw[very thick] (0.2,0) to[out = 90, in=-90]node[pos = 0.15, sloped]{$>$}node[pos=0.25,left]{\small $\alpha$}node[pos = 0.85, sloped]{$>$}node[pos=0.85,left]{\small $\alpha$} (0.8,1);
    \draw[very thick] (0.8,0) to[out = 90, in=-90]node[pos = 0.15, sloped]{$<$}node[pos=0.15,left]{\small $\alpha$} node[pos = 0.85, sloped]{$<$}node[pos=0.85,left]{\small $\alpha$} (0.2,1);
    \node[rectangle, draw, thick, fill = white, rotate = 45, inner sep = 4pt] at (0.5, 0.5){};
\end{tikzpicture} = 
\begin{tikzpicture}[xscale = 1.5, yscale = 1.5, baseline = 0.6cm]
    \def\nlines{10}
    \foreach \y in {0,...,\nlines} {
    \draw[gray!50] (0,\y/\nlines) -- ++(1,0);
    }
    \draw[very thick] (0.2,0) to[out = 90, in=-90] node[pos = 0.5, sloped]{$>$}node[pos=0.5,right]{\small $\alpha$} (0.2,1);
    \draw[very thick] (0.8,0) to[out = 90, in=-90] node[pos = 0.5, sloped]{$>$}node[pos=0.5,right]{\small $\alpha$} (0.8,1);
\end{tikzpicture}
\right]
\end{equation}
the first being the graphical description of the twist relation \eqref{eq:twistrelation}.
\end{lemma}
\begin{proof}
    One can pass from the first equality to the second by using some duality and composing with an isomorphism. The third follows by using \eqref{eq:pivNatEval}.
\end{proof}

\subsection{Twisted pivotal structures are SO(2)-structures}
Following the Cobordism Hypothesis, we expect that a rigid monoidal category $\BB$, which is 2-dualizable in the Morita 3-category $\operatorname{Alg}_1(\Pr)$ of monoidal categories by \cite{DSPS}, induces a once-categorified framed 2-TQFT. In particular, it associates vector spaces to framed surfaces. This indeed matches what we have observed above (with the caveat that isotopies of framings do not yield foliation-preserving diffeomorphims). 

To obtain a theory defined on oriented surfaces, one should equip $\BB$ with a structure of an $SO(2)$-homotopy-fixed-point. The exact data that such an orientation structure represents is conjecturally understood, and is expected to correspond precisely to an $\alpha$-twisted pivotal structure, see \cite{SchomPriesDualLowdimHighCat, DSPS}:
\begin{conjecture}
    An $SO(2)$-homotopy-fixed-point on a rigid monoidal category $\BB \in \operatorname{Alg}_1(\Pr)$ is the data of an isomorphism of $\BB$-$\BB$-bimodule categories
    $${}_{(-)}\BB_{(-)^{**}} \simeq {}_{(-)}\BB_{(-)}$$
    such that the 3-cell described in \cite[Figure 6]{SchomPriesDualLowdimHighCat} is trivial.
\end{conjecture}
In particular, a pivotal structure, which is a monoidal trivialization of the double dual functor 
$$(-)^{**}\simeq (-)$$
does induce a $SO(2)$-structure, but there are many more. In general, a trivialization of the double dual bimodule is determined by an invertible object $\alpha \in \BB^\times$
$$\begin{aligned}
    {}_{(-)}\BB_{(-)^{**}} & \to & {}_{(-)}\BB_{(-)}\\
    X &\mapsto& X\otimes \alpha
\end{aligned}$$
together with a $(-)^{**}$-twisted half braiding 
$$\apiv:(-)^{\ast\ast}  \otimes \alpha \simeq \alpha \otimes (-)$$
which exhibits $(-) \otimes \alpha$ as a bimodule functor. The triviality of the 3-cell is expected to correspond to the twist relation \eqref{eq:twistrelation}.

\begin{remark}
    In order to keep the main idea clear, we did not make explicit how to see $\BB$ as a presentable category, nor the role that the tensor ideal $\II\subseteq \BB$ plays here. Starting from $\II\subseteq \BB$, the appropriate object to consider is the cp-rigid presentable category $\EE = \operatorname{Fun}(\II^{op},\Vect) \in \operatorname{Alg}_1(\Pr)$. If $\BB$ is a tensor category and $\II$ the ideal of projectives, one can also write $\EE = \operatorname{Ind(\BB)}$.
    
    Reciprocally, starting from any cp-rigid presentable category $\EE\in \operatorname{Alg}_1(\Pr)$, we can take $\II = \EE^{cp}$ its subcategory of compact-projective objects and $\BB = \EE^{dzb}$ its subcategory of dualizable objects. A bimodule isomorphism ${}_{(-)}\EE_{(-)^{**}} \to  {}_{(-)}\EE_{(-)}$ has to be of the form $(-)\otimes \alpha$ for some tensor-invertible $\alpha \in \EE$. In particular $\alpha\in \BB$ is dualizable, and we recover the description above.
\end{remark}

\section{Graphical calculus for twisted pivotal categories}\label{sec:twistedGraphicalCalculus}
\begin{definition}
Let $\Sigma$ be a compact oriented surface and let $S \subseteq \Sigma$ be a finite set.
A \emph{foliation with isolated singularities $S$} on $\Sigma$ is a foliation
$\mathcal{F}$ on the punctured surface $\Sigma \setminus S$. The points of $S$ are
called the \emph{singularities} of $\mathcal{F}$.
\end{definition}

An important special case comes from Morse functions. Let
$f \colon \Sigma \to \mathbb{R}$ be a Morse function, and set
$S = \operatorname{crit}(f)$. On $\Sigma \setminus S$, the connected components of the
regular level sets $f^{-1}(\lambda)$ (for $\lambda \in \mathbb{R}$) form the leaves of a foliation, with positive transverse direction given by the gradient of $f$. The local triviality of the foliation follows from the submersion theorem, which gives a local product neighborhood at each regular point. 
By the Morse lemma, each critical point has one of three standard local models, determined by its index; the corresponding foliations are shown in
Figure \ref{fig:MorseSingFoliations}.
\begin{figure}
    \centering $\D_0 =\ $ 
\begin{tikzpicture}[baseline = 0pt]
    \foreach \y in {1,...,8}{
    \draw[gray!70] (0,0) circle(\y/8);
    }
    \draw[gray, ->](0,1.1)--++(0,0.3) node[midway, right = -2pt]{\small $\vec{y}$};
    \draw[gray, ->](0,-1.1)--++(0,-0.3) node[midway, right = -2pt]{\small $\vec{y}$};
    \node[star,  star point ratio=2.25, fill=black, inner sep = 1pt] at (0,0) {};
\end{tikzpicture}   
\hspace{1.6cm}
$\D_1 =\ $ 
\begin{tikzpicture}[baseline = 0pt]
        \begin{scope}
        \clip (0,0) circle(1.2);
    \draw[gray!70] (1,1)--(-1,-1);
    \draw[gray!70] (1,-1)--(-1,1);
    \def\sadlines{5}
    \foreach \y in {3,...,5}{
    \draw[gray!70] ({-\y/(\sadlines+1)},1) .. controls (0,{1.6-1.6*\y/(\sadlines+1)}) .. ({\y/(\sadlines+1)},1);
    \draw[gray!70] ({-\y/(\sadlines+1)},-1) .. controls (0,{-1.6+1.6*\y/(\sadlines+1)}) .. ({\y/(\sadlines+1)},-1);
    \draw[gray!70] (-1,{-\y/(\sadlines+1)}) .. controls ({-1.6+1.6*\y/(\sadlines+1)},0) .. (-1,{\y/(\sadlines+1)});
    \draw[gray!70] (1,{-\y/(\sadlines+1)}) .. controls ({1.6-1.6*\y/(\sadlines+1)},0) .. (1,{\y/(\sadlines+1)});
    }   \end{scope}
    \draw[gray, ->](-0.9,1)--++(-0.3,-0.3) node[midway, above left = -2pt]{\small $\vec{y}$};
    \draw[gray, ->](0.9,1)--++(0.3,-0.3) node[midway, above right = -2pt]{\small $\vec{y}$};
    \draw[gray, ->](-0.9,-1)--++(-0.3,0.3) node[midway, below left = -2pt]{\small $\vec{y}$};
    \draw[gray, ->](0.9,-1)--++(0.3,0.3) node[midway, below right = -2pt]{\small $\vec{y}$};
    \node[star,  star point ratio=2.25, fill=black, inner sep = 1pt] at (0,0) {};
\end{tikzpicture}
\hspace{1.6cm}
$\D_2 =\ $ 
\begin{tikzpicture}[baseline = 0pt]
    \foreach \y in {1,...,8}{
    \draw[gray!70] (0,0) circle(\y/8);
    }
    \draw[gray, <-](0,1.1)--++(0,0.3) node[midway, right = -2pt]{\small $\vec{y}$};
    \draw[gray, <-](0,-1.1)--++(0,-0.3) node[midway, right = -2pt]{\small $\vec{y}$};
    \node[star,  star point ratio=2.25, fill=black, inner sep = 1pt] at (0,0) {};
\end{tikzpicture}     
\caption{The three types of behavior near the singularity for foliations induced by a Morse function, respectively called the cup (index 0), the saddle (index 1) and the cap (index 2).}
    \label{fig:MorseSingFoliations}
\end{figure}
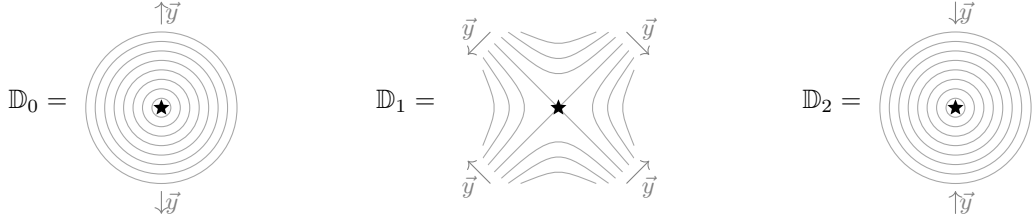
\begin{definition}
    A \emph{Morse singular foliation} on $\Sigma$ is a foliation with isolated singularities that locally arises from a Morse function as above. From now on, we will work exclusively with Morse singular foliations.
\end{definition}
\subsection{Twisted String Net Modules}
\begin{definition}\label{def:twistedStringNetModules}
    Let $\Sigma$ be a surface equipped with a Morse singular foliation. An embedded oriented graph $T\hookrightarrow\Sigma$ is \emph{progressive} if the interior of the edges avoid the singularities $S$ of $\Sigma$ and are progressive in the usual sense in $\Sigma\smallsetminus S$, and each singularity $s\in S$ has a vertex of $T$ mapping to it. By isotopy of progressive graph we mean isotopy preserving these properties. An edge of $T$ can never slide above a singular point.

    For $\BB$ a monoidal category, a $\BB$ string net is a progressive embedded oriented graphs $T\hookrightarrow\Sigma$ which intersects the boundary and the singular points of $\Sigma$ only at 1-valent vertices, and whose edges and other vertices are labeled by objects and morphisms of $\BB$ as usual. 
\end{definition}
\begin{definition}\label{def:twsitSNfunctor}
    Let $\Sigma$ be a surface equipped with a Morse singular foliation, $\BB$ a rigid $\Bbbk$-linear monoidal category equipped with an $\alpha$-twisted pivotal structure, and $X$ a $\BB$-labeling on $\partial \Sigma$. 
    The \emph{relative $\alpha$-twisted string net module} 
    $\SN^\alpha_\BB(\Sigma;X)$
    is the vector space generated by $\BB$-string nets in $\Sigma$ which locally near singularities have a single strand colored by $\alpha$ or $\alpha^{-1}$ depending on the index of the singularity as in the figures below. More precisely, for singular points of index 0 or 2, we ask that the strand is colored by $\alpha$ if it is outgoing and $\alpha^{-1}$ if it is incoming, and the opposite for index 1.
    
    We consider these string nets with singularities modulo isotopy, progressive skein relations \eqref{eq:progressiveskeinrel}, sliding over a singular point relations depicted in Equations (\ref{eq:skein-slide-0}, \ref{eq:skein-slide-1}, \ref{eq:skein-slide-2}) and changing the incoming direction at a saddle singularity depicted in Equation \eqref{eq:saddle-rel}. 
\\
Sliding over a critical point of index 0:
    \begin{equation}    \label{eq:skein-slide-0}
\begin{tikzpicture}[scale = 0.71, baseline = 0cm]
    \tikzmath{\a = 15;} 
    \foreach \y in {0,...,8}{
    \draw[gray!70] (0,0) circle(\y/4);
    }
    \draw[gray, ->] (45:2.2) -- (45:2.8) node[midway,below right = -2pt]{\small $\vec{y}$};
    \draw[sloped,black,thick] (0,0) to[out=90,in=-90] node[pos=0.5]{$>$} (0,2) node[below left]{$\alpha$};
    \node[star,  star point ratio=2.25, fill=black, inner sep = 1pt] at (0,0) {};
    \draw[sloped,black, thick] (-120+\a:{2}) node[above left]{$x^*$} to[out=60+\a, in = -75+\a] node[pos=0.5]{$<$} (-15+\a:{0.4}) node{\small $\bullet$} to[out=15+\a, in = -150+\a] node[pos = 0.5]{$<$} (60+\a:{2}) node[below right]{$x$};
\end{tikzpicture}
    \hspace{2mm} 
    \sim
    \hspace{2mm}
\begin{tikzpicture}[scale = 0.71, baseline = 0cm]
    \tikzmath{\a = 15;} 
    \foreach \y in {0,...,8}{
    \draw[gray!70] (0,0) circle(\y/4);
    }
    \draw[gray, ->] (45:2.2) -- (45:2.8) node[midway,below right = -2pt]{\small $\vec{y}$};
    \draw[sloped,black,thick] (0,0) -- node[pos=0.4]{$>$}node[pos=0.4,rotate=-90,right]{$\alpha$} (0,1.15)  -- node[pos=0.6]{$>$} (0,2) node[below left]{$\alpha$};
    \draw[sloped,black, thick] (-120+\a:{2}) node[above left]{$x^*$} to[out=60+\a,in=195+\a] node[pos=0.4]{\small $>$}  (150+\a:{0.4}) node{\small $\bullet$} to[out=105+\a,in=-120+\a] node[pos=0.2]{\small $>$}node[pos=0.9]{\small $>$} (60+\a:{2}) node[below right]{$x$};
    \node at (138:{1}) { $x^{**}$};
    \node[rectangle,rotate = 25,draw, thick, fill=white,inner sep = 4pt] at (0,1.15) {};
    \node[star,  star point ratio=2.25, fill=black, inner sep = 1pt] at (0,0) {};
\end{tikzpicture}
    \end{equation}
    Sliding over a critical point of index 1:
\begin{equation}    \label{eq:skein-slide-1}
\begin{tikzpicture}[scale = 0.71, baseline = 0cm]
    \tikzmath{\s = 1;} 
\begin{scope}[scale = 2]
        \clip (0,0) circle(1.2);
    \draw[gray!70] (1,1)--(-1,-1);
    \draw[gray!70] (1,-1)--(-1,1);
    \def\sadlines{5}
    \foreach \y in {3,...,5}{
    \draw[gray!70] ({-\y/(\sadlines+1)},1) .. controls (0,{1.6-1.6*\y/(\sadlines+1)}) .. ({\y/(\sadlines+1)},1);
    \draw[gray!70] ({-\y/(\sadlines+1)},-1) .. controls (0,{-1.6+1.6*\y/(\sadlines+1)}) .. ({\y/(\sadlines+1)},-1);
    \draw[gray!70] (-1,{-\y/(\sadlines+1)}) .. controls ({-1.6+1.6*\y/(\sadlines+1)},0) .. (-1,{\y/(\sadlines+1)});
    \draw[gray!70] (1,{-\y/(\sadlines+1)}) .. controls ({1.6-1.6*\y/(\sadlines+1)},0) .. (1,{\y/(\sadlines+1)});
    }   \end{scope}
    \node[star,  star point ratio=2.25, fill=black, inner sep = 1pt] at (0,0) {};
    
    \draw[sloped,black,thick] (0,0) to[out=-90,in=90] node[pos=0.5]{\footnotesize $<$} (0,-1-\s) node[above right]{$\alpha$};
    \draw[sloped,black,thick] (-1-\s,0) node[left]{$x^{*}$} to[out=0,in=-135] node[pos=0.5]{\footnotesize $<$} (0,1) node{\small $\bullet$} to[out=-45,in=180] node[pos=0.5]{\footnotesize $>$} (1+\s,0) node[right]{$x$};
\end{tikzpicture}
    \hspace{2mm} 
    \sim
    \hspace{2mm} 
\begin{tikzpicture}[scale = 0.71, baseline = 0cm]
    \tikzmath{\s = 1;} 
\begin{scope}[scale = 2]
        \clip (0,0) circle(1.2);
    \draw[gray!70] (1,1)--(-1,-1);
    \draw[gray!70] (1,-1)--(-1,1);
    \def\sadlines{5}
    \foreach \y in {3,...,5}{
    \draw[gray!70] ({-\y/(\sadlines+1)},1) .. controls (0,{1.6-1.6*\y/(\sadlines+1)}) .. ({\y/(\sadlines+1)},1);
    \draw[gray!70] ({-\y/(\sadlines+1)},-1) .. controls (0,{-1.6+1.6*\y/(\sadlines+1)}) .. ({\y/(\sadlines+1)},-1);
    \draw[gray!70] (-1,{-\y/(\sadlines+1)}) .. controls ({-1.6+1.6*\y/(\sadlines+1)},0) .. (-1,{\y/(\sadlines+1)});
    \draw[gray!70] (1,{-\y/(\sadlines+1)}) .. controls ({1.6-1.6*\y/(\sadlines+1)},0) .. (1,{\y/(\sadlines+1)});
    }   \end{scope}
    \node[star,  star point ratio=2.25, fill=black, inner sep = 1pt] at (0,0) {};
    
    \draw[sloped,black,thick] (0,0) to[out=-90,in=90] node[pos=0.15]{\footnotesize $<$} node[pos=0.8]{\footnotesize $<$} (0,-1-\s) node[above right]{$\alpha$};
    \draw[sloped,black,thick] (-1-\s,0) node[left]{$x^*$} to[out=0,in=140] node[pos=0.3]{\footnotesize $<$} (-0.9,-1.6) node{\small $\bullet$} to[out=50,in=180] node[pos=0.15]{\footnotesize $>$}node[pos=0.7]{\footnotesize $>$} (1+\s,0) node[right]{$x$};
    \node at (-0.68,-1) {$x^{**}$};
    \node at (-0.3,-0.3) {$\alpha$};
    \node[rectangle,rotate=25,draw, thick, fill=white,inner sep = 4pt] at (0,-0.75) {};
\end{tikzpicture}
\end{equation}
Sliding over a critical point of index 2:
\begin{equation}    \label{eq:skein-slide-2}
\begin{tikzpicture}[scale = 0.71, baseline = 0cm]
    \tikzmath{\a = 45;} 
    \foreach \y in {0,...,8}{
    \draw[gray!70] (0,0) circle(\y/4);
    }
    \draw[gray, <-] (45:2.2) -- (45:2.8) node[midway,below right = -2pt]{\small $\vec{y}$};
    \draw[black,thick] (0,0) to[out=-90,in=90] node[pos=0.5, sloped]{$<$}  node[pos = 0.5, left]{$\alpha^{-1}$}(0,-2);
    \node[star,  star point ratio=2.25, fill=black, inner sep = 1pt] at (0,0) {};
    \draw[black, thick] (-120+\a:{2})  to[out=60+\a, in = -75+\a] node[pos=0.5, sloped]{$>$}node[pos = 0.5, right]{$x^{**}$} (-15+\a:{0.4}) node{\small $\bullet$} to[out=15+\a, in = -150+\a] node[pos = 0.5, sloped]{$>$}node[pos = 0.5, above right]{$x^{*}$} (60+\a:{2});
\end{tikzpicture}
    \hspace{2mm}
    \sim
    \hspace{2mm}
\begin{tikzpicture}[scale = 0.71, baseline = 0cm]
    \tikzmath{\a = 45;} 
    \foreach \y in {0,...,8}{
    \draw[gray!70] (0,0) circle(\y/4);
    }
    \draw[gray, <-] (45:2.2) -- (45:2.8) node[midway,below right = -2pt]{\small $\vec{y}$};
    \draw[black, thick] (-120+\a:{2}) to[out=60+\a,in=195+\a] node[pos=0.15, sloped]{\small $<$} node[pos = 0.15, right]{$x^{**}$}node[pos=0.8, sloped]{\small $<$} node[pos = 0.8, left]{$x$}  (120+\a:{0.4}) node{\small $\bullet$} to[out=105+\a,in=-120+\a] node[pos=0.5, sloped]{\small $<$} node[pos = 0.6, right]{$x^{*}$} (60+\a:{2});
    \draw[thick] (-90:2) to [out = 90, in = -90]node[pos = 0.1, sloped, rotate = 180]{\small $<$} node[pos = 0.2,left]{$\alpha^{-1}$} node[pos = 0.47,rectangle,rotate =-25,draw, thick, fill=white,inner sep = 4pt]{} (0,0);
    \node[star,  star point ratio=2.25, fill=black, inner sep = 1pt] at (0,0) {};
    \end{tikzpicture}
\end{equation}
Changing the incoming direction at a saddle singularity:
\begin{equation}    \label{eq:saddle-rel}
\begin{tikzpicture}[scale = 0.61, baseline = 0cm]
    \tikzmath{\s = 0.8;} 
\begin{scope}[scale = 2]
        \clip (0,0) circle(1.2);
    \draw[gray!70] (1,1)--(-1,-1);
    \draw[gray!70] (1,-1)--(-1,1);
    \def\sadlines{5}
    \foreach \y in {3,...,5}{
    \draw[gray!70] ({-\y/(\sadlines+1)},1) .. controls (0,{1.6-1.6*\y/(\sadlines+1)}) .. ({\y/(\sadlines+1)},1);
    \draw[gray!70] ({-\y/(\sadlines+1)},-1) .. controls (0,{-1.6+1.6*\y/(\sadlines+1)}) .. ({\y/(\sadlines+1)},-1);
    \draw[gray!70] (-1,{-\y/(\sadlines+1)}) .. controls ({-1.6+1.6*\y/(\sadlines+1)},0) .. (-1,{\y/(\sadlines+1)});
    \draw[gray!70] (1,{-\y/(\sadlines+1)}) .. controls ({1.6-1.6*\y/(\sadlines+1)},0) .. (1,{\y/(\sadlines+1)});
    }   \end{scope}
    \node[star,  star point ratio=2.25, fill=black, inner sep = 1pt] at (0,0) {};
    
    \draw[sloped,black,thick] (0,0) to[out=90,in=-90] node[pos=0.5]{\footnotesize $<$} (0,1.2) node{\small $\bullet$} to[out=-60,in=180] node[pos=0.5]{\footnotesize $>$} (1+\s,0) node[right]{$\alpha^{-1}$};
    \node at (-0.3,0.75) {$\alpha$};
\end{tikzpicture}
    \hspace{2mm}
    \sim
    \hspace{2mm}
\begin{tikzpicture}[scale = 0.61, baseline = 0cm]
    \tikzmath{\s = 0.8;} 
\begin{scope}[scale = 2]
        \clip (0,0) circle(1.2);
    \draw[gray!70] (1,1)--(-1,-1);
    \draw[gray!70] (1,-1)--(-1,1);
    \def\sadlines{5}
    \foreach \y in {3,...,5}{
    \draw[gray!70] ({-\y/(\sadlines+1)},1) .. controls (0,{1.6-1.6*\y/(\sadlines+1)}) .. ({\y/(\sadlines+1)},1);
    \draw[gray!70] ({-\y/(\sadlines+1)},-1) .. controls (0,{-1.6+1.6*\y/(\sadlines+1)}) .. ({\y/(\sadlines+1)},-1);
    \draw[gray!70] (-1,{-\y/(\sadlines+1)}) .. controls ({-1.6+1.6*\y/(\sadlines+1)},0) .. (-1,{\y/(\sadlines+1)});
    \draw[gray!70] (1,{-\y/(\sadlines+1)}) .. controls ({1.6-1.6*\y/(\sadlines+1)},0) .. (1,{\y/(\sadlines+1)});
    }   \end{scope}
    \node[star,  star point ratio=2.25, fill=black, inner sep = 1pt] at (0,0) {};
    
    \draw[sloped,black,thick] (0,0) to[out=0,in=180] node[pos=0.5]{\footnotesize $>$} (1+\s,0) node[right]{$\alpha^{-1}$};
\end{tikzpicture}
    \hspace{2mm}
    \sim
    \hspace{2mm}
\begin{tikzpicture}[scale = 0.61, baseline = 0cm]
    \tikzmath{\s = 0.8;} 
\begin{scope}[scale = 2]
        \clip (0,0) circle(1.2);
    \draw[gray!70] (1,1)--(-1,-1);
    \draw[gray!70] (1,-1)--(-1,1);
    \def\sadlines{5}
    \foreach \y in {3,...,5}{
    \draw[gray!70] ({-\y/(\sadlines+1)},1) .. controls (0,{1.6-1.6*\y/(\sadlines+1)}) .. ({\y/(\sadlines+1)},1);
    \draw[gray!70] ({-\y/(\sadlines+1)},-1) .. controls (0,{-1.6+1.6*\y/(\sadlines+1)}) .. ({\y/(\sadlines+1)},-1);
    \draw[gray!70] (-1,{-\y/(\sadlines+1)}) .. controls ({-1.6+1.6*\y/(\sadlines+1)},0) .. (-1,{\y/(\sadlines+1)});
    \draw[gray!70] (1,{-\y/(\sadlines+1)}) .. controls ({1.6-1.6*\y/(\sadlines+1)},0) .. (1,{\y/(\sadlines+1)});
    }   \end{scope}
    \node[star,  star point ratio=2.25, fill=black, inner sep = 1pt] at (0,0) {};
    
    \draw[sloped,black,thick] (0,0) to[out=-90,in=90] node[pos=0.5]{\footnotesize $<$} (0,-1.2) node{\small $\bullet$} to[out=60,in=180] node[pos=0.5]{\footnotesize $>$} (1+\s,0) node[right]{$\alpha^{-1}$};
    \node at (-0.3,-0.75) {$\alpha$};
\end{tikzpicture}
\end{equation}
\end{definition}

As in Definition \ref{def:SNfunctor}, the $\alpha$-twisted string net modules of $\Sigma$ assemble into a functor
    \begin{equation}\label{eq:twistedSNfunctor}
    \begin{array}{cccccc}
    \SN^\alpha_\BB(\Sigma;-): & \SN^\alpha_\BB(\Gamma_{in})^{op} &\otimes &\SN^\alpha_B(\Gamma_{out}) &\to &\Vect\\
    & X&\otimes& Y  &\mapsto & \SN^\alpha_\BB(\Sigma;X\sqcup Y) \ .
    \end{array}
    \end{equation}

\begin{remark}
In the case $\alpha = \unit$, a $\unit$-twisted pivotal structure is simply a pivotal structure $\apiv$. The twist relation is implied by unitality of $\apiv$. The $\unit$-colored strands in the string nets with singularities can be omitted, and the sliding relations simply state that strands are allowed to pass over singularity points, i.e. that these can be omitted too. We recover the usual graphical calculus for pivotal categories \cite{Selinger} and the usual associated string net modules of surfaces \cite{LevinWen, KirillovStringNetTV}. Note that these string net modules are well-known to be defined on oriented surfaces, and do not depend on a chosen foliation or framing as above. We will devote Section \ref{sec:TQFT} to proving this result for $\alpha$-twisted string nets.
\end{remark}



\subsection{Values on disks with singularities}\label{sec:SNonD012}
We now compute some small examples of twisted string net modules, in particular to convince the reader that our choice of relations do not make these modules identically zero.
\begin{proposition}\label{prop:SNonD0}
For a foliated disk $\mathbb D_0$ with a single singularity of index 0, and for any boundary label $x \in \II$, there is an isomorphism of vector spaces
    \begin{equation*}
\SN_\II^\alpha\left( \mathbb D_0 :=
    \begin{tikzpicture}[scale = 0.5, baseline = 0pt]
    \foreach \y in {0,...,6}{
    \draw[gray!70] (0,0) circle(\y/3);
    }
    \draw[gray, ->] (45:2.2) -- (45:2.8) node[midway,below right = -2pt]{\small $\vec{y}$};
    \node[circle, fill=black, inner sep = 1pt] at (0,2) {};
    \node at (0,2.3) {$x$};
    \node[star,  star point ratio=2.25, fill=black, inner sep = 1pt] at (0,0) {};
    \end{tikzpicture} \right) \simeq \Hom_\BB(\alpha, x)
    \quad\text{    given by    }\quad
    \begin{tikzpicture}[scale = 0.5, baseline = 0pt]
    \foreach \y in {0,...,6}{
    \draw[gray!70] (0,0) circle(\y/3);
    }
    \node[circle, fill=black, inner sep = 1pt] (X) at (0,2) {};
    \node at (0,2.3) {$x$};
    \node[rectangle, draw=black, fill=white, minimum width = 20pt, inner sep = 3pt] (F) at (0,1) {$m$};
    \draw (0,0)--(F) node[midway, sloped]{\tiny $>$} node[pos = 0.3, left]{\small$\alpha$} --(X) node[midway, sloped]{\tiny $>$};
    \node[star,  star point ratio=2.25, fill=black, inner sep = 1pt] at (0,0) {};
    \end{tikzpicture}  \mapsfrom  m \ .
    \end{equation*}
\end{proposition}
\begin{proof} As $x\in\II$, the admissibility conditions on string nets with boundary $x$ are automatic, and $\SN^\alpha_\II(\D_0;x)=\SN^\alpha_\BB(\D_0;x)$.

    Consider a progressive arc $\gamma$ from the critical point $c$ to the boundary as below. The complement of $\gamma$ is a disk with standard foliation, on which one can apply a progressive skein relation.    
    \begin{equation*}
    \begin{tikzpicture}[scale = 1, baseline = 4mm]
    \fill[gray!10] (0,0) rectangle (1,1);
    \foreach \y in {0,...,5} {
    \draw[gray!50] (0,\y/5) -- ++(1,0);
    }
    \draw[thick] (0,0) rectangle (1,1);
    \end{tikzpicture}
    \hookrightarrow
    \begin{tikzpicture}[scale = 0.5, baseline = 0pt]
    \fill[gray!10] (-60:0.4) -- (-60:1.7) arc(-60:240:1.7) -- (240:0.4) arc(240:-60:0.4);
    \foreach \y in {0,...,6}{
    \draw[gray!70] (0,0) circle(\y/3);
    }
    \draw[thick, red!50!black] (0,0) -- (0,-2) node[pos = 0.7, sloped]{\small $>$} node[pos = 0.7, right=-1pt]{$\gamma$};
    \draw[thick] (-60:0.4) -- (-60:1.7) arc(-60:240:1.7) -- (240:0.4) arc(240:-60:0.4);
    \node[star,  star point ratio=2.25, fill=black, inner sep = 1pt] at (0,0) {};
    \end{tikzpicture}
    \end{equation*} By our rule for string nets with singularities, a string net in $\mathbb D_0 \smallsetminus \gamma$ must have incoming boundary $\alpha$ and outgoing boundary $x$, and will evaluate to a well-defined morphism $m \in \Hom_\BB(\alpha, x)$ as claimed. 

    It remains to be proven that any string net in $\D_0$ can be made disjoint from $\gamma$ in a canonical way. In other words, if we denote $\SN_\II^\alpha(\D_0,\gamma;x)$ the space of string nets which are disjoint from $\gamma$ except at the critical point, where it is not tangent to $\gamma$, modulo isotopies preserving this property and skein relations away from $\gamma$, then we want to prove that the inclusion
    \begin{equation}
        \SN_\II^\alpha(\D_0,\gamma;x) \to \SN_\II^\alpha(\D_0;x)
    \end{equation}
    is an isomorphism.
    
    Consider a string net in $\SN_\BB^\alpha(\D_0;x)$ and choose a representative $T$, not considered up to isotopy. Generically, we can ensure that $T$ intersects $\gamma$ transversely. In particular, we ensure that the strand of $T$ ending at the critical point $c$ is not tangent to $\gamma$ at $c$. 

    Consider an intersection point $p\in T \cap \gamma$, which we will assume is the closest to $c$ along $\gamma$. Then in a neighborhood of $\gamma$ we are in one of the situations below, depending on whether the algebraic intersection of $\gamma$ and $T$ is $+1$ or $-1$. 
    \begin{equation*}
        \begin{tikzpicture}[scale = 1, baseline = -5pt]
        \clip (-0.5,-1.85) rectangle (0.5,1);
    \foreach \y in {0,...,6}{
    \draw[gray!70] (0,0) circle(\y/3);
    }
    \node[circle, fill=black, inner sep = 1pt] at (0,2) {};
    \node[star,  star point ratio=2.25, fill=black, inner sep = 1pt] at (0,0) {};
    \draw (0,0) -- (0,2) node[pos = 0.3, sloped]{\small $>$} node[pos = 0.3, left]{$\alpha$};
    \draw[thick, red!50!black] (0,0) -- (0,-2) node[pos = 0.7, sloped]{\small $>$} node[pos = 0.7, left]{$\gamma$};
    \draw (-0.5,-0.6)-- (0.5,-1.5)node[pos = 0.7, sloped]{\small $>$} node[pos = 0.7, above right]{$x$};
    \end{tikzpicture} \quad\quad \text{or} \quad\quad
            \begin{tikzpicture}[scale = 1, baseline = -5pt]
        \clip (-0.5,-1.85) rectangle (0.5,1);
    \foreach \y in {0,...,6}{
    \draw[gray!70] (0,0) circle(\y/3);
    }
    \node[circle, fill=black, inner sep = 1pt] at (0,2) {};
    \node[star,  star point ratio=2.25, fill=black, inner sep = 1pt] at (0,0) {};
    \draw (0,0) -- (0,2) node[pos = 0.3, sloped]{\small $>$} node[pos = 0.3, left]{$\alpha$};
    \draw[thick, red!50!black] (0,0) -- (0,-2) node[pos = 0.7, sloped]{\small $>$} node[pos = 0.7, right]{$\gamma$};
    \draw (0.5,-0.6)-- (-0.5,-1.5)node[pos = 0.7, sloped]{\small $<$} node[pos = 0.7, above left]{$x$};
    \end{tikzpicture}
    \end{equation*}
In either of these cases, we will apply the sliding relation \eqref{eq:skein-slide-0} to eliminate the intersection point $p$. The relation can only be applied in the presence of an appropriate coevaluation, which we can introduce by a snake relation. We get:
\begin{equation*}
    \begin{tikzpicture}[xscale = 2,yscale = 1, baseline = -5pt]
        \clip (-0.5,-1.85) rectangle (0.5,1);
    \foreach \y in {0,...,6}{
    \draw[gray!70] (0,0) circle(\y/3);
    }
    \node[circle, fill=black, inner sep = 1pt] at (0,2) {};
    \node[star,  star point ratio=2.25, fill=black, inner sep = 1pt] at (0,0) {};
    \draw (0,0) -- (0,2) node[pos = 0.3, sloped]{\footnotesize $>$} node[pos = 0.3, left]{$\alpha$};
    \draw[thick, red!50!black] (0,0) -- (0,-2) node[pos = 0.7, sloped]{\small $>$} node[pos = 0.7, left]{$\gamma$};
    \draw (-0.5,-0.6)-- (0.5,-1.5)node[pos = 0.7, sloped]{\small $>$} node[pos = 0.7, above right]{$x$};
    \end{tikzpicture}
    \ \overset{\text{snake}}=\ 
    \begin{tikzpicture}[xscale = 2,yscale = 1, baseline = -5pt]
        \clip (-0.5,-1.85) rectangle (0.5,1);
    \foreach \y in {0,...,6}{
    \draw[gray!70] (0,0) circle(\y/3);
    }
    \node[circle, fill=black, inner sep = 1pt] at (0,2) {};
    \node[star,  star point ratio=2.25, fill=black, inner sep = 1pt] at (0,0) {};
    \draw (0,0) -- (0,2) node[pos = 0.3, sloped]{\footnotesize $>$} node[pos = 0.3, left]{$\alpha$};
    \draw[thick, red!50!black] (0,0) -- (0,-2) node[pos = 0.7, sloped]{\footnotesize $>$} node[pos = 0.7, left]{$\gamma$};
    \draw (-0.5,-0.6) to[out = -45, in = 125]node[pos = 0.3, sloped]{\footnotesize $>$} node[pos = 0.3, above right]{$x$} (0.2, -1.2) node{\small $\bullet$} to [out = 80, in=-80]node[pos = 0.7, sloped]{\footnotesize $>$} node[pos = 0.7, left]{$x^*$} (0.2,-0.1)node{\small $\bullet$} to[out = 30, in = 125]node[pos = 0.9, sloped]{\footnotesize $>$} node[pos = 0.9, below]{$x$} (0.5,-1.5);
    \end{tikzpicture} 
    \ \overset{\eqref{eq:skein-slide-0}}=\ 
    \begin{tikzpicture}[xscale = 2,yscale = 1, baseline = -5pt]
        \clip (-0.5,-1.85) rectangle (0.5,1);
    \foreach \y in {0,...,6}{
    \draw[gray!70] (0,0) circle(\y/3);
    }
    \node[circle, fill=black, inner sep = 1pt] at (0,2) {};
    \node[star,  star point ratio=2.25, fill=black, inner sep = 1pt] at (0,0) {};
    \draw (0,0) -- (0,2) node[pos = 0.4, sloped]{\footnotesize $>$} node[pos = 0.4, left]{$\alpha$};
    \draw[thick, red!50!black] (0,0) -- (0,-2) node[pos = 0.7, sloped]{\footnotesize $>$} node[pos = 0.7, left]{$\gamma$};
    \draw (-0.5,-0.6) to[out = -45, in = 125]node[pos = 0.3, sloped]{\footnotesize $>$} node[pos = 0.3, below right]{$x$} 
    (0.2, -1.2) node{\small $\bullet$} to [out = 100, in=-160]node[pos = 0.5, sloped]{\footnotesize $>$} node[pos = 0.5, right]{$x^*$} 
    (-0.2,0.1)node{\small $\bullet$} to[out = 90, in = 95, out distance = 5mm, in distance = 4mm]node[pos = 0.2, sloped]{\footnotesize $>$} node[pos = 0.2, left]{$x^{**}$}
    (0.45,0) to[out= -85, in = 125]node[pos = 0.9, sloped]{\footnotesize $>$} node[pos = 0.9, below]{$x$} (0.5,-1.5);
    \node[rectangle, draw, fill=white, inner sep = 3pt] at (0,0.4){};
    \end{tikzpicture} 
    \ \overset{\text{isotopy}}=\ 
    \begin{tikzpicture}[xscale = 2,yscale = 1, baseline = -5pt]
        \clip (-0.5,-1.85) rectangle (0.5,1);
    \foreach \y in {0,...,6}{
    \draw[gray!70] (0,0) circle(\y/3);
    }
    \node[circle, fill=black, inner sep = 1pt] at (0,2) {};
    \node[star,  star point ratio=2.25, fill=black, inner sep = 1pt] at (0,0) {};
    \draw (0,0) -- (0,2) node[pos = 0.4, sloped]{\footnotesize $>$} node[pos = 0.4, left]{$\alpha$};
    \draw[thick, red!50!black] (0,0) -- (0,-2) node[pos = 0.7, sloped]{\footnotesize $>$} node[pos = 0.7, left]{$\gamma$};
    \draw (-0.5,-0.6) to[out = -45, in = 125]node[pos = 0.3, sloped]{\footnotesize $>$} node[pos = 0.3, below]{$x$} 
    (-0.2, -1) node{\small $\bullet$} to [out = 100, in=-160]node[pos = 0.5, sloped]{\footnotesize $>$} node[pos = 0.5, right]{$x^*$} 
    (-0.2,0.1)node{\small $\bullet$} to[out = 90, in = 95, out distance = 5mm, in distance = 4mm]node[pos = 0.2, sloped]{\footnotesize $>$} node[pos = 0.2, left]{$x^{**}$}
    (0.45,0) to[out= -85, in = 125]node[pos = 0.9, sloped]{\footnotesize $>$} node[pos = 0.9, below]{$x$} (0.5,-1.5);
    \node[rectangle, draw, fill=white, inner sep = 3pt] at (0,0.4){};
    \end{tikzpicture} 
\end{equation*}
or 
\begin{equation*}
    \begin{tikzpicture}[xscale = 2,yscale = 1, baseline = -5pt]
        \clip (-0.5,-1.85) rectangle (0.5,1);
    \foreach \y in {0,...,6}{
    \draw[gray!70] (0,0) circle(\y/3);
    }
    \node[circle, fill=black, inner sep = 1pt] at (0,2) {};
    \node[star,  star point ratio=2.25, fill=black, inner sep = 1pt] at (0,0) {};
    \draw (0,0) -- (0,2) node[pos = 0.3, sloped]{\footnotesize $>$} node[pos = 0.3, left]{$\alpha$};
    \draw[thick, red!50!black] (0,0) -- (0,-2) node[pos = 0.7, sloped]{\small $>$} node[pos = 0.7, right]{$\gamma$};
    \draw (0.5,-0.6)-- (-0.5,-1.5)node[pos = 0.7, sloped]{\footnotesize $<$} node[pos = 0.7, above left]{$x$};
    \end{tikzpicture}
    \ \overset{\text{snake}}=\ 
    \begin{tikzpicture}[xscale = 2,yscale = 1, baseline = -5pt]
        \clip (-0.5,-1.85) rectangle (0.5,1);
    \foreach \y in {0,...,6}{
    \draw[gray!70] (0,0) circle(\y/3);
    }
    \node[circle, fill=black, inner sep = 1pt] at (0,2) {};
    \node[star,  star point ratio=2.25, fill=black, inner sep = 1pt] at (0,0) {};
    \draw (0,0) -- (0,2) node[pos = 0.3, sloped]{\footnotesize $>$} node[pos = 0.3, left]{$\alpha$};
    \draw[thick, red!50!black] (0,0) -- (0,-2) node[pos = 0.7, sloped]{\footnotesize $>$} node[pos = 0.7, right]{$\gamma$};
    \draw (0.5,-0.6) to[out = -135, in = 45]node[pos = 0.3, sloped]{\footnotesize $<$} node[pos = 0.3, below]{$x$}
    (0.2, -0.9) node{\small $\bullet$} to [out = 80, in=45, in distance = 4mm]node[pos = 0.6, sloped]{\footnotesize $<$} node[pos = 0.6, above]{${}^*x$} 
    (0.2,-0.1) node{\small $\bullet$} to[out = -80, in = 45]node[pos = 0.8, sloped]{\footnotesize $<$} node[pos = 0.8, below]{$x$} (-0.5,-1.5);
    \end{tikzpicture} 
    \ \overset{\eqref{eq:skein-slide-0}}=\ 
    \begin{tikzpicture}[xscale = 2,yscale = 1, baseline = -5pt]
        \clip (-0.5,-1.85) rectangle (0.5,1);
    \foreach \y in {0,...,6}{
    \draw[gray!70] (0,0) circle(\y/3);
    }
    \node[circle, fill=black, inner sep = 1pt] at (0,2) {};
    \node[star,  star point ratio=2.25, fill=black, inner sep = 1pt] at (0,0) {};
    \draw (0,0) -- (0,2) node[pos = 0.4, sloped]{\footnotesize $>$} node[pos = 0.4, left]{$\alpha$};
    \draw[thick, red!50!black] (0,0) -- (0,-2) node[pos = 0.7, sloped]{\footnotesize $>$} node[pos = 0.7, right]{$\gamma$};
    \draw (0.5,-0.6) to[out = -135, in = 45]node[pos = 0.3, sloped]{\footnotesize $<$} node[pos = 0.3, below]{$x$}
    (0.2, -0.9) node{\small $\bullet$} to [out = 80, in=-90, in distance = 4mm]node[pos = 0.6, sloped]{\footnotesize $<$} node[pos = 0.6, left]{${}^*x$}
    (0.45,0) to[out = 95, in=90, in distance = 5mm, out distance = 4mm] node[pos = 0.8, sloped]{\footnotesize $>$} node[pos = 0.8, left]{$x^*$}
    (-0.2,0.1)node{\small $\bullet$} to [out = -160, in = 45]node[pos = 0.8, sloped]{\footnotesize $<$} node[pos = 0.8, below]{$x$}(-0.5,-1.5);
    \node[rectangle, draw, fill=white, inner sep = 3pt] at (0,0.4){};
    \end{tikzpicture} 
\end{equation*}
Doing this procedure for every intersection point of $T$ and $\gamma$ produces a graph which is disjoint from $\gamma$, hence a morphism in $\Hom_\BB(\alpha, x)$ as desired. We need to check that this procedure is well-defined and does not depend on the choice of generic representative for our string net.

Two representatives $T$ and $T'$ for the same string net will differ by isotopies of progressive graphs, sliding over the critical point relations and progressive skein relations. Up to isotopy, any progressive skein relation can be pushed to happen away from $\gamma$, so we only have to check invariance under the first two.

A generic isotopy between $T$ and $T'$ can be obtained as a sequence of isotopies of graphs intersecting $\gamma$ transversely, and the following moves:
\begin{itemize}
    \item two intersection points of $\gamma$ and $T$ cancel:
    \begin{tikzpicture}[scale = 1, baseline = -30pt]
        \clip (-0.5,-1.85) rectangle (0.5,-0.1);
    \foreach \y in {0,...,6}{
    \draw[gray!70] (0,0) circle(\y/3);
    }
    \node[circle, fill=black, inner sep = 1pt] at (0,2) {};
    \node[star,  star point ratio=2.25, fill=black, inner sep = 1pt] at (0,0) {};
    \draw (0,0) -- (0,2) node[pos = 0.3, sloped]{\small $>$} node[pos = 0.3, left]{$\alpha$};
    \draw[thick, red!50!black] (0,0) -- (0,-2) node[pos = 0.7, sloped]{\small $>$} node[pos = 0.7, right]{$\gamma$};
    \draw (-0.5,-0.3)to[out = -45, in = 25, looseness=2.2] (-0.5,-1.5)node[pos = 0.7, sloped]{\small $>$};
    \end{tikzpicture} 
    $\leadsto$
    \begin{tikzpicture}[scale = 1, baseline = -30pt]
        \clip (-0.5,-1.85) rectangle (0.5,-0.1);
    \foreach \y in {0,...,6}{
    \draw[gray!70] (0,0) circle(\y/3);
    }
    \node[circle, fill=black, inner sep = 1pt] at (0,2) {};
    \node[star,  star point ratio=2.25, fill=black, inner sep = 1pt] at (0,0) {};
    \draw (0,0) -- (0,2) node[pos = 0.3, sloped]{\small $>$} node[pos = 0.3, left]{$\alpha$};
    \draw[thick, red!50!black] (0,0) -- (0,-2) node[pos = 0.7, sloped]{\small $>$} node[pos = 0.7, right]{$\gamma$};
    \draw (-0.5,-0.3)to[out = -45, in = 25, looseness=1.2] (-0.5,-1.5)node[pos = 0.7, sloped]{\small $>$};
    \end{tikzpicture} 
    and its mirror. Invariance under this move is easily checked by applying the procedure above to both intersection points and using \eqref{eq:pivNatEval} and snake relations.

    \item a vertex of $T$ passes through $\gamma$: \begin{tikzpicture}[scale = 1, baseline = -30pt]
        \clip (-0.5,-1.85) rectangle (0.5,-0.1);
    \foreach \y in {0,...,6}{
    \draw[gray!70] (0,0) circle(\y/3);
    }
    \draw[thick, red!50!black] (0,0) -- (0,-2) node[pos = 0.7, sloped]{\small $>$} node[pos = 0.7, left]{$\gamma$};
    \draw (-0.5,-0.6)-- (0.5,-1.5)node[pos = 0.7, sloped]{\tiny $>$}node[pos = 0.1, sloped]{\tiny $>$} coordinate[pos=0.3] (P);
    \draw (-0.5,-0.3) -- (P)node[pos = 0.5, sloped]{\tiny $>$}node{\small $\bullet$};
    \draw (0.5,-0.3) -- (P)node[pos = 0.5, sloped]{\tiny $<$};
    \draw (-0.5,-1.5) -- (P)node[pos = 0.5, sloped]{\tiny $<$};
    \end{tikzpicture} 
    $\leadsto$ 
    \begin{tikzpicture}[scale = 1, baseline = -30pt]
        \clip (-0.5,-1.85) rectangle (0.5,-0.1);
    \foreach \y in {0,...,6}{
    \draw[gray!70] (0,0) circle(\y/3);
    }
    \node[circle, fill=black, inner sep = 1pt] at (0,2) {};
    \node[star,  star point ratio=2.25, fill=black, inner sep = 1pt] at (0,0) {};
    \draw (0,0) -- (0,2) node[pos = 0.3, sloped]{\small $>$} node[pos = 0.3, left]{$\alpha$};
    \draw[thick, red!50!black] (0,0) -- (0,-2) node[pos = 0.8, sloped]{\small $>$} node[pos = 0.7, below right]{$\gamma$};
    \draw (-0.5,-0.6)-- (0.5,-1.5)node[pos = 0.85, sloped]{\tiny $>$}node[pos = 0.2, sloped]{\tiny $>$} coordinate[pos=0.7] (P);
    \draw (-0.5,-0.3) -- (P)node[pos = 0.5, sloped]{\tiny $>$}node{\small $\bullet$};
    \draw (0.5,-0.3) -- (P)node[pos = 0.5, sloped]{\tiny $<$};
    \draw (-0.5,-1.5) -- (P)node[pos = 0.5, sloped]{\tiny $<$};
    \end{tikzpicture} . We will decompose this into smaller pieces.
    
    First, if the vertex is adjacent to only two edges coming from opposite directions, then invariance under \begin{tikzpicture}[scale = 1, baseline = -30pt]
        \clip (-0.5,-1.85) rectangle (0.5,-0.1);
    \foreach \y in {0,...,6}{
    \draw[gray!70] (0,0) circle(\y/3);
    }
    \node[circle, fill=black, inner sep = 1pt] at (0,2) {};
    \node[star,  star point ratio=2.25, fill=black, inner sep = 1pt] at (0,0) {};
    \draw (0,0) -- (0,2) node[pos = 0.3, sloped]{\small $>$} node[pos = 0.3, left]{$\alpha$};
    \draw[thick, red!50!black] (0,0) -- (0,-2) node[pos = 0.7, sloped]{\small $>$} node[pos = 0.7, left]{$\gamma$};
    \draw (-0.5,-0.6)-- (0.5,-1.5)node[pos = 0.7, sloped]{\small $>$} coordinate[pos=0.3] (P)node[pos=0.3]{\small $\bullet$};
    \end{tikzpicture} 
    $\leadsto$ 
    \begin{tikzpicture}[scale = 1, baseline = -30pt]
        \clip (-0.5,-1.85) rectangle (0.5,-0.1);
    \foreach \y in {0,...,6}{
    \draw[gray!70] (0,0) circle(\y/3);
    }
    \node[circle, fill=black, inner sep = 1pt] at (0,2) {};
    \node[star,  star point ratio=2.25, fill=black, inner sep = 1pt] at (0,0) {};
    \draw (0,0) -- (0,2) node[pos = 0.3, sloped]{\small $>$} node[pos = 0.3, left]{$\alpha$};
    \draw[thick, red!50!black] (0,0) -- (0,-2) node[pos = 0.7, sloped]{\small $>$} node[pos = 0.7, left]{$\gamma$};
    \draw (-0.5,-0.6)-- (0.5,-1.5)node[pos = 0.85, sloped]{\small $>$} coordinate[pos=0.7] (P) node[pos=0.7]{\small $\bullet$};
    \end{tikzpicture} follows from naturality \eqref{eq:pivNatGraphically}. 

    Second, if the vertex is 3-valent and colored by an identity coupon (i.e. the fusion of two lines) all of which intersect $\gamma$ with the same sign during the move \begin{tikzpicture}[scale = 1, baseline = -30pt]
        \clip (-0.5,-1.85) rectangle (0.5,-0.1);
    \foreach \y in {0,...,6}{
    \draw[gray!70] (0,0) circle(\y/3);
    }
    \node[circle, fill=black, inner sep = 1pt] at (0,2) {};
    \node[star,  star point ratio=2.25, fill=black, inner sep = 1pt] at (0,0) {};
    \draw (0,0) -- (0,2) node[pos = 0.3, sloped]{\small $>$} node[pos = 0.3, left]{$\alpha$};
    \draw[thick, red!50!black] (0,0) -- (0,-2) node[pos = 0.7, sloped]{\small $>$} node[pos = 0.7, left]{$\gamma$};
    \draw (-0.5,-0.6)-- (0.5,-1.5)node[pos = 0.7, sloped]{\small $>$} coordinate[pos=0.3] (P);
    \draw (-0.5,-0.3)to[out = -55,in = 125] (P)node{\tiny $\bullet$};
    \end{tikzpicture} 
    $\leadsto$ 
    \begin{tikzpicture}[scale = 1, baseline = -30pt]
        \clip (-0.5,-1.85) rectangle (0.5,-0.1);
    \foreach \y in {0,...,6}{
    \draw[gray!70] (0,0) circle(\y/3);
    }
    \node[circle, fill=black, inner sep = 1pt] at (0,2) {};
    \node[star,  star point ratio=2.25, fill=black, inner sep = 1pt] at (0,0) {};
    \draw (0,0) -- (0,2) node[pos = 0.3, sloped]{\small $>$} node[pos = 0.3, left]{$\alpha$};
    \draw[thick, red!50!black] (0,0) -- (0,-2) node[pos = 0.7, sloped]{\small $>$} node[pos = 0.7, left]{$\gamma$};
    \draw (-0.5,-0.6)-- (0.5,-1.5)node[pos = 0.85, sloped]{\small $>$} coordinate[pos=0.7] (P);
    \draw (-0.5,-0.3)to[out = -55,in = 125] (P)node{\tiny $\bullet$};
    \end{tikzpicture}, invariance follows from  \eqref{eq:pivMonoidalGraphically}.
    
Finally, we can decompose the general case as:
\begin{equation*}
\begin{tikzpicture}[scale = 1.5, baseline = -30pt]
        \clip (-0.5,-1.85) rectangle (0.5,-0.1);
    \foreach \y in {0,...,6}{
    \draw[gray!70] (0,0) circle(\y/3);
    }
    \draw[thick, red!50!black] (0,0) -- (0,-2) node[pos = 0.8, sloped]{\small $>$} node[pos = 0.7, below right]{$\gamma$};
    \draw (-0.5,-0.6)-- (0.5,-1.5)node[pos = 0.7, sloped]{\tiny $>$}node[pos = 0.1, sloped]{\tiny $>$} coordinate[pos=0.3] (P);
    \draw (-0.5,-0.3) -- (P)node[pos = 0.5, sloped]{\tiny $>$}node{\small $\bullet$};
    \draw (0.5,-0.3) -- (P)node[pos = 0.5, sloped]{\tiny $<$};
    \draw (-0.5,-1.5) -- (P)node[pos = 0.5, sloped]{\tiny $<$};
    \end{tikzpicture} 
\overset{\eqref{eq:pivNatEval}}\leadsto
    \begin{tikzpicture}[scale = 1.5, baseline = -30pt]
        \clip (-0.5,-1.85) rectangle (0.5,-0.1);
    \foreach \y in {0,...,6}{
    \draw[gray!70] (0,0) circle(\y/3);
    }
    \draw[thick, red!50!black] (0,0) -- (0,-2) node[pos = 0.8, sloped]{\small $>$} node[pos = 0.7, below right]{$\gamma$};
    \draw (-0.5,-0.6)-- (0.5,-1.5)node[pos = 0.7, sloped]{\tiny $>$}node[pos = 0.1, sloped]{\tiny $>$} coordinate[pos=0.3] (P);
    \draw (-0.5,-0.3) to[out = -70,in = 120] node[pos = 0.5, sloped]{\tiny $>$}(P)node{\small $\bullet$};
    \draw (0.5,-0.3) -- (P)node[pos = 0.5, sloped]{\tiny $<$};
    \draw (-0.5,-1.6) to[in = -50, out = 25, in distance = 30pt] node[pos = 0.3, sloped]{\tiny $<$}(P);
    \end{tikzpicture} 
\overset{\eqref{eq:pivMonoidalGraphically}}\leadsto
    \begin{tikzpicture}[scale = 1.5, baseline = -30pt]
        \clip (-0.5,-1.85) rectangle (0.5,-0.1);
    \foreach \y in {0,...,6}{
    \draw[gray!70] (0,0) circle(\y/3);
    }
    \draw[thick, red!50!black] (0,0) -- (0,-2) node[pos = 0.8, sloped]{\small $>$} node[pos = 0.7, below right]{$\gamma$};
    \draw (-0.5,-0.6)-- (0.5,-1.5)node[pos = 0.85, sloped]{\tiny $>$}node[pos = 0.1, sloped]{\tiny $>$} coordinate[pos=0.3] (P)coordinate[pos=0.7] (Pprime);
    \draw (-0.5,-0.3) to[out = -70,in = 120] node[pos = 0.5, sloped]{\tiny $>$}(P)node{\small $\bullet$};
    \draw (0.5,-0.3) -- (P)node[pos = 0.5, sloped]{\tiny $<$};
    \draw (-0.5,-1.6) to[in = -55, out = 25] node[pos = 0.5, sloped]{\tiny $<$}(Pprime) node{\tiny $\bullet$};
    \end{tikzpicture} 
\overset{\eqref{eq:pivNatGraphically}}\leadsto
    \begin{tikzpicture}[scale = 1.5, baseline = -30pt]
        \clip (-0.5,-1.85) rectangle (0.5,-0.1);
    \foreach \y in {0,...,6}{
    \draw[gray!70] (0,0) circle(\y/3);
    }
    \draw[thick, red!50!black] (0,0) -- (0,-2) node[pos = 0.8, sloped]{\small $>$} node[pos = 0.7, below right]{$\gamma$};
    \draw (-0.5,-0.6)-- (0.5,-1.5)node[pos = 0.85, sloped]{\tiny $>$}node[pos = 0.1, sloped]{\tiny $>$} coordinate[pos=0.7] (P)coordinate[pos=0.3] (Pprime);
    \draw (-0.5,-0.3) to[out = -70,in = 120] node[pos = 0.5, sloped]{\tiny $>$}(Pprime)node{\tiny $\bullet$};
    \draw (0.5,-0.3) to[out = -130, in = 110, in distance = 15pt] node[pos = 0.5, sloped]{\tiny $<$} (Pprime);
    \draw (-0.5,-1.6) -- (P) node[pos = 0.5, sloped]{\tiny $<$} node{\small $\bullet$};
    \end{tikzpicture} 
\overset{\eqref{eq:pivMonoidalGraphically}}\leadsto
    \begin{tikzpicture}[scale = 1.5, baseline = -30pt]
        \clip (-0.5,-1.85) rectangle (0.5,-0.1);
    \foreach \y in {0,...,6}{
    \draw[gray!70] (0,0) circle(\y/3);
    }
    \draw[thick, red!50!black] (0,0) -- (0,-2) node[pos = 0.8, sloped]{\small $>$} node[pos = 0.7, below right]{$\gamma$};
    \draw (-0.5,-0.6)-- (0.5,-1.5)node[pos = 0.85, sloped]{\tiny $>$}node[pos = 0.1, sloped]{\tiny $>$} coordinate[pos=0.7] (P)coordinate[pos=0.3] (Pprime);
    \draw (-0.5,-0.3) to[out = -70,in = 130] node[pos = 0.5, sloped]{\tiny $>$}(P);
    \draw (0.5,-0.3) to[out = -130, in = 125, in distance = 35pt] node[pos = 0.5, sloped]{\tiny $<$} (P);
    \draw (-0.5,-1.6) -- (P) node[pos = 0.5, sloped]{\tiny $<$} node{\small $\bullet$};
    \end{tikzpicture} 
\overset{\eqref{eq:pivNatEval}}\leadsto
    \begin{tikzpicture}[scale = 1.5, baseline = -30pt]
        \clip (-0.5,-1.85) rectangle (0.5,-0.1);
    \foreach \y in {0,...,6}{
    \draw[gray!70] (0,0) circle(\y/3);
    }
    \node[circle, fill=black, inner sep = 1pt] at (0,2) {};
    \node[star,  star point ratio=2.25, fill=black, inner sep = 1pt] at (0,0) {};
    \draw (0,0) -- (0,2) node[pos = 0.3, sloped]{\small $>$} node[pos = 0.3, left]{$\alpha$};
    \draw[thick, red!50!black] (0,0) -- (0,-2) node[pos = 0.8, sloped]{\small $>$} node[pos = 0.7, below right]{$\gamma$};
    \draw (-0.5,-0.6)-- (0.5,-1.5)node[pos = 0.85, sloped]{\tiny $>$}node[pos = 0.2, sloped]{\tiny $>$} coordinate[pos=0.7] (P);
    \draw (-0.5,-0.3) -- (P)node[pos = 0.5, sloped]{\tiny $>$}node{\small $\bullet$};
    \draw (0.5,-0.3) -- (P)node[pos = 0.5, sloped]{\tiny $<$};
    \draw (-0.5,-1.5) -- (P)node[pos = 0.5, sloped]{\tiny $<$};
    \end{tikzpicture} 
\end{equation*}
    \item $T$ becomes tangent to $\gamma$ at $c$: \begin{tikzpicture}[scale = 1, baseline = -5pt]
        \clip (-0.5,-0.95) rectangle (0.5,0.95);
    \foreach \y in {0,...,6}{
    \draw[gray!70] (0,0) circle(\y/3);
    }
    \node[circle, fill=black, inner sep = 1pt] at (0,2) {};
    \node[star,  star point ratio=2.25, fill=black, inner sep = 1pt] at (0,0) {};
    \draw (0,0) to[out = -135, in = -110, out distance = 8mm] node[pos = 0.5, sloped]{\small $>$} node[pos = 0.5, right]{$\alpha$} (0,2);
    \draw[thick, red!50!black] (0,0) -- (0,-2) node[pos = 0.3, sloped]{\small $>$} node[pos = 0.3, left]{$\gamma$};
    \end{tikzpicture}
    $\leadsto$ 
    \begin{tikzpicture}[scale = 1, baseline = -5pt]
        \clip (-0.5,-0.95) rectangle (0.5,0.95);
    \foreach \y in {0,...,6}{
    \draw[gray!70] (0,0) circle(\y/3);
    }
    \node[circle, fill=black, inner sep = 1pt] at (0,2) {};
    \node[star,  star point ratio=2.25, fill=black, inner sep = 1pt] at (0,0) {};
    \draw (0,0) to[out = -35, in = -80, out distance = 5mm, in distance = 6mm] (-0.4,0) to[out = 100, in = -90] node[pos = 0.2, sloped]{\small $>$} node[pos = 0.2, right]{$\alpha$} (0,2);
    \draw[thick, red!50!black] (0,0) -- (0,-2) node[pos = 0.3, sloped]{\small $>$} node[pos = 0.3, left]{$\gamma$};
    \end{tikzpicture}. Invariance under this move follows from the twist relation \eqref{eq:otherTwist}:
    
\begin{equation}\label{eq:wrappingalphaaround}
    \begin{tikzpicture}[scale = 0.64, baseline = 0pt]
    \foreach \y in {0,...,6}{
    \draw[gray!70] (0,0) circle(\y/3);
    }
    \draw[sloped,black,thick] (0,0) to[out=90,in=154] (45:{0.4}) ;
    \draw[sloped,black,thick] (135:{1.4}) to[out=45,in=-90] node[pos=0]{\footnotesize $>$} (0,2) node[above]{$\alpha$};
    \node[star,  star point ratio=2.25, fill=black, inner sep = 1pt] at (0,0) {};
    \draw[black, thick] plot[domain=0:270, variable=\r, samples=200] (45-\r:{\r/270 + 0.4});
    \end{tikzpicture}
    \ \leadsto \
    \begin{tikzpicture}[scale = 0.64, baseline = 0pt]
    \foreach \y in {0,...,6}{
    \draw[gray!70] (0,0) circle(\y/3);
    }
    \draw[sloped,black,thick] (0,0) to[out=90,in=-180] node[pos=0.5]{\footnotesize $>$} (70:{1.4}) node {\footnotesize $\bullet$} to[out=-90,in=30] (0:{0.3}) node {$\bullet$} to[out=-30,in=20] (-90:{0.7}) to[out=200,in=-80] (180:{1}) to[out=100,in=-90] (0,2);
    \node[star,  star point ratio=2.25, fill=black, inner sep = 1pt] at (0,0) {};
    \end{tikzpicture}
    \ \leadsto \
    \begin{tikzpicture}[scale = 0.64, baseline = 0pt]
    \foreach \y in {0,...,6}{
    \draw[gray!70] (0,0) circle(\y/3);
    }
    \draw[sloped,black,thick] (0,0) to[out=90,in=-180] node[pos=0.8]{\footnotesize $>$} node[pos = 0.35, rectangle, draw, fill=white, inner sep = 3pt, rotate = 10]{} (70:{1.4}) node {\footnotesize $\bullet$} to[out=-60,in=110] (180:{0.4}) node {\footnotesize $\bullet$}  to[out=150,in=-90] (0,2);
    \node[star,  star point ratio=2.25, fill=black, inner sep = 1pt] at (0,0) {};
    \end{tikzpicture}
    \ \sim \
    \begin{tikzpicture}[scale = 0.64, baseline = 0pt]
    \foreach \y in {0,...,6}{
    \draw[gray!70] (0,0) circle(\y/3);
    }
    \draw[black,thick] (0,0) to[out=90,in=-90] node[pos=0.5,sloped]{\footnotesize $>$}node[pos=0.5, right]{$\alpha$} (0,2);
    \node[star,  star point ratio=2.25, fill=black, inner sep = 1pt] at (0,0) {};
    \end{tikzpicture}
    \end{equation}
\end{itemize}
Finally, if $T$ and $T'$ differ by a sliding relation \eqref{eq:skein-slide-0}, then $T$ has one more intersection point with $\gamma$, that we eliminate as above using the same sliding operation, so our procedure is invariant under this move.
\end{proof}

\begin{proposition}\label{prop:SNonD1}
For a foliated disk $\mathbb D_1$ with a single singularity of index 1, and for any boundary label $x \in \II$ as displayed below, there is an isomorphism of vector spaces
    \begin{equation*}
    \SN_\II^\alpha\left(
        \begin{tikzpicture}[baseline = 0pt]
        \begin{scope}
        \clip (0,0) circle(1.2);
    \draw[gray!70] (1,1)--(-1,-1);
    \draw[gray!70] (1,-1)--(-1,1);
    \def\sadlines{5}
    \foreach \y in {3,...,5}{
    \draw[gray!70] ({-\y/(\sadlines+1)},1) .. controls (0,{1.6-1.6*\y/(\sadlines+1)}) .. ({\y/(\sadlines+1)},1);
    \draw[gray!70] ({-\y/(\sadlines+1)},-1) .. controls (0,{-1.6+1.6*\y/(\sadlines+1)}) .. ({\y/(\sadlines+1)},-1);
    \draw[gray!70] (-1,{-\y/(\sadlines+1)}) .. controls ({-1.6+1.6*\y/(\sadlines+1)},0) .. (-1,{\y/(\sadlines+1)});
    \draw[gray!70] (1,{-\y/(\sadlines+1)}) .. controls ({1.6-1.6*\y/(\sadlines+1)},0) .. (1,{\y/(\sadlines+1)});
    }   \end{scope}
    \draw[gray, ->](-0.9,-1)--++(-0.3,0.3) node[midway, below left = -2pt]{\small $\vec{y}$};
    \draw[gray, ->](0.9,-1)--++(0.3,0.3) node[midway, below right = -2pt]{\small $\vec{y}$};
    \node[star,  star point ratio=2.25, fill=black, inner sep = 1pt] at (0,0) {};
    \node[circle, fill=black, inner sep = 1pt] at (0,-0.85) {};
    \node at (0,-1.1) {$x$};
\end{tikzpicture}
\right)
\simeq \Hom_\BB(x, \alpha)
    \quad \text{given by }\
        \begin{tikzpicture}[baseline = 0pt]
        \begin{scope}
        \clip (0,0) circle(1.2);
    \draw[gray!70] (1,1)--(-1,-1);
    \draw[gray!70] (1,-1)--(-1,1);
    \def\sadlines{5}
    \foreach \y in {3,...,5}{
    \draw[gray!70] ({-\y/(\sadlines+1)},1) .. controls (0,{1.6-1.6*\y/(\sadlines+1)}) .. ({\y/(\sadlines+1)},1);
    \draw[gray!70] ({-\y/(\sadlines+1)},-1) .. controls (0,{-1.6+1.6*\y/(\sadlines+1)}) .. ({\y/(\sadlines+1)},-1);
    \draw[gray!70] (-1,{-\y/(\sadlines+1)}) .. controls ({-1.6+1.6*\y/(\sadlines+1)},0) .. (-1,{\y/(\sadlines+1)});
    \draw[gray!70] (1,{-\y/(\sadlines+1)}) .. controls ({1.6-1.6*\y/(\sadlines+1)},0) .. (1,{\y/(\sadlines+1)});
    }   \end{scope}
    \draw[gray, ->](-0.9,-1)--++(-0.3,0.3) node[midway, below left = -2pt]{\small $\vec{y}$};
    \draw[gray, ->](0.9,-1)--++(0.3,0.3) node[midway, below right = -2pt]{\small $\vec{y}$};
    \node[star,  star point ratio=2.25, fill=black, inner sep = 1pt] at (0,0) {};
    \node[circle, fill=black, inner sep = 1pt] at (0,-0.85) {};
    \node at (0,-1.1) {$x$};
    \draw[very thick] (0,-0.85) -- (0,0)  node[pos = 0.1, sloped]{\small $>$}  node[pos = 0.8, sloped]{\small $>$} node[pos = 0.85, right]{$\alpha$};
    \node[rectangle, draw, thick, fill=white, minimum width = 0.9cm] at (0,-0.5) {\small $m$};
\end{tikzpicture}
\mapsfrom m \ .
    \end{equation*}
\end{proposition}
\begin{proof}
As above, we consider a progressive arc $\gamma$ joining the critical point to the boundary, displayed below, and make any string net $T \subseteq \D_1$ disjoint from $\gamma$. 
    \begin{equation*}
        \begin{tikzpicture}[baseline = 0pt]
        \begin{scope}
        \clip (0,0) circle(1.2);
    \draw[gray!70] (1,1)--(-1,-1);
    \draw[gray!70] (1,-1)--(-1,1);
    \def\sadlines{5}
    \foreach \y in {3,...,5}{
    \draw[gray!70] ({-\y/(\sadlines+1)},1) .. controls (0,{1.6-1.6*\y/(\sadlines+1)}) .. ({\y/(\sadlines+1)},1);
    \draw[gray!70] ({-\y/(\sadlines+1)},-1) .. controls (0,{-1.6+1.6*\y/(\sadlines+1)}) .. ({\y/(\sadlines+1)},-1);
    \draw[gray!70] (-1,{-\y/(\sadlines+1)}) .. controls ({-1.6+1.6*\y/(\sadlines+1)},0) .. (-1,{\y/(\sadlines+1)});
    \draw[gray!70] (1,{-\y/(\sadlines+1)}) .. controls ({1.6-1.6*\y/(\sadlines+1)},0) .. (1,{\y/(\sadlines+1)});
    }   \end{scope}
    \draw[gray, ->](-0.9,-1)--++(-0.3,0.3) node[midway, below left = -2pt]{\small $\vec{y}$};
    \draw[gray, ->](0.9,-1)--++(0.3,0.3) node[midway, below right = -2pt]{\small $\vec{y}$};
    \node[star,  star point ratio=2.25, fill=black, inner sep = 1pt] at (0,0) {};
    \node[circle, fill=black, inner sep = 1pt] at (0,-0.85) {};
    \node at (0,-1.1) {$x$};
    \draw[thick, red!50!black] (0,0) -- (0,0.85) node[pos = 0.5, sloped]{\small $<$} node[pos = 0.5, left]{$\gamma$};
\end{tikzpicture}
    \end{equation*}
    Again, let us denote $\SN_\II^\alpha(\D_1,\gamma;x)$ the space of string nets which are disjoint from $\gamma$ except at the critical point, where it is not tangent to $\gamma$, modulo isotopies preserving this property and skein relations away from $\gamma$. We need to prove both that the inclusion
    \begin{equation}
        \SN_\II^\alpha(\D_0,\gamma;x) \to \SN_\II^\alpha(\D_0;x)
    \end{equation}
    is an isomorphism, and that the provided map 
    $$\Hom_\BB(x,\alpha) \to \SN_\II^\alpha(\D_0,\gamma;x)$$ is an isomorphism.

    The first point is very similar to the proof above. The analogue of \eqref{eq:wrappingalphaaround} is
    \begin{equation}  
\begin{tikzpicture}[rotate = 180, scale = 0.5, baseline = 0cm]
    \begin{scope}[scale = 2]
    \clip (0,0) circle(1.2);
    \draw[gray!70] (1,1)--(-1,-1);
    \draw[gray!70] (1,-1)--(-1,1);
    \def\sadlines{5}
    \foreach \y in {3,...,5}{
    \draw[gray!70] ({-\y/(\sadlines+1)},1) .. controls (0,{1.6-1.6*\y/(\sadlines+1)}) .. ({\y/(\sadlines+1)},1);
    \draw[gray!70] ({-\y/(\sadlines+1)},-1) .. controls (0,{-1.6+1.6*\y/(\sadlines+1)}) .. ({\y/(\sadlines+1)},-1);
    \draw[gray!70] (-1,{-\y/(\sadlines+1)}) .. controls ({-1.6+1.6*\y/(\sadlines+1)},0) .. (-1,{\y/(\sadlines+1)});
    \draw[gray!70] (1,{-\y/(\sadlines+1)}) .. controls ({1.6-1.6*\y/(\sadlines+1)},0) .. (1,{\y/(\sadlines+1)});
    }   \end{scope}
    \node[star,  star point ratio=2.25, fill=black, inner sep = 1pt] at (0,0) {};
    \draw[thick] (0,-1.7) to[out=90,in=-75, looseness = 2] node[pos=0.5, sloped]{\footnotesize $<$} node[pos = 0.5, right]{$\alpha$} (0,0);
\end{tikzpicture}
\overset{\eqref{eq:saddle-rel}}\sim
\begin{tikzpicture}[rotate = 180, scale = 0.5, baseline = 0cm]
    \begin{scope}[scale = 2]
    \clip (0,0) circle(1.2);
    \draw[gray!70] (1,1)--(-1,-1);
    \draw[gray!70] (1,-1)--(-1,1);
    \def\sadlines{5}
    \foreach \y in {3,...,5}{
    \draw[gray!70] ({-\y/(\sadlines+1)},1) .. controls (0,{1.6-1.6*\y/(\sadlines+1)}) .. ({\y/(\sadlines+1)},1);
    \draw[gray!70] ({-\y/(\sadlines+1)},-1) .. controls (0,{-1.6+1.6*\y/(\sadlines+1)}) .. ({\y/(\sadlines+1)},-1);
    \draw[gray!70] (-1,{-\y/(\sadlines+1)}) .. controls ({-1.6+1.6*\y/(\sadlines+1)},0) .. (-1,{\y/(\sadlines+1)});
    \draw[gray!70] (1,{-\y/(\sadlines+1)}) .. controls ({1.6-1.6*\y/(\sadlines+1)},0) .. (1,{\y/(\sadlines+1)});
    }   \end{scope}
    \node[star,  star point ratio=2.25, fill=black, inner sep = 1pt] at (0,0) {};
    \draw[thick] (0,-1.7)to[out=90,in=-145] node[pos=0.5, sloped]{\footnotesize $>$} node[pos = 0.5, left]{$\alpha$} (1.3,0) node{\small $\bullet$} to[out = 160, in = -35]node[pos=0.5, sloped]{\footnotesize $>$}node[pos = 0.5,below left]{$\alpha^{-1}$} (0,1)node{\small $\bullet$} to[out = -145, in = 20]node[pos=0.5, sloped]{\footnotesize $<$} (-1.4,0)node{\small $\bullet$} to[out = -45, in = 145] node[pos=0.5, sloped]{\footnotesize $<$}(-0.3,-1.3)node{\small $\bullet$}  to[out = 80, in = -80] node[pos=0.5, sloped]{\footnotesize $>$}(0,0);
\end{tikzpicture}
\overset{\eqref{eq:skein-slide-1}}\sim
\begin{tikzpicture}[rotate = 180, scale = 0.5, baseline = 0cm]
    \begin{scope}[scale = 2]
    \clip (0,0) circle(1.2);
    \draw[gray!70] (1,1)--(-1,-1);
    \draw[gray!70] (1,-1)--(-1,1);
    \def\sadlines{5}
    \foreach \y in {3,...,5}{
    \draw[gray!70] ({-\y/(\sadlines+1)},1) .. controls (0,{1.6-1.6*\y/(\sadlines+1)}) .. ({\y/(\sadlines+1)},1);
    \draw[gray!70] ({-\y/(\sadlines+1)},-1) .. controls (0,{-1.6+1.6*\y/(\sadlines+1)}) .. ({\y/(\sadlines+1)},-1);
    \draw[gray!70] (-1,{-\y/(\sadlines+1)}) .. controls ({-1.6+1.6*\y/(\sadlines+1)},0) .. (-1,{\y/(\sadlines+1)});
    \draw[gray!70] (1,{-\y/(\sadlines+1)}) .. controls ({1.6-1.6*\y/(\sadlines+1)},0) .. (1,{\y/(\sadlines+1)});
    }   \end{scope}
    \node[star,  star point ratio=2.25, fill=black, inner sep = 1pt] at (0,0) {};
    \draw[thick] (0,-1.7) to[out=90,in=-145] node[pos=0.5, sloped]{\footnotesize $>$} node[pos = 0.5, left]{$\alpha$} (1.3,0) node{\small $\bullet$} to[out = 180, in = 25]node[pos = 0.78,rectangle, draw, fill=white, inner sep = 2.5pt, rotate = 12]{} (-0.5,-0.8)node{\small $\bullet$} to[out = 125, in = 10]node[pos=0.5, sloped]{\footnotesize $<$} (-1.4,0)node{\small $\bullet$} to[out = -45, in = 145] node[pos=0.5, sloped]{\footnotesize $<$}(-0.3,-1.3)node{\small $\bullet$}  to[out = 80, in = -80](0,0);
\end{tikzpicture}
\overset{\eqref{eq:otherTwist}}\sim
\begin{tikzpicture}[rotate = 180, scale = 0.5, baseline = 0cm]
    \begin{scope}[scale = 2]
    \clip (0,0) circle(1.2);
    \draw[gray!70] (1,1)--(-1,-1);
    \draw[gray!70] (1,-1)--(-1,1);
    \def\sadlines{5}
    \foreach \y in {3,...,5}{
    \draw[gray!70] ({-\y/(\sadlines+1)},1) .. controls (0,{1.6-1.6*\y/(\sadlines+1)}) .. ({\y/(\sadlines+1)},1);
    \draw[gray!70] ({-\y/(\sadlines+1)},-1) .. controls (0,{-1.6+1.6*\y/(\sadlines+1)}) .. ({\y/(\sadlines+1)},-1);
    \draw[gray!70] (-1,{-\y/(\sadlines+1)}) .. controls ({-1.6+1.6*\y/(\sadlines+1)},0) .. (-1,{\y/(\sadlines+1)});
    \draw[gray!70] (1,{-\y/(\sadlines+1)}) .. controls ({1.6-1.6*\y/(\sadlines+1)},0) .. (1,{\y/(\sadlines+1)});
    }   \end{scope}
    \node[star,  star point ratio=2.25, fill=black, inner sep = 1pt] at (0,0) {};
    \draw[thick] (0,-1.7)to[out=90,in=-110] node[pos=0.5, sloped]{\footnotesize $<$} node[pos = 0.5, right]{$\alpha$} (0,0);
\end{tikzpicture}
    \end{equation}

    For the second, we need to put any skein disjoint from $\gamma$ in standard position. Using \eqref{eq:saddle-rel}, we can make the $\alpha$ strand at the critical point come from the bottom as desired. Then, a skein disjoint from $\gamma$ can be retraced in the bottom region by a progressive isotopy as below:
    \begin{equation*}
        \begin{tikzpicture}[baseline = 0pt]
        \begin{scope}
        \clip (0,0) circle(1.2);
    \draw[gray!70] (1,1)--(-1,-1);
    \draw[gray!70] (1,-1)--(-1,1);
    \def\sadlines{5}
    \foreach \y in {3,...,5}{
    \draw[gray!70] ({-\y/(\sadlines+1)},1) .. controls (0,{1.6-1.6*\y/(\sadlines+1)}) .. ({\y/(\sadlines+1)},1);
    \draw[gray!70] ({-\y/(\sadlines+1)},-1) .. controls (0,{-1.6+1.6*\y/(\sadlines+1)}) .. ({\y/(\sadlines+1)},-1);
    \draw[gray!70] (-1,{-\y/(\sadlines+1)}) .. controls ({-1.6+1.6*\y/(\sadlines+1)},0) .. (-1,{\y/(\sadlines+1)});
    \draw[gray!70] (1,{-\y/(\sadlines+1)}) .. controls ({1.6-1.6*\y/(\sadlines+1)},0) .. (1,{\y/(\sadlines+1)});
    }   \end{scope}
    \draw[gray, ->](-0.9,-1)--++(-0.3,0.3) node[midway, below left = -2pt]{\small $\vec{y}$};
    \draw[gray, ->](0.9,-1)--++(0.3,0.3) node[midway, below right = -2pt]{\small $\vec{y}$};
    \node[star,  star point ratio=2.25, fill=black, inner sep = 1pt] at (0,0) {};
    \node[circle, fill=black, inner sep = 1pt] at (0,-0.85) {};
    \node at (0,-1.1) {$x$};
    \draw[thick, red!50!black] (0,0) -- (0,0.85) node[pos = 0.5, sloped]{\footnotesize $<$};
    \draw[very thick] (0,-0.85) -- (0,0) node[pos = 0.8, sloped]{\footnotesize $>$};
    \draw[very thick] (-0.27,0.65) to [out = -110, in = 15] node[pos = 0, sloped, rectangle, draw, thick, fill=white]{} node[pos = 0.5, sloped]{\footnotesize $<$}  (-0.67,0.05) node[rectangle, draw, thick, fill=white, minimum height = 0.5cm]{} to[out = -15, in = 90]node[pos = 0.5, sloped]{\footnotesize $<$} (-0.3,-0.5);
    \draw[very thick] (0.27,0.65) to [out = -70, in = 165] node[pos = 0, sloped, rectangle, draw, thick, fill=white]{} node[pos = 0.5, sloped]{\footnotesize $>$}  (0.67,0.05) node[rectangle, draw, thick, fill=white, minimum height = 0.5cm]{} to[out = -165, in = 90]node[pos = 0.5, sloped]{\footnotesize $>$} (0.3,-0.5);
    \node[rectangle, draw, thick, fill=white, minimum width = 0.9cm] at (0,-0.57) {};
\end{tikzpicture}
\ \sim \
\begin{tikzpicture}[baseline = 0pt]
        \begin{scope}
        \clip (0,0) circle(1.2);
    \draw[gray!70] (1,1)--(-1,-1);
    \draw[gray!70] (1,-1)--(-1,1);
    \def\sadlines{5}
    \foreach \y in {3,...,5}{
    \draw[gray!70] ({-\y/(\sadlines+1)},1) .. controls (0,{1.6-1.6*\y/(\sadlines+1)}) .. ({\y/(\sadlines+1)},1);
    \draw[gray!70] ({-\y/(\sadlines+1)},-1) .. controls (0,{-1.6+1.6*\y/(\sadlines+1)}) .. ({\y/(\sadlines+1)},-1);
    \draw[gray!70] (-1,{-\y/(\sadlines+1)}) .. controls ({-1.6+1.6*\y/(\sadlines+1)},0) .. (-1,{\y/(\sadlines+1)});
    \draw[gray!70] (1,{-\y/(\sadlines+1)}) .. controls ({1.6-1.6*\y/(\sadlines+1)},0) .. (1,{\y/(\sadlines+1)});
    }   \end{scope}
    \draw[gray, ->](-0.9,-1)--++(-0.3,0.3) node[midway, below left = -2pt]{\small $\vec{y}$};
    \draw[gray, ->](0.9,-1)--++(0.3,0.3) node[midway, below right = -2pt]{\small $\vec{y}$};
    \node[star,  star point ratio=2.25, fill=black, inner sep = 1pt] at (0,0) {};
    \node[circle, fill=black, inner sep = 1pt] at (0,-0.85) {};
    \node at (0,-1.1) {$x$};
    \draw[very thick] (0,-0.85) -- (0,0)  node[pos = 0.1, sloped]{\footnotesize $>$}  node[pos = 0.8, sloped]{\footnotesize $>$} node[pos = 0.85, right]{$\alpha$};
    \node[rectangle, draw, thick, fill=white, minimum width = 0.9cm] at (0,-0.5) {\small $m$};
\end{tikzpicture} \quad \text{where} \quad m = 
\begin{tikzpicture}[baseline = 0pt]
    \node[circle, fill=black, inner sep = 1pt] at (0,-0.85) {};
    \node at (0,-1.1) {$x$};
    \draw[very thick] (0,-0.85) -- (0,0.85)  node[pos = 0.1, sloped]{\footnotesize $>$}  node[pos = 0.9, sloped]{\footnotesize $>$} node[pos = 0.95, right]{$\alpha$};
    \draw[very thick] (-0.2,0) to [out = 90, in = -70] node[pos = 0, sloped, rectangle, draw, thick, fill=white]{} node[pos = 0.5, sloped]{\footnotesize $<$}  (-0.4,0.45) node[rectangle, draw, thick, fill=white, minimum width = 0.5cm]{} to[out = -130, in = 100]node[pos = 0.5, sloped]{\footnotesize $<$} (-0.4,-0.4);
    \draw[very thick] (0.2,0) to [out = 90, in = -110] node[pos = 0, sloped, rectangle, draw, thick, fill=white]{} node[pos = 0.5, sloped]{\footnotesize $>$}  (0.4,0.45) node[rectangle, draw, thick, fill=white, minimum width = 0.5cm]{} to[out = -50, in = 80]node[pos = 0.5, sloped]{\footnotesize $>$} (0.4,-0.4);
    \node[rectangle, draw, thick, fill=white, minimum width = 1.1cm] at (0,-0.4) {};
\end{tikzpicture} \ .
    \end{equation*}
\end{proof}
\begin{proposition}\label{prop:SNonD2}
For a foliated disk $\mathbb D_2$ with a single singularity of index 2, and for any boundary label $x \in \II$, there is an isomorphism of vector spaces
    \begin{equation*}
\SN_\II^\alpha\left( \mathbb D_2 :=
    \begin{tikzpicture}[scale = 0.5, baseline = 0pt]
    \foreach \y in {0,...,6}{
    \draw[gray!70] (0,0) circle(\y/3);
    }
    \draw[gray, <-] (45:2.2) -- (45:2.8) node[midway,below right = -2pt]{\small $\vec{y}$};
    \node[circle, fill=black, inner sep = 1pt] at (0,-2) {};
    \node at (0,-2.3) {$x$};
    \node[star,  star point ratio=2.25, fill=black, inner sep = 1pt] at (0,0) {};
    \end{tikzpicture} \right) \simeq \Hom_\BB(x, \alpha^{-1})
    \quad\text{    given by    }\quad
    \begin{tikzpicture}[scale = 0.5, baseline = 0pt]
    \foreach \y in {0,...,6}{
    \draw[gray!70] (0,0) circle(\y/3);
    }
    \node[star,  star point ratio=2.25, fill=black, inner sep = 1pt] at (0,0) {};
    \node[circle, fill=black, inner sep = 1pt] (X) at (0,-2) {};
    \node at (0,-2.3) {$x$};
    \node[rectangle, draw=black, fill=white, minimum width = 20pt, inner sep = 3pt] (F) at (0,-1) {$m$};
    \draw (0,0)--(F) node[midway, sloped]{\tiny $<$} node[pos = 0.3, left]{\small$\alpha^{-1}$} --(X)node[midway, sloped]{\tiny $<$};
    \end{tikzpicture}  \mapsfrom  m \ .
    \end{equation*}
\end{proposition}
\begin{proof}
    The proof is similar to Proposition \ref{prop:SNonD0} noticing that the mate of the twisted pivotal structure satisfies analogues of (\ref{eq:pivNatGraphically} -- \ref{eq:pivNatEval}).
\end{proof}
\begin{remark}
We expressed the results above for a single boundary label $x$ to avoid heavy formulas. The general result can be obtained by stacking with a string net in a collar of the boundary relating any set of boundary labels to a single point.
\end{remark}

\subsection{Excision}
To compute twisted string nets on more complex surfaces we will use the following gluing result.
\begin{proposition}\label{prop:excision}
    Let $\Sigma = \Sigma_1 \underset{\Gamma}{\cup} \Sigma_2$ be a decomposition of a foliated surface $\Sigma$ along a leaf $\Gamma$ disjoint from the singular locus. Let $\II\subseteq \BB$ be a tensor ideal in a twisted-pivotal category. Then for any boundary labels $x_1 \subseteq \partial\Sigma\cap \Sigma_1 ,\ x_2 \subseteq \partial\Sigma\cap \Sigma_2$, gluing string nets along $\Gamma$ induces an isomorphism of vector spaces
    \begin{equation}
\int^{x \in \SN_\II(\Gamma)} \SN_\II^\alpha(\Sigma_1; x_1,x) \otimes \SN_\II^\alpha(\Sigma_2; x, x_2) 
\ \simeq\
\SN_\II^\alpha(\Sigma;x_1,x_2) 
    \end{equation}
\end{proposition}
Explicitly, the coend of the LHS above is the quotient of the vector space $$\bigoplus_{x\in\SN_\II(\Gamma)}\SN_\II^\alpha(\Sigma_1; x_1,x) \otimes \SN_\II^\alpha(\Sigma_2; x, x_2) $$
by the relation generated by 
$$(T,R\cdot S) \sim (R\cdot T, S)\, ,\quad T\in\SN_\II^\alpha(\Sigma_1; x_1,x),\ S\in \SN_\II^\alpha(\Sigma_2; y, x_2),\ R\in \SN_\II(\Gamma\times I;x,y)$$
and $-\cdot -$ is the action of $\SN_\II(\Gamma)$ on the string net modules from \eqref{eq:twistedSNfunctor}.

    The proof is similar to \cite{WalkerNotes, BHskcat, RunkelSchweigertThamExciAdmSkeins}. We include it for completeness and because the presence of foliations modify some arguments. The main difference is in surjectivity of $\iota$ below.
\begin{proof}
The gluing along $\Gamma$ factors as
    $$\int^{x \in \SN_\II(\Gamma)} \SN_\II^\alpha(\Sigma_1; x_1,x) \otimes \SN_\II^\alpha(\Sigma_2; x, x_2) 
\overset {gl_\Gamma}\to
\SN_\II^{\alpha,\Gamma}(\Sigma;x_1,x_2) \overset\iota\to\SN_\II^{\alpha}(\Sigma;x_1,x_2) $$
where $\SN_\II^{\alpha,\Gamma}(\Sigma,x_1,x_2)$ is the space of twisted string nets in $\Sigma$ which intersect every connected component of $\Gamma$ along at least one strand colored by an object in $\II$, modulo isotopies and skein relations preserving this property, as in \cite[Lemma 2.11]{BHskcat}. In other words, the intersection with $\Gamma$ is admissible. We prove that both of these maps are isomorphisms.

\textbf{Surjectivity of $gl_\Gamma$:}
A string net $T \in \SN_\II^{\alpha,\Gamma}(\Sigma,x_1,x_2)$ must be transverse to the foliation, hence intersects the leaf $\Gamma$ transversely but for the fact that vertices of $T$ must be made disjoint from $\Gamma$, which is always possible via a small isotopy. By assumption, $T$ intersects $\Gamma$ along an admissible labelling $x \in \SN_\II(\Gamma)$, hence $T = gl_\Gamma((T\cap\Sigma_1)\otimes (T\cap\Sigma_2))$.

\textbf{Injectivity of $gl_\Gamma$:}
An isotopy of string nets decompose as isotopies happening in small disks. The coend relation realizes one specific isotopy pushing a collar of $\Gamma$ from $\Sigma_1$ to $\Sigma_2$. Up to conjugation by this isotopy, any isotopy in a sufficiently small disk can be made supported away from $\Gamma$. Skein relations can be pushed away from $\Gamma$ up to isotopy, and sliding relations happen away from $\Gamma$. 

\textbf{Surjectivity of $\iota$:}
Let $T\in \SN_\II^{\alpha}(\Sigma,x_1,x_2)$. We need to show that we can make $T$ intersect $\Gamma$ along an $\II$-colored edge. Because $T$ is admissible, it contains at least one $\II$-colored edge in the connected component of $\Gamma$. Pick a generic path $\gamma$ going from $\Gamma$ to this edge. We want to pull the $\II$-colored strand (along with every other strand intersecting $\gamma$) along $\gamma$ to make it intersect $\Gamma$ as in \cite[Lemma 2.11]{BHskcat}. This is unfortunately not an isotopy among progressive string nets if $\gamma$ is not progressive. However, it can be solved by introducing an appropriate number of evaluation and coevaluation coupons. A generic $\gamma$ is (positively or negatively) progressive except at finitely many points where it is tangent to the foliation:
\begin{equation*}
    \begin{tikzpicture}[scale = 1.2, baseline = 0pt]
        \clip (-1.1,0.7) rectangle (-0.1,-0.7);
    \foreach \y in {0,...,10}{
    \draw[gray!70] (0,0) circle(\y/5);
    }
    \draw[thick, red!50!black] (-0.6,1) --++ (0,-2) node[pos = 0.7, sloped]{\small $>$} node[pos = 0.7, left]{$\gamma$};
    \end{tikzpicture} 
\end{equation*}
Dragging an edge of $T$ along $\gamma$ is then given by:
\begin{equation*}
    \begin{tikzpicture}[scale = 1.2, baseline = -5pt]
        \clip (-1.1,0.7) rectangle (-0.1,-1);
    \foreach \y in {0,...,10}{
    \draw[gray!70] (0,0) circle(\y/5);
    }
    \draw[thick, red!50!black, opacity = 0.5] (-0.6,1) --++ (0,-1.8) node[pos = 0.7, sloped]{\small $>$} node[pos = 0.7, left]{$\gamma$};
    \draw[thick] (-1.1,-0.8) --++ (1,0) node[pos = 0.7, sloped]{\small $>$} node[pos = 0.7, above]{$x$};
    \end{tikzpicture} 
\quad = \quad
    \begin{tikzpicture}[scale = 1.2, baseline = -5pt]
        \clip (-1.1,0.7) rectangle (-0.1,-1);
    \foreach \y in {0,...,10}{
    \draw[gray!70] (0,0) circle(\y/5);
    }
    \draw[thick, red!50!black, opacity = 0.5] (-0.6,1) --++ (0,-1.8) node[pos = 0.7, sloped]{\small $>$} node[pos = 0.7, left]{$\gamma$};
    \draw[thick] (-1.1,-0.8) to[out = 0, in = -110] (-0.5, -0.4) node{\small $\bullet$} to [out = -90, in = 65] (-0.6, -0.9) node{\small $\bullet$} to[out = 10, in = 180]node[pos = 0.7, sloped]{\small $>$} (-0.1,-0.8) ;
    \end{tikzpicture} 
\quad = \quad
    \begin{tikzpicture}[scale = 1.2, baseline = -5pt]
        \clip (-1.1,0.7) rectangle (-0.1,-1);
    \foreach \y in {0,...,10}{
    \draw[gray!70] (0,0) circle(\y/5);
    }
    \draw[thick, red!50!black, opacity = 0.5] (-0.6,1) --++ (0,-1.8) node[pos = 0.7, sloped]{\small $>$} node[pos = 0.7, left]{$\gamma$};
    \draw[thick] (-1.1,-0.8) to[out = 0, in = -130] (-0.55, -0.1) node{\small $\bullet$} to [out = 130, in = -90] (-0.7,0.6)node{\small $\bullet$} to [out = -60, in = 130] (-0.4,-0.1)node{\small $\bullet$} to [out = -120, in = 65] (-0.6, -0.9) node{\small $\bullet$} to[out = 10, in = 180]node[pos = 0.7, sloped]{\small $>$} (-0.1,-0.8) ;
    \end{tikzpicture} 
\end{equation*}
    Applying this procedure for every edge of $T$ intersecting $\gamma$, we obtain an equivalent string net $T'$ modified in a neighborhood of $\gamma$ which does intersect $\Gamma$ admissibly.
    
\textbf{Injectivity of $\iota$:} This follows from the arguments of \cite[Lemma 2.11]{BHskcat}, the main idea being that we can find another path $\gamma'$ disjoint from $\gamma$ satisfying the same requirements and use either $\gamma$ or $\gamma'$ to make $T$ intersect $\Gamma$ admissibly.
\end{proof}


\subsection{Twisted string nets as a foliated categorified TQFT}\label{sec:foliatedTQFT}
So far, we have introduced an assignment of algebraic data---the twisted string net categories and modules---to oriented 1-manifolds and to smooth foliated surfaces. We will now explain in what sense they form a TQFT, albeit one that depends on a foliation. 

\label{sec:or_tqfts}
We begin with a brief reminder of the usual notions in the oriented setting. 
\begin{definition}[Oriented cobordisms]
Given two smooth, closed, oriented $1$-manifolds $\Gamma_{in}$ and $\Gamma_{out}$, an \emph{oriented cobordism} $\Sigma\colon \Gamma_{in}\to \Gamma_{out}$ from $\Gamma_{in}$ to $\Gamma_{out}$ is a compact oriented surface with boundary $\Sigma$ together with orientation-preserving boundary identifications $-\Gamma_{in}\sqcup \Gamma_{out} \xrightarrow{\cong}  \partial \Sigma$
where $\partial\Sigma$ is equipped with the induced {outward-normal-first} orientation.
\end{definition}

\begin{definition}
    The \emph{$2$-dimensional oriented bordism $(2,1)$-category} $\on{Bord}_{12\sim}^{\on{or}}$ is defined as follows:
    \begin{description}
        \item[Objects] closed oriented $1$-manifolds $\Gamma$
        \item[1-morphisms] oriented surface bordisms $\Sigma:\Gamma_{in}\to \Gamma_{out}$ (with boundary collars)
        \item[2-morphisms] diffeomorphisms preserving the boundary collars, up to isotopy of such.
    \end{description}
    It has the usual composition, associators and unitors, and it is symmetric monoidal under disjoint union, see e.g. \cite{SPPhD, Filippos, HaiounHandle}.
\end{definition}
\begin{definition}
    An \emph{oriented categorified $2$-TQFT}\footnote{These also appear in the literature as $(1+1+\epsilon)$-TQFTs or topological modular functors.} is a symmetric monoidal functor
    \[
        \mathcal{Z}:\on{Bord}_{12\sim}^{\on{or}} \to \mathcal{C},
    \]
    for some target symmetric monoidal $(2,1)$-category $\mathcal{C}$.
\end{definition}
Showing that a given assignment defines an oriented categorified TQFT can, in general, be hard work---especially when the assignment depends \emph{a priori} on extra geometric data on bordisms.
This is the case for twisted string net modules, which depend on a choice of Morse singular foliation.
We undertake the task of removing this dependence and constructing an oriented TQFT in full detail in Section~5.


We now refine the oriented bordism $(2,1)$-category by equipping each surface bordism with Morse-singular foliation data and restricting $2$-morphisms accordingly.
\begin{definition}\label{def:Bord^Morse_fol}
    The \emph{$2$-dimensional Morse-foliated bordism $(2,1)$-category} $\Bord_{12\sim}^{\on{Morse\ fol}}$ is defined as follows:
    \begin{description}
        \item[Objects] closed oriented 1-manifolds $\Gamma$
        \item[1-morphisms] oriented surface bordisms $\Sigma:\Gamma_{in}\to \Gamma_{out}$ equipped with:
        \begin{itemize}
            \item a Morse singular foliation such that $\Gamma_{in}$ and $\Gamma_{out}$ are unions of leaves, and the positive transverse direction is incoming along $\Gamma_{in}$ and outgoing along $\Gamma_{out}$;
            \item progressive boundary collars $\Gamma_{in}\times [0,1) \hookrightarrow \Sigma$ and $\Gamma_{in}\times (-1,0] \hookrightarrow \Sigma$
        \end{itemize}
        \item[2-morphisms] progressive diffeomorphisms preserving the boundary collars, up to isotopy of such.
    \end{description}
It has the usual composition, associators and unitors, and it is symmetric monoidal under disjoint union.

\end{definition}
\begin{definition}
    A \emph{Morse-foliated categorified $2$-TQFT} is a symmetric monoidal functor
    \[
        \mathcal{Z}:\Bord_{12\sim}^{\on{Morse\ fol}}\to \mathcal{C},
    \]
    for some target symmetric monoidal $(2,1)$-category $\mathcal{C}$.
\end{definition}

We finally turn to showing that twisted string net modules form a Morse-foliated 2d TQFT. As we will see, excision is the main ingredient here. But first, we introduce the target symmetric monoidal bicategory.

\begin{definition}
    Let $\Bimod$ be the bicategory with:
    \begin{description}
        \item[Objects] small $\Bbbk$-linear categories $\CC$
        \item[1-morphisms] bimodules $F:\CC\otimes\mathcal{D}^{op}\to \Vect$, with composition of 1-morphisms given by the coend $$(G\circ F)(C,E) := \int^{D\in\mathcal{D}} F(C,D)\otimes G(D,E)\ .$$
        \item[2-morphisms] natural transformations.
    \end{description}
    We write $\Bimod^{hop}$ the opposite bicategory in the direction of 1-morphisms.
\end{definition}

\begin{theorem}\label{thm:foliated_tqft}
    Let $\II\subseteq\BB$ be a tensor ideal in a twisted-pivotal category. Then twisted string net modules define a Morse-foliated categorified 2-TQFT
    \begin{equation}
        \SN_\II^\alpha:\Bord_{12\sim}^{\on{Morse\ fol}}\to \Bimod^{hop}
    \end{equation}
\end{theorem}
\begin{proof}
    We have seen in \eqref{eq:SNcobordismBimodule} that a foliated cobordism $\Sigma:\Gamma_{in}\to\Gamma_{out}$ defines a bimodule from $\SN_\II^\alpha(\Gamma_{out})$ to $\SN_\II^\alpha(\Gamma_{in})$, and it is clear that transporting string nets under a progressive diffeomorphism preserves this bimodule structure. By Proposition \ref{prop:excision}, this assignment preserves composition. 

    Taking disjoint union of graphs and labellings provide canonical isomorphisms 
    $$\SN_\II^\alpha(\Sigma;X)\otimes \SN_\II^\alpha(\Sigma';X') \simeq \SN_\II^\alpha(\Sigma\sqcup\Sigma';X\sqcup X')$$
    which induces a symmetric monoidal structure on $\SN_\II^\alpha$.
\end{proof}

\begin{remark}
As explained in Section \ref{sec:rigidGraphicalCalculus} the admissible string net modules for foliated surfaces without singularity are defined for any tensor ideal in a monoidal category, without requiring a twisted pivotal structure. The reader should expect that there is an analogous (2,1)-category of foliated surfaces without singularity and that these string net modules form a foliated categorified 2-TQFT. This is indeed true if $\BB$ is rigid, or if $\II=\BB$. However, the proof of excision does not hold if neither of these assumption is made. A counter-example, in the context of \cite{BHskcat}, was given to us by Theo Johnson-Freyd.
\end{remark}


\section{Twisted string nets as an oriented categorified TQFT} \label{sec:TQFT}
\sloppy

In the previous section, we constructed a Morse-foliated categorified $2$-dimensional TQFT 
$\SN_{\II}^{\alpha}$ from the data of a tensor ideal $\II \subset \BB$ in a twisted-pivotal 
category $\BB$. The goal of this section is to show that the same data also determine an 
\emph{oriented} categorified $2$-dimensional TQFT. To do so, it will be convenient to work 
with a $(2,1)$-category of oriented bordisms that is equivalent to $\on{Bord}^{\on{or}}_{12\sim}$ 
but whose $1$-morphisms come equipped with extra structure inducing a Morse singular foliation. 
A natural candidate for this role is the $(2,1)$-category $\on{Bord}^{\on{Cerf,\ or}}_{12\sim}$ 
introduced in \cite{Filippos}, which we now recall schematically:
\begin{description}
    \item[Objects] closed oriented $1$-dimensional manifolds;
    \item[$1$-morphisms] oriented $2$-dimensional bordisms equipped with excellent Morse 
    functions, where a Morse function is called \emph{excellent} if no two of its critical 
    points share the same critical value;
    \item[$2$-morphisms] equivalence classes of mapping cylinders equipped with generic paths 
    between Morse functions. Here, a \emph{generic path} between two Morse functions $f_0$ 
    and $f_1$ on a bordism $\Sigma$ is a smooth $1$-parameter family $\{f_t\}_{t \in [0,1]}$ 
    which is Morse except at finitely many parameter values, where it undergoes an elementary 
    Cerf move --- either a birth or death of a pair of critical points, or a crossing of two 
    critical values. Two such generic paths are declared equivalent if they are connected by a generic
    homotopy rel endpoints (which is always the case).
\end{description}
In \cite{Filippos}, it is shown that the forgetful functor
\[
    F^{\on{Cerf}}: \on{Bord}^{\on{Cerf,\ or}}_{12\sim} \rightarrow \on{Bord}^{\on{or}}_{12\sim}
\]
is an equivalence and, moreover, that $\on{Bord}^{\on{Cerf,\ or}}_{12\sim}$ admits an explicit 
presentation by generators and relations, which we denote by $\on{Sur}_{12\sim}^{\on{or}}$. 
We will not recall the definition of $\on{Bord}^{\on{Cerf,\ or}}_{12\sim}$ in any more detail here, referring 
the interested reader to Chapter~3 of \cite{Filippos} instead; what we will carefully recall later
in this section is precisely its presentation 
$\on{Sur}_{12\sim}^{\on{or}}$. Having such a presentation at hand greatly helps with the 
construction of symmetric monoidal functors out of $\on{Bord}^{\on{Cerf,\ or}}_{12\sim}$ and, 
thus, by extension, oriented categorified $2$-dimensional TQFTs. We can now state the main 
theorem of this section:

\begin{theorem}\label{thm:main_thm}
    Let $\II$ be a tensor ideal in an $\alpha$-twisted-pivotal category $\BB$. There exists 
    an oriented categorified $2$-dimensional TQFT
    \[
        \orSN_{\II}^{\alpha}\colon \Bord_{12\sim}^{\on{or}} \to \Bimod^{hop}
    \]
    such that the composite 
    \[
        \orSN_{\II}^{\alpha} \circ F^{\on{Cerf}}\colon \Bord_{12\sim}^{\on{Cerf,\ or}} \to \Bimod^{hop}
    \]
    associates to a closed oriented $1$-dimensional manifold $\Gamma$ the twisted string net 
    category $\SN_{\II}^{\alpha}(\Gamma)$, and to a $2$-dimensional bordism equipped with a 
    Morse function $(\Sigma, f)\colon \Gamma_{in} \rightarrow \Gamma_{out}$ the twisted string 
    net module $\SN_{\II}^{\alpha}(\Sigma; \Gamma_{in} \cup \Gamma_{out})$ where $\Sigma$ has foliation induced by $f$.
\end{theorem}

\begin{remark}
Recall that, by Definition~\ref{def:twsitSNfunctor}, in order to define the twisted string net modules of an oriented surface $\Sigma$, we first need to equip $\Sigma$ with a Morse singular foliation. The bordisms appearing in $\on{Bord}^{\on{Cerf,\ or}}_{12\sim}$ are equipped with Morse functions, and it is their induced Morse singular foliations that we work with; however, not every Morse singular foliation arises in this way. What the theorem above establishes is that the twisted string net modules associated to an oriented surface $\Sigma$ equipped with a Morse singular foliation induced by a Morse function is independent, up to unique isomorphism, of the choice of such foliation. It does not establish the stronger claim that the twisted string net module is independent of an arbitrary choice of Morse singular foliation on $\Sigma$. Whether this stronger statement holds is, at present, unknown to the authors, but it is not necessary for the construction of our TQFT.
\end{remark}

We will recall the generators of $\on{Sur}_{12\sim}^{\on{or}}$, together with all the necessary Morse and Cerf theory needed to understand them, in Section~\ref{subsec:Morse_data+sur}. We will then describe what $\orSN_{\II}^{\alpha}$ assigns to the generating objects, $1$- and $2$-morphisms in Section~\ref{subsec:generators}. Finally, in Section~\ref{subsec:relations+result}, we will recall the relations of the presentation $\on{Sur}_{12\sim}^{\on{or}}$ and verify that the proposed assignment $\orSN_{\II}^{\alpha}$ satisfies them, thus completing the proof of Theorem~\ref{thm:main_thm}.


\subsection{The surgery presentation }\label{subsec:Morse_data+sur}


The purpose of this section is to introduce the surgery presentation $\on{Sur}_{12\sim}^{\on{or}}$. This presentation generates a $(2,1)$-category $\on{FSur}_{12\sim}^{\on{or}}$, which admits a symmetric monoidal functor
\[
    \BRF\colon \on{FSur}_{12\sim}^{\on{or}} \rightarrow \on{Bord}^{\on{or}}_{12\sim}
\]
that is an equivalence of symmetric monoidal bicategories. We will begin by recalling what a presentation of a symmetric monoidal bicategory is; we will then give a description of $\on{Sur}_{12\sim}^{\on{or}}$ and of some aspects of the functor $\BRF$. We emphasize that this section contains no new results; it is purely expository, recalling the definitions and theorems needed for our construction of $\orSN_{\II}^{\alpha}$ and referring the reader to~\cite{Filippos} and~\cite{SPPhD} for proofs and further details.

\subsubsection{Presentations of symmetric monoidal bicategories}

Constructing TQFTs often proceeds in analogy with constructing representations of a group $G$: one identifies a presentation of $G$ by generators and relations, and then describes representations by assigning each generator of $G$ to a morphism in $\on{Vect}_k$ subject to the relations. Extending this analogy, a natural route to constructing TQFTs is to work with a generators-and-relations presentation of the relevant bordism category. In our case, this is $\on{Bord}_{12\sim}^{\on{or}}$, which is a symmetric monoidal $(2,1)$-category, and we would like a presentation compatible with this extra structure. 

Before giving the precise definition, let us first describe what we would like a presentation of a symmetric monoidal bicategory to be. In analogy with presentations of groups, a presentation $P$ should consist of generating objects $G_0$, generating $1$-morphisms $G_1$, generating $2$-morphisms $G_2$, and generating relations $\mathcal{R}$, and should generate a symmetric monoidal bicategory $FP$. We say that $P$ is a presentation of a 
symmetric monoidal bicategory $\mathcal{M}$ if there exists a symmetric monoidal equivalence 
$FP \xrightarrow{\sim} \mathcal{M}$. Moreover, given a presentation $P = (G_0, G_1, G_2; 
\mathcal{R})$ of $\mathcal{M}$ and a symmetric monoidal bicategory $\mathcal{C}$, an 
assignment
\[
    G_0 \rightarrow \on{Ob}(\mathcal{C}), \quad 
    G_1 \rightarrow 1\text{-}\on{Mor}(\mathcal{C}), \quad 
    G_2 \rightarrow 2\text{-}\on{Mor}(\mathcal{C})
\]
satisfying the relations $\mathcal{R}$ 
should uniquely determine a symmetric monoidal functor $\mathcal{M} \rightarrow \mathcal{C}$ 
up to natural isomorphism.

Making this precise is 
nontrivial --- even the definition of a symmetric monoidal bicategory is already a formidable undertaking. Fortunately, both of these issues have been addressed in detail by Schommer-Pries in~\cite{SPPhD}, and the definition of $\on{Sur}_{12\sim}^{\on{or}}$ follows 
the technical framework developed there. We provide an informal summary of this framework below, following Chapter~2 of~\cite{SPPhD}. In the language of Section~2.12.3 of~\cite{SPPhD}, $\on{Sur}_{12\sim}^{\on{or}}$ is an \emph{unbiased semistrict symmetric monoidal 
$3$-computad}, which consists of the 
following data:
\begin{description}
    \item[Generating objects] A set $G_0$. The objects of the generated category are the words $W(G_0)$ in $G_0$, that is, formal tensor products of elements of $G_0$.
    \item[Generating $1$-morphisms] A set $G_1$ together with source and target maps $s_1, t_1\colon G_1 \rightarrow W(G_0)$. The $1$-morphisms of the generated category are the sentences $S(G_1, s_1, t_1)$, that is, all $1$-morphisms that can be formally built from the elements of $G_1$ using composition, the monoidal structure and the symmetric braiding. 
    \item[Generating $2$-morphisms] A set $G_2$ together with source and target maps $s_2, t_2\colon G_2 \rightarrow S(G_1, s_1, t_1)$, required to satisfy the globularity condition
    \[
    s_1(s_2(e)) = s_1(t_2(e)) \quad \text{and} \quad t_1(s_2(e)) = t_1(t_2(e))
    \]
    for all $e\in G_2$. The $2$-morphisms of the generated category are the paragraphs $P(G_2, s_2, t_2)$, that is, all $2$-morphisms that can be formally built from the elements of $G_2$ using horizontal and vertical composition, and all the structure of a symmetric monoidal bicategory.
    \item[Generating relations] A set $\mathcal{R}$ of pairs $(\alpha, \beta)$ in $P(G_2, s_2, t_2)$ such that $\alpha$ and $\beta$ have the same source and target. This determines an equivalence relation—also denoted $\mathcal{R}$—on $P(G_2, s_2, t_2)$, defined as the finest equivalence relation containing $R$ and respecting the symmetric monoidal structure.
\end{description}

Since this is the only notion of presentation that appears here, we will simply refer to unbiased semistrict symmetric monoidal $3$-computads as \emph{presentations}. A presentation  $P= (G_0, G_1, G_2; \mathcal{R})$ generates a symmetric monoidal bicategory $FP$,\footnote{More precisely, $FP$ is an \emph{unbiased semistrict symmetric monoidal $2$-category}, which is a stricter version of symmetric monoidal bicategory, but we are not using the precise structure of $FP$ in this paper.} as per Definition~2.75 in~\cite{SPPhD}. Informally, it is the category generated by the prescribed objects, 1 and 2-morphisms together with their formal tensor products and compositions modulo the prescribed relations and all relations necessary to define a symmetric monoidal bicategory. A presentation $P$ also provides us with a notion of $P$-shaped data as follows:
\begin{definition}
    Given a presentation $P= (G_0, G_1, G_2; \mathcal{R})$ and a symmetric monoidal bicategory $\mathcal{C}$, a $P$-shaped datum in $\mathcal{C}$ is an assignment 
    \[
    G_0 \rightarrow \on{Ob}(\mathcal{C}), \quad 
    G_1 \rightarrow 1\text{-}\on{Mor}(\mathcal{C}), \quad 
    G_2 \rightarrow 2\text{-}\on{Mor}(\mathcal{C})
    \]
    that is compatible with the source and target maps of $P$ and satisfy the relations $\mathcal{R}$.
\end{definition}
$P$-shaped data correspond to symmetric monoidal functors out of $F(P)$. In particular, the main technical point we need about presentations is the following:
\begin{proposition}
    Given a presentation $P = (G_0, G_1, G_2; \mathcal{R})$ and a symmetric monoidal bicategory $\mathcal{C}$, a $P$-shaped datum in $\mathcal{C}$ uniquely determines a symmetric monoidal functor $FP\rightarrow \mathcal{C}$ up to natural isomorphism. 
\end{proposition}
\begin{proof}
    This is a corollary of Proposition~2.77 and Proposition~2.74 in~\cite{SPPhD}, applied in the context of unbiased semistrict symmetric monoidal $3$-computads, as discussed in Section~2.12.3 and Section~2.13 in~\cite{SPPhD}.
\end{proof}

\subsubsection{The presentation $\on{Sur}_{12\sim}^{\on{or}}$}
We will now describe the generators of the surgery presentation $\on{Sur}_{12\sim}^{\on{or}}$. 
Each of the generating $1$-morphisms represents a certain bordism and each of the generating $2$-morphisms represents a certain isotopy class of diffeomorphisms. We will make this correspondence more explicit when we describe the bordism realization functor in Section~\ref{subsubsec:bord_realization}.

\begin{definition}
    The surgery presentation $\on{Sur}^{\on{or}}_{12\sim}$ is defined by:
    \begin{description}
        \item[Generating objects] Closed, connected, oriented $1$-dimensional manifolds. Words 
        in the set of generating objects can therefore be canonically identified with closed, 
        oriented $1$-dimensional manifolds (with an order on their $\pi_0$) and will typically be denoted by capital Greek letters 
        ($A, B, \Gamma, \ldots$).
        
        \item[Generating $1$-morphisms] There are two types of generating $1$-morphisms:
        
        \begin{table}[H]
          \centering
          \setlength\extrarowheight{2pt} 
          \begin{tabular}{|c|c|}
            \hline
            \textbf{\textcolor{blue}{Diffeomorphism type}} &
            \textbf{\textcolor{red}{Surgery type}} \\
            \hline
        
            \begin{tikzpicture}
              \node at (0, 0.75) {};
              \begin{scope}[yshift = -1cm] 
        
                \begin{scope}[
                    baseline = {([yshift=-0.5ex]current bounding box.center)},
                    scale     = 0.5  
                  ]
                  \draw[ultra thick] (0,0) -- (0,3); 
                  \fill[blue] (0,1.5) circle (1.5mm);  
                  \node[left, blue] at (-0.2, 1.5) {$\phi$};  
                  \node[above left] at (-0.1, 0)   {$A$};     
                  \node[below left] at (-0.1, 3)   {$B$};     
                \end{scope}
        
                \begin{scope}[xshift = 2.75cm, yshift = 1cm]
                  \node at (0,    0) {where $\phi: A \rightarrow B$};
                  \node at (0, -0.5) {is a diffeomorphism};
                \end{scope}
        
              \end{scope}
            \end{tikzpicture}
        
            &
        
            \begin{tikzpicture}
              \node at (0, 0.75) {};
              \begin{scope}[yshift = -1cm] 
        
                \begin{scope}[
                    baseline = {([yshift=-0.5ex]current bounding box.center)},
                    scale     = 0.5  
                  ]
                  \draw[ultra thick] (0,0) -- (0,3); 
                  \fill[red] (0,1.5) circle (1.5mm);   
                  \node[left, red] at (-0.2, 1.5) {$\sigma = (\sigma_A, \sigma_B, \phi_\sigma)$}; 
                  \node[above left] at (-0.1, 0)  {$A$}; 
                  \node[below left] at (-0.1, 3)  {$B$}; 
                \end{scope}
        
                \begin{scope}[xshift = 2.5cm, yshift = 1.35cm]
                  \node at (0,    0) {where $\sigma$ is a};
                  \node at (0, -0.5) {surgery triple};
                  \node at (0, -1.0) {from $A$ to $B$};
                \end{scope}
        
              \end{scope}
            \end{tikzpicture}
        
            \\
            \hline
          \end{tabular}
        \end{table}

        Let us clarify the notation and terminology used above:
        \begin{enumerate}
            \item We read diagrams representing generating $1$-morphisms from bottom to 
            top, so in both cases the source is $A$ and the target is $B$.
            \item A surgery triple is the data of orientation-preserving embeddings
            \[
                \sigma_A: S^{k-1} \times D^{2-k} \hookrightarrow A, \quad 
                \sigma_B: D^{k} \times S^{1-k} \hookrightarrow B,
            \]
            together with an orientation-preserving diffeomorphism
            \[
                \phi_{\sigma}: A \setminus \sigma_A(S^{k-1} \times \{0\}) 
                \overset{\cong}{\longrightarrow} 
                B \setminus \sigma_B(\{0\} \times S^{1-k})
            \]
            whose restriction on the image of the framed attaching spheres is given by 
            the diffeomorphism
            \[
                \begin{aligned}
                    S^{k-1} \times (D^{2-k} \setminus \{0\}) &\rightarrow 
                    (D^{k} \setminus \{0\}) \times S^{1-k} \\
                    (u, \lambda v) &\mapsto (v, \lambda u), \quad 0 < \lambda < 1.
                \end{aligned}
            \]

            \item Composition of diagrams is given by vertical concatenation.

        \end{enumerate}

        \item[Generating $2$-morphisms] There are two types of generating $2$-morphisms:

        \begin{table}[H]
        \centering
        \setlength\extrarowheight{4pt}
        \renewcommand{\arraystretch}{1.2}
        
        \begin{tabularx}{\textwidth}{
          |>{\centering\arraybackslash}X
          |>{\centering\arraybackslash}X
          !{\vrule width 2pt}>{\centering\arraybackslash}X|
        }
        \hline
        
        \multicolumn{2}{|c!{\vrule width 2pt}}{\textbf{\textcolor{blue}{Foliation-preserving}}} &
        \multicolumn{1}{c|}{\textbf{\textcolor{red}{Non-foliation-preserving}}} \\
        \hline
        
        
        \vspace{2pt}
        
        \textcolor{blue}{\underline{Diffeo isotopies}}
        
        \vspace{4pt}
        \begin{tikzpicture}[baseline={([yshift=-0.5ex]current bounding box.center)}, scale=0.5]
          \draw[ultra thick] (0,0) -- (0,4);
          \draw[ultra thick] (4,0) -- (4,4);
          \draw[color=blue, ultra thick] (0,2) -- (4,2);
          \fill[color=blue] (2,2) circle (1.5mm);
          \node[above left]  at (-0.1,0) {$A$};
          \node[below left]  at (-0.1,4) {$B$};
          \node[left, blue]  at (-0.1,2) {$\phi$};
          \node[above right] at (4.1,0)  {$A$};
          \node[below right] at (4.1,4)  {$B$};
          \node[right, blue] at (4.1,2)  {$\psi$};
          \node[above, blue] at (2,2.2)  {$\theta$};
        \end{tikzpicture}
        
        \vspace{2pt}
        {\small $\theta$ is an isotopy from\newline $\phi$ to $\psi$}
        
        \vspace{6pt}
        &
        
        \vspace{2pt}

        \textcolor{blue}{\underline{Surgery isotopies}}
        
        \begin{tikzpicture}[baseline={([yshift=-0.5ex]current bounding box.center)}, scale=0.5]
          \draw[ultra thick] (0,0) -- (0,4);
          \draw[ultra thick] (4,0) -- (4,4);
          \draw[ultra thick, red] (0,2) -- (4,2);
          \fill[red] (2,2) circle (1.5mm);
          \node[left, red, scale=0.8]  at (-0.2,2) {$\sigma$};
          \node[above left]  at (-0.1,0) {$A$};
          \node[below left]  at (-0.1,4) {$B$};
          \node[right, red, scale=0.8] at (4.2,2) {$\sigma'$};
          \node[above right] at (4.1,0)  {$A$};
          \node[below right] at (4.1,4)  {$B$};
          \node[above, red]  at (2,2.2) {$\theta^{s}$};
        \end{tikzpicture}
        
        \vspace{2pt}
        {\small $\theta^{s}$ is an isotopy of $\sigma_A, \sigma_B, \chi_{\sigma}$}
        
        \vspace{6pt}
        &

        \vspace{2pt}
        
        \textcolor{red}{\underline{Cusp-birth $\Cb(\sigma,\tau)$}}
        
        \vspace{4pt}
        \begin{tikzpicture}[xscale = -1, baseline={([yshift=-0.5ex]current bounding box.center)}, scale=0.5]
          \draw[ultra thick] (0,0) -- (0,4);
          \draw[ultra thick] (4,0) -- (4,4);
          \draw[red, ultra thick] (1.5,2) parabola (0.7,1.2);
          \draw[red, ultra thick] (1.5,2) parabola (0.7,2.8);
          \draw[red, ultra thick] (0,3) .. controls (0.5,3) and (0.6,2.9) .. (0.7,2.8);
          \draw[red, ultra thick] (0,1) .. controls (0.5,1) and (0.6,1.1) .. (0.7,1.2);
          \draw[ultra thick, blue] (1.5,2) -- (4,2);
          \node[right, scale=0.8, red] at (-0.2,1) {$\sigma$};
          \node[right, scale=0.8, red] at (-0.2,3) {$\tau$};
          \node[above right]  at (-0.1,0) {$A$};
          \node[right]        at (-0.1,2) {$B$};
          \node[below right]  at (-0.1,4) {$\Gamma$};
          \node[left, scale=0.8, blue] at (4.2,2) {$\phi$};
          \node[above left] at (4.1,0)  {$A$};
          \node[below left] at (4.1,4)  {$\Gamma$};
        \end{tikzpicture}
        
        \vspace{2pt}
        {\small $\sigma, \tau$ and $\phi$ satisfy the cancellation condition}
        
        \vspace{6pt}
        \\
        \hline
        
        
        \vspace{2pt}
        
        \hspace{-16pt}\textcolor{blue}{\underline{Compositors}}
        
        \vspace{4pt}
        
        \begin{tikzpicture}[baseline={([yshift=-0.5ex]current bounding box.center)}, scale=0.5]
          \draw[ultra thick] (0,0) -- (0,4);
          \draw[ultra thick] (4,0) -- (4,4);
          \draw[color=blue, ultra thick] (2,2) -- (0,1);
          \draw[color=blue, ultra thick] (2,2) -- (0,3);
          \draw[color=blue, ultra thick] (2,2) -- (4,2);
          \node[left, blue]  at (-0.1,1.2) {$\phi$};
          \node[left, blue]  at (-0.1,2.8) {$\psi$};
          \node[above left]  at (-0.1,0) {$A$};
          \node[left]        at (-0.1,2) {$B$};
          \node[below left]  at (-0.1,4) {$\Gamma$};
          \node[above right] at (4.1,0)  {$A$};
          \node[below right] at (4.1,4)  {$\Gamma$};
          \node[right, blue] at (4.1,2)  {$\psi\circ\phi$};
        \end{tikzpicture}

        \vspace{2pt}
        {\small and their inverse; $\phi$ and $\psi$ are composable diffeomorphisms}
        
        \vspace{6pt}
        &
        
        \vspace{2pt}
        
        \textcolor{blue}{\underline{Diffeo-surgery $2$-morphisms}}

        \vspace{4pt}
        
        \begin{tikzpicture}[baseline={([yshift=-0.5ex]current bounding box.center)}, scale=0.33]
          \draw[ultra thick] (0,0) -- (0,6);
          \draw[ultra thick] (6,0) -- (6,6);
          \draw[ultra thick, red]  (0,3) -- (6,3);
          \draw[ultra thick, blue] (0,1) -- (1,3);
          \draw[ultra thick, blue] (6,5) -- (5,3);
          \node[left, scale=0.8, blue] at (-0.2,1.1) {$\phi$};
          \node[right, scale=0.8, blue] at (6.2,4.9) {$\psi$};
          \node[left, scale=0.8, red]  at (-0.2,3) {$\sigma$};
          \node[above left]  at (-0.1,-0.7) {$A$};
          \node[left]        at (-0.1,2) {$A'$};
          \node[below left]  at (-0.1,6.7) {$B$};
          \node[right, scale=0.8, red] at (6.2,3) {$\tau$};
          \node[above right] at (6.1,-0.7)  {$A$};
          \node[right] at (6.1,4)  {$B'$};
          \node[below right]  at (6.1,6.7) {$B$};

          \node[above, red] at (3,3) {\footnotesize $\phi_*\sigma = \psi^*\tau$};
        \end{tikzpicture}

        \vspace{2pt}
        {\small and their inverse;}
        
        \vspace{6pt}
        &
        
        \vspace{2pt}
        
        \textcolor{red}{\underline{Cusp-death $\Cd(\sigma,\tau)$}}
        
        \vspace{4pt}
        \begin{tikzpicture}[baseline={([yshift=-0.5ex]current bounding box.center)}, scale=0.5]
          \draw[ultra thick] (0,0) -- (0,4);
          \draw[ultra thick] (4,0) -- (4,4);
          \draw[red, ultra thick] (1.5,2) parabola (0.7,1.2);
          \draw[red, ultra thick] (1.5,2) parabola (0.7,2.8);
          \draw[red, ultra thick] (0,3) .. controls (0.5,3) and (0.6,2.9) .. (0.7,2.8);
          \draw[red, ultra thick] (0,1) .. controls (0.5,1) and (0.6,1.1) .. (0.7,1.2);
          \draw[ultra thick, blue] (1.5,2) -- (4,2);
          \node[left, scale=0.8, red] at (-0.2,1) {$\sigma$};
          \node[left, scale=0.8, red] at (-0.2,3) {$\tau$};
          \node[above left]  at (-0.1,0) {$A$};
          \node[left]        at (-0.1,2) {$B$};
          \node[below left]  at (-0.1,4) {$\Gamma$};
          \node[right, scale=0.8, blue] at (4.2,2) {$\phi$};
          \node[above right] at (4.1,0)  {$A$};
          \node[below right] at (4.1,4)  {$\Gamma$};
        \end{tikzpicture}
        
        \vspace{2pt}
        {\small $\sigma, \tau$ and $\phi$ satisfy the cancellation condition}
        
        \vspace{6pt}
        \\
        \hline
        
        
        \vspace{2pt}
        
        \textcolor{blue}{\underline{Permutations}}
        
        \vspace{4pt}
        \begin{tikzpicture}[baseline={([yshift=-0.5ex]current bounding box.center)}, scale=0.5]
          \draw[ultra thick] (0,0) -- (0,4);
          \draw[ultra thick] (4,0) -- (4,4);
          \draw[color=green, ultra thick] (0,2) -- (2,2);
          \draw[color=blue, ultra thick]  (2,2) -- (4,2);
          \fill[teal] (2,2) circle (1.5mm);
          \node[above left]  at (-0.1,0) {$A$};
          \node[below left]  at (-0.1,4) {$B$};
          \node[left, green] at (-0.1,2) {$\omega$};
          \node[above right] at (4.1,0)  {$A$};
          \node[below right] at (4.1,4)  {$B$};
          \node[right, blue] at (4.1,2)  {$\omega$};
        \end{tikzpicture}
        
        \vspace{2pt}
        {\small and their inverse; $\omega$ is a permutation of the connected components of $A$}
        
        \vspace{6pt}
        &
        
        \vspace{2pt}

        \textcolor{blue}{\underline{Frame-changes}}
        
        \begin{tikzpicture}[baseline={([yshift=-0.5ex]current bounding box.center)}, scale=0.5]
          \draw[ultra thick] (0,0) -- (0,4);
          \draw[ultra thick] (4,0) -- (4,4);
          \draw[ultra thick, red] (0,2) -- (4,2);
          \fill[black] (2,2) circle (1.5mm);
          \node[left, red, scale=0.8]  at (-0.2,2) {$\sigma$};
          \node[above left]  at (-0.1,0) {$A$};
          \node[below left]  at (-0.1,4) {$B$};
          \node[right, red, scale=0.8] at (4.2,2) {$Q\cdot\sigma$};
          \node[above right] at (4.1,0)  {$A$};
          \node[below right] at (4.1,4)  {$B$};
          \node[above]        at (2,2.2)  {$Q$};
        \end{tikzpicture}
        
        \vspace{2pt}
        {\small $\sigma$ is of index $k$, and $Q\in(\on{O}(k)\times\on{O}(2-k))\cap\on{SO}(2)$}
        
        \vspace{6pt}
        &
        
        \vspace{2pt}
        
        \textcolor{red}{\underline{Crossing $\XX(\sigma,\rho,\sigma',\rho')$}}
        
        \vspace{4pt}
        \begin{tikzpicture}[baseline={([yshift=-0.5ex]current bounding box.center)}, scale=0.5]
          \draw[ultra thick] (0,0) -- (0,4);
          \draw[ultra thick] (4,0) -- (4,4);
          \draw[ultra thick, red] (0,1) -- (4,3);
          \draw[ultra thick, red] (0,3) -- (4,1);
          \node[left, scale=0.8, red]  at (-0.2,1.1) {$\sigma$};
          \node[left, scale=0.8, red]  at (-0.2,2.9) {$\rho$};
          \node[above left]  at (-0.1,0) {$A$};
          \node[left]        at (-0.1,2) {$B$};
          \node[below left]  at (-0.1,4) {$\Gamma$};
          \node[right, scale=0.8, red] at (4.2,1.1) {$\rho'$};
          \node[right, scale=0.8, red] at (4.2,2.9) {$\sigma'$};
          \node[above right] at (4.1,0)  {$A$};
          \node[right]        at (4.1,2)  {$B'$};
          \node[below right] at (4.1,4)  {$\Gamma$};
        \end{tikzpicture}
        
        \vspace{2pt}
        {\small $\sigma,\rho, \sigma'$ and $\rho'$ satisfy the crossing condition}
        
        \vspace{6pt}
        \\
        \hline
        
        \end{tabularx}
        \end{table}

        Let us clarify the notation and terminology used above:
        \begin{description}
            \item[Permutations:] 
            Here $A = A_1 \sqcup A_2 \sqcup \ldots \sqcup A_n$ and $\omega \in \Sigma_n$ permutes the connected components of $A$. On one hand, $\omega$ gives rise to a coherence $1$-morphism (drawn in green) in the $2$-category generated by the presentation, since that category is unbiased semistrict symmetric monoidal. On the other hand, $\omega$ is a diffeomorphism and thus a generating $1$-morphism (drawn in blue). The symmetry $2$-morphism is an invertible generating $2$-morphism identifying the two.

            \item[Diffeo-surgery $2$-morphisms:]
             For $\sigma = (\sigma_{A'}, \sigma_B, \chi_{\sigma})$ we define
            \( 
            \phi_*\sigma = (\phi^{-1}\circ \sigma_{A'}, \sigma_B, \chi_{\sigma} \circ \phi)
            \)
            and $\psi^*\tau$ is defined analogously.

            \item[Frame-change $2$-morphisms:]
            The group $(\on{O}(k)\times\on{O}(2-k))\cap\on{SO}(2)$ acts on the space of embeddings of $S^{k-1}\times D^{2-k}$ in $A$ and $D^k \times S^{1-k}$ in $B$ by pre-composition; $Q\cdot \sigma$ denotes this action. 

            \item[Cusp-births and cusp-deaths:]
            We say $\sigma$, $\tau$, and $\phi$ satisfy the cancellation condition if:
            \begin{itemize}
                \item $\sigma_B(\{0\}\times S^{1-k})$ and $\tau_B(S^{k}\times \{0\})$ intersect transversely in $B$ at a single point. We moreover require that the framings also intersect in a standard way (which can always be realized by an isotopy). See \cite[Def. 1.6.2]{Filippos}. 
                \item $\phi$ is the canonical canceling diffeomorphism associated to $(\sigma, \tau)$, as constructed in Section~1.6 of~\cite{Filippos} and reviewed in Section~\ref{subsubsec:bord_realization} below.
            \end{itemize}
            

        \item[Crossings:] We say that $\sigma$, $\rho$, $\sigma'$, and $\rho'$ satisfy the crossing condition if:
        \begin{itemize}
            \item $\on{Im}(\sigma_B) \cap \on{Im}(\rho_B) = \emptyset$ and $\on{Im}(\rho'_{B'}) \cap \on{Im}(\sigma'_{B'}) = \emptyset$.
            \item $\chi_{\rho} \circ \chi_{\sigma} = \chi_{\sigma} \circ \chi_{\rho}$ wherever both sides are defined.
            \item $\chi_{\sigma}\circ \rho'_A = \rho_B$, $\chi_{\rho'}\circ \sigma_A = \sigma'_{B'}$, $\chi_{\rho}\circ \sigma_B = \sigma'_{\Gamma}$, and $\chi_{\sigma'}\circ \rho'_{B'} = \rho_{\Gamma}$.
        \end{itemize}
            
        \end{description}

        \item[Relations.] All relations in Section~3.3 of~\cite{Filippos}. These are reviewed in Section~\ref{subsec:relations+result} below, where we simultaneously verify that our proposed TQFT satisfies them.
    \end{description} 
\end{definition}

\subsubsection{The bordism realization functor}\label{subsubsec:bord_realization}

Having introduced the generators of the surgery presentation $\on{Sur}^{\on{or}}_{12\sim}$, we now describe the bordism realization functor 
$$\BRF: \on{FSur}^{\on{or}}_{12\sim} \to \on{Bord}^{\on{or}}_{12\sim}$$
emphasizing the aspects needed later in the construction of our TQFT. For a more detailed treatment of $\BRF$ and a proof that it is an equivalence, we refer the reader to Chapter~3 of~\cite{Filippos}.

To define $\BRF$, it suffices to describe its assignment on the generating data of $\on{Sur}^{\on{or}}_{12\sim}$ and then check that this assignment satisfies the relations. We now describe this assignment, somewhat informally, as follows:
\begin{description}
    \item[Assignment on generating $1$-morphisms:] For generating $1$-morphisms of 
    diffeomorphism type, $\BRF$ sends the diagram labeled by the diffeomorphism $\phi$ 
    to the mapping cylinder of $\phi$, which comes canonically equipped with a Morse 
    function.

    For generating $1$-morphisms of surgery type, $\BRF$ sends the diagram labeled by 
    the surgery triple $\sigma = (\sigma_A, \sigma_B, \chi_{\sigma})$ to the handle 
    attachment defined by $\sigma_A$, together with the boundary identifications 
    determined by $\sigma_B$ and $\chi_{\sigma}$. We denote this bordism by 
    $W(A, B, \sigma)$, as above; it also comes canonically equipped with a Morse 
    function with a single critical point, of index $k$ if $\sigma_A$ is an embedding 
    of $S^{k-1} \times D^{n-k}$ in $A$. Moreover, using Milnor's trace of the surgery 
    construction~\cite{Milnor_hcobordism}, $W(A, B, \sigma)$ comes canonically equipped 
    with a standard Morse model embedding and a compatible gradient-like vector field, 
    which we collectively refer to as \emph{Morse data}. Informally, the standard Morse 
    model is a handle $D^k \times D^{n-k}$ with its corners smoothed out in a prescribed 
    way; an explicit description can be found in~\cite{Milnor_hcobordism} 
    or~\cite{Filippos}. This assignment is essentially bijective: every surgery diagram 
    (i.e., every formal composite of these generating $1$-morphisms) determines a bordism 
    with Morse data, and conversely, every bordism equipped with Morse data determines a 
    surgery diagram together with an explicit diffeomorphism to its bordism realization.

    \item[Assignment on generating $2$-morphisms:] \ \vskip0pt
    \begin{itemize}
        \item Foliation-preserving generators:
        As noted above, the assignment of $\BRF$ on $1$-morphisms comes canonically equipped with a Morse function, which induces a Morse singular foliation. $\BRF$ maps the foliation-preserving generating $2$-morphisms to diffeomorphisms preserving the Morse singular foliations on the source and target bordisms. This property suffices for the construction of our TQFT; for the precise definition of $\BRF$ on these generators, see Chapter~3 of~\cite{Filippos}.

        \item Cusps:
        We focus on the cusp-birth; the cusp-death is its inverse. Here, we have 
        $A = (S^1)^{\sqcup n}$, $\sigma: A \to B \simeq A \sqcup S^1$ an index $0$ surgery, 
        and $\tau: B \to \Gamma \simeq A$ an index $1$ surgery. The bordism realization functor 
        sends the target of the cusp-birth to the bordism
        \[
            W := W(A, B, \sigma) \cup W(B, \Gamma, \tau),
        \]
        equipped with a Morse function with two critical points $p_{\sigma}$ and $p_{\tau}$ 
        of index $0$ and $1$, respectively (Figure~\ref{fig:cusp-death}), standard Morse model 
        embeddings $\iota_{\sigma}: M_0 \rightarrow W$ and $\iota_{\tau}: M_1 \rightarrow W$ 
        adapted to $p_{\sigma}$ and $p_{\tau}$ (Figure~\ref{fig:cusp-death-morse-data}), and 
        a gradient-like vector field.
    
        \begin{figure}[htbp]
            \centering
            \includegraphics[width=0.5\linewidth]{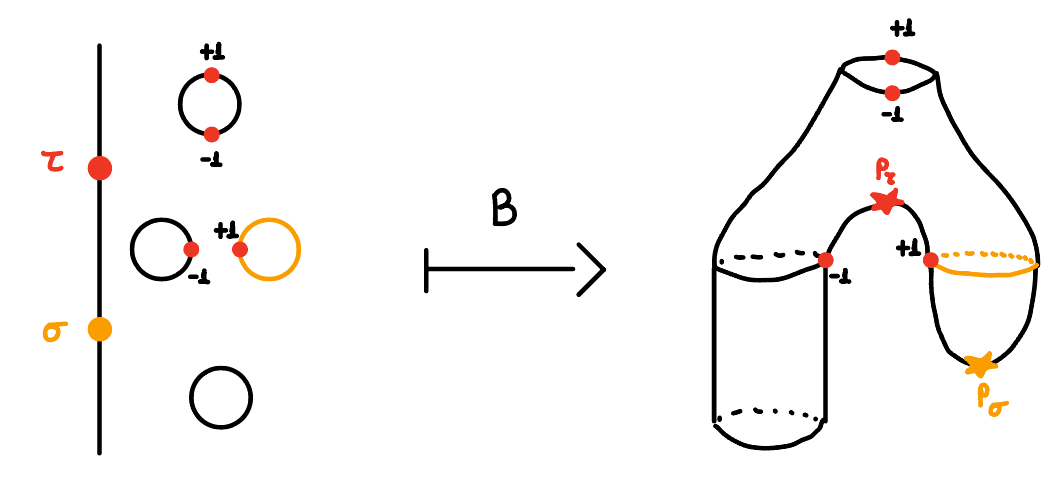}
            \caption{The bordism $W = W(A,B,\sigma) \cup W(B,\Gamma,\tau)$ with its two critical 
            points $p_{\sigma}$ and $p_{\tau}$ of index $0$ and $1$, respectively.}
            \label{fig:cusp-death}
        \end{figure}

        \begin{figure}[htbp]
            \centering
            \begin{subfigure}[b]{0.48\textwidth}
                \centering
                \includegraphics[width=0.6\linewidth]{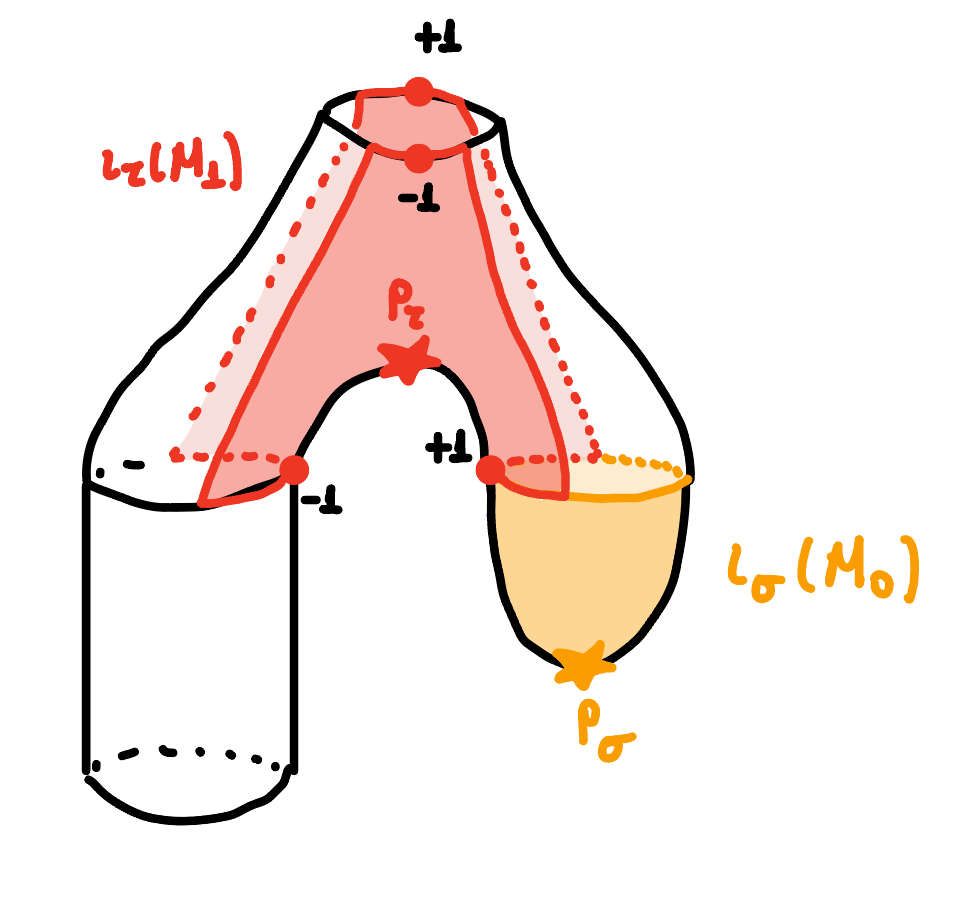}
                \caption{Standard Morse model embeddings in cancelling position.}
                \label{fig:cusp-death-morse-data}
            \end{subfigure}
            \hfill
            \begin{subfigure}[b]{0.48\textwidth}
                \centering
                \includegraphics[width=0.6\linewidth]{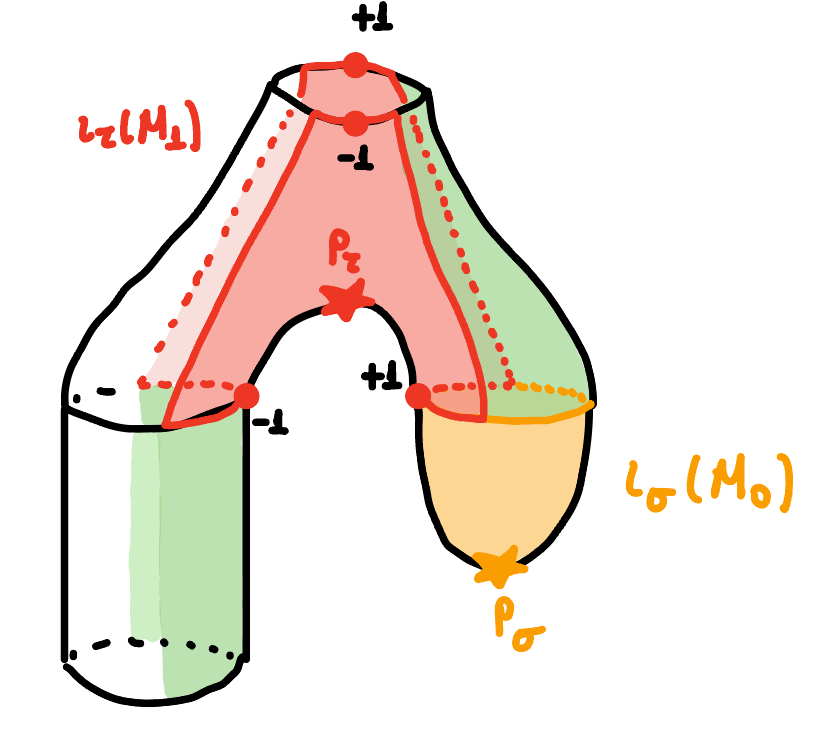}
                \caption{The saturated double neighborhood for a pair of standard Morse models 
            in canceling position.}                             \label{fig:extended_double_neighborhood}
            \end{subfigure}
            \caption{(Saturated) double neighborhoods.}
            \label{fig:both}
        \end{figure}

        The \emph{double neighborhood} of $p_\sigma$ and $p_\tau$ is 
        $\iota_{\sigma}(M_0) \cup \iota_{\tau}(M_1)$. Using the gradient-like vector field, 
        it can be extended to a \emph{saturated double neighborhood} $U$, stretching from the 
        lowest level set $l_0$ to the highest level set $l_1$ of $W$; see 
        Figure~\ref{fig:extended_double_neighborhood} for a visual representation.

        The diffeomorphism induced by the cusp-birth $$\Phi:=\BRF(\Cb(\sigma,\tau)): \BRF(\phi) \to W$$ is foliation preserving outside ${U}$, and is depicted inside $U$ in Equation \eqref{eq:defUandVcusp}.

        More precisely, there is an \emph{elementary birth path} $\{b_\lambda\}_{0 \leq \lambda \leq 1}$ 
        supported inside $U$, agreeing with the original Morse function on $W \setminus U$, 
        such that $b_0$ has no critical points, $b_{1/2}$ has a single cusp critical point, 
        and $b_1$ is the original Morse function on $W$. The saturated double neighborhood 
        carries a canonical Riemannian metric, which together with the original gradient-like 
        vector field on $W \setminus U$ yields a gradient-like vector field for 
        $b_0$ on all of $W$. Integrating this vector field induces the diffeomorphism 
        $\phi: A \xrightarrow{\sim} \Gamma$ and a diffeomorphism $\Phi$ from the mapping 
        cylinder of $\phi$ to $W$. This is the diffeomorphism the bordism realization functor 
        assigns to the cusp-birth $\Cb(\sigma, \tau)$. We refer the reader to Chapter~1.6 
        of~\cite{Filippos} for further details.


        \item Crossings:
        The bordism realization functor sends the source of the crossing to the bordism 
        \[  W := W(A, B, \sigma) \cup W(B, \Gamma, \rho)\]
        equipped with a Morse function with two critical points $p_{\sigma}$ and $p_{\rho}$, standard Morse model embeddings $\iota_{\sigma}: M_0 \rightarrow W$ and $\iota_{\rho}: M_1 \rightarrow W$ 
        adapted to $p_{\sigma}$ and $p_{\rho}$, and 
        a gradient-like vector field. It sends the target to 
        \[W' := W(A, B', 
        \tilde{\rho}) \cup W(B', \Gamma, \tilde{\sigma}),\]
        equipped with the same kind of extra structure. Similarly to the treatment of the cusps, the \emph{double neighborhood} of $p_\sigma$ and $p_\rho$ is 
        $\iota_{\sigma}(M_0) \cup \iota_{\rho}(M_1)$. Using the gradient-like vector field, 
        it can be extended to a \emph{saturated double neighborhood} $U$, stretching from the 
        lowest level set $l_0$ to the highest level set $l_1$ of $W$. In contrast to the cusp, the saturated double neighborhood $U$ consists of the disjoint union of two standard Morse model embeddings. 

        The diffeomorphism induced by the crossing
        \[\Phi:=\BRF(\XX(\sigma,\rho,\sigma',\rho')):W \to W' \]
        is foliation-preserving outside of $U$. If either $\sigma$ or $\rho$ has index $0$ or $2$, then the crossing can be 
        realized by a foliation-preserving diffeomorphism everywhere; if both have index $1$, the diffeomorphism $\Phi$ inside $U$ is depicted in 
        Equation~\eqref{eq:defUandVcrossing}.
        
        More precisely, there is an \emph{elementary 
        crossing path} $\{c_\lambda\}_{\lambda 
        \in [0,1]}$ on $W$ supported inside $U$, agreeing with the original Morse function on $W\setminus U$ that interchanges the critical levels of the two critical points in $U$. The saturated double neighborhood 
        carries a canonical Riemannian metric, which together with the original gradient-like 
        vector field on $W \setminus U$ yields a gradient-like vector field for 
        $c_1$ on all of $W$, and hence new Morse data for $W$. This new Morse data determines the surgery diagram depicted in the target of the crossing $2$-morphism and the diffeomorphism $\Phi$ from $W$ to $W'$. We refer the reader to Chapter~1.6 
        of~\cite{Filippos} for further details.   
    \end{itemize}
   
\end{description}

\subsection{Assignment on generators}\label{subsec:generators}
\subsubsection{Foliated generators}
\begin{itemize}
    \item \textbf{Generating objects}\newline
     The functor $\SN_{\II}^{\alpha}$ assigns to a connected oriented 1-manifold $\Gamma$ the string net category $\SN_\II(\Gamma)$ of Definition~\ref{def:SN_on_1-mflds}.

    \item \textbf{Generating $1$-morphisms}\newline
    For both kinds of generating $1$-cells, the bordism realization functor $\BRF: \on{FSur}_{12\sim}^{\on{or}} \rightarrow \Bord_{12\sim}^{\on{or}}$ sends them to bordisms equipped with a Morse function and hence with an induced Morse foliation. 

    We set $$\orSN_\II^\alpha(\phi) = \SN_\II^\alpha(\BRF(\phi)) \text{ and } \orSN_\II^\alpha(\sigma) = \SN_\II^\alpha(\BRF(\sigma))\ .$$
    Note that the source and target of generating 1-morphisms are words in generating objects, i.e. formal disjoint union of connected 1-manifolds, and map under our assignment $\orSN_\II^\alpha$ to the formal tensor product of the string net categories.
    If $A = \Gamma_1 \sqcup \cdots \sqcup \Gamma_n$ is not connected, we implicitly use the isomorphism of Theorem \ref{thm:foliated_tqft}:
    $$\orSN_\II^\alpha(A) := \SN_\II^\alpha(\Gamma_1) \otimes \cdots\otimes \SN_\II^\alpha(\Gamma_1) \simeq  \SN_\II^\alpha( \Gamma_1 \sqcup \cdots \sqcup \Gamma_n)\ .$$

\item \textbf{Foliation-preserving generating 2-morphisms}\newline
All these generators induce under $\BRF$ a diffeomorphism preserving the Morse functions, and hence the Morse singular foliations. 
We set $$\orSN_\II^\alpha(D) = \SN_\II^\alpha(\BRF(D))\ .$$
Again, the source and target of $D$ are formal sentences in generating 1-morphisms, and we identify their image under our assignment $\orSN_\II^\alpha$ with the string net module of the total surface using the isomorphisms of Theorem \ref{thm:foliated_tqft}.
\end{itemize}


\subsubsection{Cusps}
We now turn to the assignment of $\SN_{\II}^{\alpha}$ on the cusp generating $2$-morphisms. Let us treat in detail the cancellation of $0$- and $1$-handles; the cancellation 
of $1$- and $2$-handles is very similar and can be obtained by reversing the orientation and trading $\alpha$ for $\alpha^{-1}$. 

Let $\sigma$ and $\tau$ be a canceling pair of 0 and 1 handles, and $\Phi:=\BRF(\Cb(\sigma,\tau)):\BRF(\phi){\to} W$ the diffeomorphism induced by the cusp. We use the notations of the cusp paragraph in Section \ref{subsubsec:bord_realization}, in particular $U$ denotes the saturated double neighborhood.

The diffeomorphism $\Phi$ 
is foliation-preserving outside ${U}$, and even outside a slightly smaller neighborhood $V$, so we may focus our attention there. We denote $V' = \Phi^{-1}(V) \subseteq  U' = \Phi^{-1}(U) \subseteq \BRF(\phi)$. The Morse singular foliation on ${U}$ and $U'$ are depicted as follows:
\begin{equation}\label{eq:defUandVcusp}
    \begin{tikzpicture}[xscale = 0.6, yscale = -0.6, baseline = 0.5cm]
    \foreach \y in {0,...,12}{
    \draw[gray!70] (-3,-4.4+0.51*\y) -- ++(6,0);
    }
    \draw[gray!70] (-3,2.3) --++(6,0) node[right]{$U'$};
    \draw[red!70!black] (-2,-3.6) rectangle (2,1.5) node[right]{$V'$}; 
    \draw[gray, ->] (3.2,-2.5) --++(0,-0.5) node[midway,right = -2pt]{\small $\vec{y}$};
    \end{tikzpicture}
\overset{\Phi}{\simeq}
\quad
    \begin{tikzpicture}[xscale = 0.6, yscale = -0.6, baseline = 0.5cm]
    \tikzmath{\gap=1;}
    \def\sadlines{2}
    \draw[smooth, gray!70] plot coordinates {(-2,1) (-1.5,0.88) (0,0) (1,-\gap) (0,-\gap-1) (-1,-\gap) (0,0)(1.5,0.88)(2,1)};
    \draw[ gray!70] (2,1) --++(1,0);
    \draw[ gray!70] (-3,1) --++(1,0);
    \foreach \y in {0,...,\sadlines}{
    \tikzmath{\step=\y/(\sadlines+1);}
    \draw[ gray!70] (-3,1.5-\step/2) --++(1,0);
    \draw[smooth, gray!70] plot coordinates {(-2,1.5-\step/2) (-1.5,1.5-0.65*\step) (0,1.5-1.5*\step) (1.5,1.5-0.65*\step) (2,1.5-\step/2)} -- ++(1,0); 
    }

    \draw[gray!70] (-3,-2.6-\gap) -- (3,-2.6-\gap);
    \draw[smooth, gray!70] plot coordinates {(-2,-1.8-\gap) (-1.5, -2.9) (0,-2.4-\gap) (1.5, -2.9)(2,-1.8-\gap)} --++(1,0);
     \draw[smooth, gray!70] plot coordinates {(-2,-1-\gap)(-1.5, -2.14) (0,-2.2-\gap)(1.5, -2.14) (2,-1-\gap)} --++(1,0);
    \draw[smooth, gray!70] plot coordinates {(-2,-0.3-\gap)(-1.5, -1.52)  (0,-2-\gap)(1.5, -1.52) (2,-0.3-\gap)}--++(1,0);
    \draw[ gray!70] (-3,-1.8-\gap) --++(1,0);
    \draw[ gray!70] (-3,-1-\gap) --++(1,0);
    \draw[ gray!70] (-3,-0.3-\gap) --++(1,0);

    \foreach \y in {1,...,\sadlines}{
    \tikzmath{\step=\gap/(\sadlines+1);}
    \tikzmath{\x=2*\y/(\sadlines+1);}
    \draw[ gray!70] (-3,1-\y*\step-0.5*\step) --++(1,0);
    \draw[smooth, gray!70] plot coordinates {(-2,1-\y*\step-0.5*\step)(-1.5,1-0.95*\y*\step-0.85*\step) (-\x,-0.15*\gap*\x*\x) (-1-0.45*\x,-\gap) (0,-\gap-1 - 0.5*\x) (1+0.45*\x,-\gap) (\x,-0.15*\gap*\x*\x) (1.5, 1-0.95*\y*\step-0.85*\step) (2,1-\y*\step-0.5*\step)} --++(1,0); 
    }
   
    \foreach \y in {1,...,\sadlines}{
    \tikzmath{\step=\y/(\sadlines+1);}
    \draw[gray!70] plot[smooth cycle] coordinates {(-1+\step,-\gap) (0,-\gap*\step) (1-\step,-\gap) (0,-\gap-1+\step)};
    }
    
    \node[star,  star point ratio=2.25, fill=black, inner sep = 1pt] at (0,0) {};
    \node[star,  star point ratio=2.25, fill=black, inner sep = 1pt] at (0,-\gap) {};

    \draw[gray!70] (-3,1.9) --++(6,0);
    \draw[gray!70] (-3,2.3) --++(6,0) node[right]{$U$};
    \draw[gray!70] (-3,-4) --++(6,0);
    \draw[gray!70] (-3,-4.4) --++(6,0);

    \draw[red!70!black] (-2,-3.6) rectangle (2,1.5) node[right]{$V$};
    \end{tikzpicture}
\end{equation}

\begin{proposition}
    There exists a unique natural transformation
    $$\orSN_\II^\alpha(\Cb(\sigma,\tau)) : \SN_\II^\alpha(\BRF(\phi)) \to \SN_\II^\alpha(W)$$
    which maps a string net $T \subseteq \BRF(\phi)$ which is disjoint from $V'$ to $\Phi(T)\sqcup \Theta$ where $\Theta$ is the string net inside $V$ given by:
    $$\Theta:= \begin{tikzpicture}[xscale = 0.6, yscale = -0.6, baseline = 0.5cm]
        \cusp
        \draw[thick] (0,-\gap) -- (0,0)  node[pos=0.5,sloped]{\footnotesize $<$} node[right,pos=0.5]{$\alpha$};
    \end{tikzpicture} $$
\end{proposition}
\begin{proof}
Choose any point $p\in V'$ and let $T$ be any skein in $\BRF(\phi)$. If $T$ does not intersect $p$, then it can be isotoped into a string net $T'$ disjoint from $V'$ in a canonical way, and we set $\orSN_\II^\alpha(\Cb(\sigma,\tau))(T) := \Phi(T')\sqcup \Theta$. Remember that $\Phi$ preserves the foliation except inside $V'$, so indeed $\Phi(T')\sqcup \Theta$ is a progressive string net in $W$ with the same boundary labels as $T$. 

If $T$ does intersect $p$, then we can isotope it slightly to be disjoint from $p$, and apply the procedure above. There are two choices for how to isotope $T$ away from $p$, which give respectively: 
\begin{equation*}
    \begin{tikzpicture}[xscale = 0.6, yscale = -0.6, baseline = 0.5cm]
    \cusp
    \draw[thick] (0,-\gap) -- (0,0)  node[pos=0.5,sloped]{\footnotesize $<$} node[right,pos=0.5]{$\alpha$};
    \draw[black,thick,sloped] (0,1.5) to[out=-90,in=90] (1.8,-0.8*\gap) to[out=-90,in=90,looseness=0.7] node[pos=0,sloped]{\footnotesize $<$}  (0,-\gap-2.6);
    \end{tikzpicture}
    \ \sim \
    \begin{tikzpicture}[xscale = 0.6, yscale = -0.6, baseline = 0.5cm]
    \cusp
    \draw[thick] (0,-\gap) -- (0,0)  node[pos=0.5,sloped]{\footnotesize $<$} node[right,pos=0.5]{$\alpha$};
    \draw[black,thick] (0,1.5) to[out=-90,in=120] node[pos=1]{\footnotesize $\bullet$} node[pos=0.5, rotate = 45]{\footnotesize $>$}  (1.2,-0.2*\gap) to[out=-150,in=30] node[pos=0.5, rotate = -60]{\footnotesize $>$} node[pos=1,sloped]{$\bullet$} (0.5,-\gap) to[out=-60,in=90] node[pos=0.5, rotate = 120]{\footnotesize $>$}  (0,-\gap-2.6);
    \end{tikzpicture}
    \ \sim \
    \begin{tikzpicture}[xscale = 0.6, yscale = -0.6, baseline = 0.5cm]
    \cusp
    \draw[thick] (0,-\gap) -- (0,0)  node[pos=0.3,sloped]{\footnotesize $<$} node[right,pos=0.3]{$\alpha$};
    \draw[black,thick] (0,1.5) to[out=-90,in=120] node[pos=1]{\footnotesize $\bullet$} node[pos=0.5, rotate = 50]{\footnotesize $>$}  (1.2,-0.2*\gap) to[out=-180,in=80] node[pos=0.6,draw, rectangle, fill=white, inner sep = 3pt, rotate = -10]{} node[pos=1,sloped]{$\bullet$} (-0.4,-1.2) to[out=-150,in=90] node[pos=0.5,rotate = 70]{\footnotesize $>$}  (0,-\gap-2.6);
    \end{tikzpicture}
    \ \sim \
    \begin{tikzpicture}[xscale = 0.6, yscale = -0.6, baseline = 0.5cm]
    \cusp
    \draw[thick] (0,-\gap) -- (0,0)  node[pos=0.5,sloped]{\footnotesize $<$} node[right,pos=0.5]{$\alpha$};
    \draw[black,thick] (0,1.5) to[out=-90,in=60] node[pos=1]{\footnotesize $\bullet$} node[pos=0.5, rotate = 135]{\footnotesize $>$}  (-1.2,-0.2*\gap) to[out=-30,in=120] node[pos=0.5,rotate = 55]{\footnotesize $<$} node[pos=1,sloped]{$\bullet$} (-0.4,-\gap) to[out=-150,in=90] node[pos=0.5,rotate = 70]{\footnotesize $>$}  (0,-\gap-2.6);
    \end{tikzpicture}
    \ \sim \
     \begin{tikzpicture}[xscale = 0.6, yscale = -0.6, baseline = 0.5cm]
     \cusp
    \draw[thick] (0,-\gap) -- (0,0)  node[pos=0.5,sloped]{\footnotesize $<$} node[right,pos=0.5]{$\alpha$};
    \draw[black,thick,sloped] (0,1.5) to[out=-90,in=90] (-1.8,-0.8*\gap) to[out=-90,in=90,looseness=0.7] node[pos=0,sloped]{$<$}  (0,-\gap-2.6);
    \end{tikzpicture}
\end{equation*}
hence the map is well-defined. Naturality is clear as this procedure happens away from the boundary.
\end{proof}

\begin{definition}
    A string net $T \subseteq W$ is said to be \textit{in canceling position} if $T\cap V$ is equal to $\Xi \subseteq V$ where
    $$\Xi := \begin{tikzpicture}[xscale = 0.6, yscale = -0.6, baseline = 0.5cm]
    \cusp
    \draw[thick] (0,1.5) -- (0,0)  node[pos=0.5,sloped]{\footnotesize $<$} node[right,pos=0.5]{$\alpha$};
    \draw[thick] (0,-\gap) -- (0,-\gap-2.6)  node[pos=0.5,sloped]{\footnotesize $<$} node[right,pos=0.5]{$\alpha$};
    \end{tikzpicture}$$
    In this case, we denote $T_{\smallsetminus V} = T\cap (W\smallsetminus V)$, so that $T = T_{\smallsetminus V}\cup \Xi$.
\end{definition}

\begin{proposition}
    There exists a unique natural transformation
    $$\orSN_\II^\alpha(\Cd(\sigma,\tau)) :\SN_\II^\alpha(W) \to \SN_\II^\alpha(\BRF(\phi))$$
    which maps a string net $T \subseteq W$ in canceling position to the string net $\Phi^{-1}(T_{\smallsetminus V}) \cup L  \subseteq \BRF(\phi)$ where $L\subseteq V'$ is the $\alpha$-colored line shown below:
    $$T_{\smallsetminus V}\cup\begin{tikzpicture}[xscale = 0.6, yscale = -0.6, baseline = 0.5cm]
    \cusp
    \draw[thick] (0,1.5) -- (0,0)  node[pos=0.5,sloped]{\footnotesize $<$} node[right,pos=0.5]{$\alpha$};
    \draw[thick] (0,-\gap) -- (0,-\gap-2.6)  node[pos=0.5,sloped]{\footnotesize $<$} node[right,pos=0.5]{$\alpha$};
    \end{tikzpicture}
    \longmapsto \Phi^{-1}(T_{\smallsetminus V})\cup     \begin{tikzpicture}[xscale = 0.6, yscale = -0.6, baseline = 0.5cm]    
    \tikzmath{\gap=1.4;}
    \draw[thick] (0,1.5) -- (0,-\gap-2.6)  node[pos=0.5,sloped]{\footnotesize $<$} node[right,pos=0.5]{$\alpha$};
    \foreach \y in {0,...,10}{
    \tikzmath{\h=(1.5+2.6+\gap)*\y/10;}
    \draw[gray!70] (-2,1.5 - \h) -- (2,1.5 - \h);
    }
    \end{tikzpicture}$$
\end{proposition}
\begin{proof}
Let $T$ be a string net in $W$, which we do not consider up to isotopy, skein and sliding relations for now. We need to put $T$ in canceling position. We consider the progressive arc $\gamma$ joining both critical points as below. A string net disjoint from $\gamma$ can be put in canceling position in a canonical way using the saddle relations \eqref{eq:saddle-rel} and isotopies during which the string net does not intersect $\gamma$. It therefore suffices to make $T$ disjoint from $\gamma$.

In a small-enough neighborhood of the index-0 critical point $p_\sigma$, $T$ consists of a single $\alpha$-colored strand. If this strand is not tangent to $\gamma$ at $p_\sigma$, we can isotope it to be vertical in a canonical way. If it is tangent to $\gamma$ at $p_{\sigma}$, first apply a small isotopy so that it is no longer tangent. Now we can embed a disk $\D_1$ containing the other critical point and all of $T \cap \gamma$ as follows:
\begin{equation*}
    \begin{tikzpicture}[xscale = 0.6, yscale = -0.6, baseline = 0.5cm]
    \fill[gray!30] (-1,1.5) -- (1,1.5) to[out=-90, in = 120] (1.5, 0) to[out = -170, in = 120, looseness = 2] (1.45,-1.1) to [out = -160, in = 10] (0.2,-1.4) arc(0:180:0.2 and 0.3) to [out = 170, in = -20] (-1.45,-1.1) to[out = 60, in = -10, looseness = 2](-1.5, 0)to[out=60, in = -90](-1,1.5);
    \cusp
    \draw[red!70!black] (0,0) --(0,-\gap) node[pos=0.5,sloped]{\footnotesize $>$} node[right,pos=0.5]{$\gamma$} ;
    \draw[very thick] (-1,1.5) -- (1,1.5) to[out=-90, in = 120] (1.5, 0) to[out = -170, in = 120, looseness = 2] (1.45,-1.1) to [out = -160, in = 10] (0.2,-1.4) arc(0:180:0.2 and 0.3) to [out = 170, in = -20] (-1.45,-1.1) to[out = 60, in = -10, looseness = 2](-1.5, 0)to[out=60, in = -90](-1,1.5);
    \end{tikzpicture}
\end{equation*}
Applying Proposition \ref{prop:SNonD1} in this embedded $\D_1$, we can canonically put $T$ in canceling position as desired. Recall that this involves sliding strands intersecting $\gamma$ across the saddle critical point $p_{\tau}$ using the slide relation for a saddle \eqref{eq:skein-slide-1}.

We now have to check that this procedure does not depend on $T$ up to isotopy and slide relations. Sliding over $p_\tau$ is already taken care of in Proposition \ref{prop:SNonD1}, so we need to consider sliding over $p_\sigma$. The string nets before and after such a sliding give the same result by:
\begin{equation*}
    \hskip-3pt\begin{tikzpicture}[xscale = 0.6, yscale = -0.6, baseline = 0.5cm]    
    \tikzmath{\gap=1.4;}
    \draw[thick] (0,1.5) -- (0,-\gap-2.6)  node[pos=0.5,sloped]{\footnotesize $<$} node[right,pos=0.5]{$\alpha$};
    \foreach \y in {0,...,10}{
    \tikzmath{\h=(1.5+2.6+\gap)*\y/10;}
    \draw[gray!70] (-2,1.5 - \h) -- (2,1.5 - \h);
    }
    \draw[thick] (-2,-2.4) to[out = 20, in = -125, in distance = 0.9cm] node[pos=0.1, rotate = -10]{\footnotesize $<$} node[above,pos=0.1]{$x$} (-0.5,-1.6) node{\small $\bullet$} to[out = -70, in = 160] node[pos = 0.3, rectangle, draw, fill=white, thick, inner sep = 3pt]{} (2,-2.4);
    \end{tikzpicture}
    \mapsfrom
    \begin{tikzpicture}[xscale = 0.6, yscale = -0.6, baseline = 0.5cm]
    \cusp
    \draw[thick] (0,1.5) -- (0,0)  node[pos=0.5,sloped]{\footnotesize $<$} node[right,pos=0.5]{$\alpha$};
    \draw[thick] (0,-\gap) -- (0,-\gap-2.6)  node[pos=0.5,sloped]{\footnotesize $<$} node[right,pos=0.5]{$\alpha$};
    \draw[thick] (-2,-2.4) to[out = 20, in = -125, in distance = 0.9cm] node[pos=0.1, rotate = -10]{\footnotesize $<$} node[above,pos=0.1]{$x$} node[pos=0.1, rotate = -10]{\footnotesize $<$} node[above,pos=0.1]{$x$} (-0.5,-1.6) node{\small $\bullet$} to[out = -70, in = 160] node[pos = 0.3, rectangle, draw, fill=white, thick, inner sep = 3pt]{} (2,-2.4);
    \end{tikzpicture}
    \overset{\eqref{eq:skein-slide-0}}{\sim}
    \begin{tikzpicture}[xscale = 0.6, yscale = -0.6, baseline = 0.5cm]
    \cusp
    \draw[thick] (0,1.5) -- (0,0)  node[pos=0.5,sloped]{\footnotesize $<$} node[right,pos=0.5]{$\alpha$};
    \draw[thick] (0,-\gap) -- (0,-\gap-2.6)  node[pos=0.5,sloped]{\footnotesize $<$} node[right,pos=0.5]{$\alpha$};
    \draw[thick] (-2,-2.4) to[out = 20, in = 125] node[pos=0.1, rotate = -40]{\footnotesize $<$} node[above,pos=0.1]{$x$} (0,-1) node{\small $\bullet$} to[out = 55, in = 160] (2,-2.4);
    \end{tikzpicture}
    \overset{\eqref{eq:skein-slide-1}}{\sim}\hskip-5pt
    \begin{tikzpicture}[xscale = 0.6, yscale = -0.6, baseline = 0.5cm]
    \cusp
    \draw[thick] (0,1.5) -- (0,0)  node[pos=0.2,sloped]{\footnotesize $<$} node[right,pos=0.2]{$\alpha$};
    \draw[thick] (0,-\gap) -- (0,-\gap-2.6)  node[pos=0.5,sloped]{\footnotesize $<$} node[right,pos=0.5]{$\alpha$};
    \draw[thick] (-2,-2.4) to[out = 20, in = -115, out distance = 0.2cm] node[pos=0.2, rotate = -65]{\footnotesize $<$} node[right,pos=0.2]{$x$}  (-0.5,1) node{\small $\bullet$} to[out = -35, in = 160, in distance  = 0.2cm, looseness = 2] node[pos = 0.064, rectangle, draw, fill=white, thick, inner sep = 3pt, rotate = 10]{} (2,-2.4);
    \end{tikzpicture}
    \mapsto \hskip-3pt\hskip-3pt
    \begin{tikzpicture}[xscale = 0.6, yscale = -0.6, baseline = 0.5cm]    
    \tikzmath{\gap=1.4;}
    \draw[thick] (0,1.5) -- (0,-\gap-2.6)  node[pos=0.5,sloped]{\footnotesize $<$} node[right,pos=0.5]{$\alpha$};
    \foreach \y in {0,...,10}{
    \tikzmath{\h=(1.5+2.6+\gap)*\y/10;}
    \draw[gray!70] (-2,1.5 - \h) -- (2,1.5 - \h);
    }
    \draw[thick] (-2,-2.4) to[out = 20, in = -115, out distance = 0.2cm] node[pos=0.2, rotate = -65]{\footnotesize $<$} node[right,pos=0.2]{$x$}  (-0.5,1) node{\small $\bullet$} to[out = -35, in = 160, in distance  = 0.2cm, looseness = 2] node[pos = 0.064, rectangle, draw, fill=white, thick, inner sep = 3pt, rotate = 10]{} (2,-2.4);
    \end{tikzpicture}
\end{equation*}
Finally, isotopies that change intersections with $\gamma$ are taken care of in Proposition \ref{prop:SNonD1}, so it remains to check an isotopy along which the $\alpha$-strand at $p_\sigma$ passes through a point at which it is tangent to $\gamma$. This does not affect the result either by a special case of the computation above with $x = \alpha$.
\end{proof}

\begin{proposition}\label{prop:cuspbirth_and_cupsdeath_are_inverses}
    The natural transformations $\orSN_\II^\alpha(\Cd(\sigma,\tau))$ and $\orSN_\II^\alpha(\Cb(\sigma,\tau))$ are inverses to each other.
\end{proposition}

\begin{proof}
Consider first the composition $\orSN_\II^\alpha(\Cb(\sigma,\tau)) \circ \orSN_\II^\alpha(\Cd(\sigma,\tau)):\orSN_{\II}^{\alpha}(W) \to \orSN_{\II}^{\alpha}(W) $. For a string net $T\subset W$ in canceling position, this map is supported inside $V$, where it is given by
\begin{equation*}
    \begin{tikzpicture}[xscale = 0.6, yscale = -0.6, baseline = 0.5cm]
    \cusp
    \draw[thick] (0,1.5) -- (0,0)  node[pos=0.5,sloped]{\footnotesize $<$} node[right,pos=0.5]{$\alpha$};
    \draw[thick] (0,-\gap) -- (0,-\gap-2.6)  node[pos=0.5,sloped]{\footnotesize $<$} node[right,pos=0.5]{$\alpha$};
    \end{tikzpicture}\mapsto \begin{tikzpicture}[xscale=0.6,yscale=-0.6,baseline=0.5cm] 
    \cusp
    \draw[thick] (0,0) -- (0,-\gap)  node[pos=0.5,sloped]{\footnotesize $>$} node[right,pos=0.5]{$\alpha$};
    \draw[black,thick] (0,1.5) to[out=-90,in=90] (-1.8,-0.8*\gap) to[out=-90,in=90,looseness=0.7] node[pos=0,sloped]{$<$} node[pos=0,right]{$\alpha$} (0,-\gap-2.6);
    \end{tikzpicture}
\end{equation*}
Here, we move the $\alpha$-string arising from the cusp-death to the left before applying the cusp-birth. The merge relation \eqref{eq:alpha-inv-rels} for the invertibility of $\alpha$ gives the key step in the following simplifications:
\begin{equation*}
    \begin{tikzpicture}[xscale=0.6,yscale=-0.6,baseline=0.5cm] 
    \cusp
    \draw[thick] (0,0) -- (0,-\gap)  node[pos=0.5,sloped]{\footnotesize $>$} node[right,pos=0.5]{$\alpha$};
    \draw[black,thick] (0,1.5) to[out=-90,in=90] (-1.8,-0.8*\gap) to[out=-90,in=90,looseness=0.7] node[pos=0,sloped]{$<$} node[pos=0,right]{$\alpha$} (0,-\gap-2.6);
        \end{tikzpicture}
        \overset{\eqref{eq:saddle-rel}}{\sim}
       \begin{tikzpicture}[xscale = 0.6, yscale = 0.6, baseline = 0.5cm]
    \begin{scope}[yscale=-1]
       \cusp  
    \end{scope}
    \tikzmath{\gap=1.4;}
     \draw[thick] (0,0) -- (0,-0.8)  node[pos=0.5,sloped]{\footnotesize $<$} node[right,pos=0.5]{$\alpha$} node[pos=1] {\footnotesize $\bullet$};
    \draw[thick,sloped] (0,-0.8) to[out=135,in=-80] (-1.45,\gap) to[out=100,in=-120] node[pos=0]{\footnotesize $<$} node[right,pos=0,sloped=false]{$\alpha^{-1}$} node[pos=1] {\footnotesize $\bullet$} (-1.1,1+\gap) to[out=-45,in=90] node[right,pos=0.5,sloped=false]{$\alpha$} node[pos=0.5]{\footnotesize $<$} (0,\gap) ;

    \draw[black,thick] (0,-1.5) to[out=90,in=-80,looseness=0.8] (-1.8,0.8) to[out=100,in=-135] node[pos=0.5,sloped]{$>$} node[pos=0.5,left]{$\alpha$}(-1.2,\gap+1.4) to[out=45,in=-90,looseness=0.8]  (0,\gap+2.6);

    \end{tikzpicture}
    \overset{\eqref{eq:alpha-inv-rels}}{\sim}
    \begin{tikzpicture}[xscale = 0.6, yscale = 0.6, baseline = 0.5cm]
    \begin{scope}[yscale=-1]
       \cusp  
    \end{scope}
    \tikzmath{\gap=1.4;}
    \draw[thick] (0,0) to[out=-90,in=90] node[pos=1] {\footnotesize $\bullet$} node[pos=0.5,sloped]{\footnotesize $<$} node[right,pos=0.5]{$\alpha$}  (0,-0.8)  to[out=135,in=-20] node[pos=1] {\footnotesize $\bullet$}  node[pos=0.5,sloped]{\footnotesize $<$} (-0.6,-0.5)   to[out=-80,in=90] node[pos=0.5,sloped]{\footnotesize $<$} node[left,pos=0.5]{$\alpha$} (0,-1.5) ;
    \draw[thick]  (-0.85,\gap) to[out=100,in=-70] node[pos=0] {\footnotesize $\bullet$} node[pos=1] {\footnotesize $\bullet$} node[pos=0.5,sloped]{\footnotesize $>$} (-0.9,0.8+\gap) to[out=0,in=90] node[right,pos=0.5]{$\alpha$} node[pos=0.5,sloped]{\footnotesize $<$} (0,\gap) ;
    \draw[black,thick]  (-0.85,\gap) to[out=150,in=-135] node[pos=0.5,sloped]{$<$} node[pos=0.5,left]{$\alpha$}(-1.2,\gap+1.4) to[out=45,in=-90,looseness=0.8]  (0,\gap+2.6);
    \end{tikzpicture}
    \sim
    \begin{tikzpicture}[xscale = 0.6, yscale = -0.6, baseline = 0.5cm]
    \cusp
    \draw[thick] (0,1.5) -- (0,0)  node[pos=0.5,sloped]{\footnotesize $<$} node[right,pos=0.5]{$\alpha$};
    \draw[thick] (0,-\gap) -- (0,-\gap-2.6)  node[pos=0.5,sloped]{\footnotesize $<$} node[right,pos=0.5]{$\alpha$};
    \end{tikzpicture}
\end{equation*}

Consider now the composition $\orSN_\II^\alpha(\Cd(\sigma,\tau)) \circ \orSN_\II^\alpha(\Cb(\sigma,\tau)): \orSN_{\II}^{\alpha}(\BB(\phi)) \to  \orSN_{\II}^{\alpha}(\BB(\phi))$. First note that the string net $\Theta$ arising from the cusp-birth can be put into canceling position as follows:
\begin{equation*}
    \begin{tikzpicture}[xscale = 0.6, yscale = -0.6, baseline = 0.5cm]
        \cusp 
        \draw[thick] (0,-\gap) -- (0,0)  node[pos=0.5,sloped]{\footnotesize $<$} node[right,pos=0.5]{$\alpha$};
    \end{tikzpicture} \sim
    \begin{tikzpicture}[xscale=0.6,yscale=-0.6,baseline=0.5cm]
    \cusp
    \begin{scope}[yscale=-1]
    \tikzmath{\gap=1.4;}
    \draw[thick] (0,0) -- (0,-0.8)  node[pos=0.5,sloped]{\footnotesize $<$} node[left,pos=0.5]{$\alpha$} node[pos=1] {\footnotesize $\bullet$};
    \draw[thick,sloped] (0,-0.8) to[out=45,in=-90] node[pos=0.5]{\footnotesize $>$} node[right,pos=0.5,sloped=false]{$\alpha^{-1}$} node[pos=1] {\footnotesize $\bullet$} (1.7,\gap+1)  ;
    \draw[thick,sloped] (1.7,1+\gap) to[out=-135,in=90] node[pos=0.5]{\footnotesize $>$} node[below,pos=0.5,sloped=false]{$\alpha$} (0,\gap)   ;
    \end{scope}
    \end{tikzpicture}
\end{equation*}
After applying $\orSN_\II^\alpha(\Cd(\sigma,\tau))$ this is seen to be equivalent to the identity by applying the eye relation \eqref{eq:alpha-inv-rels} for the invertibility of $\alpha$.
\end{proof}

\subsubsection{Crossings}\label{subsubsec:crossings}
Let $\sigma,\rho,\sigma',\rho'$ be two pairs of saddles satisfying the crossing condition. We use the notations of the crossing paragraph in Section \ref{subsubsec:bord_realization}. The diffeomorphism $\Phi:=\BRF(\XX(\sigma,\rho,\sigma',\rho')):W \to W' $ induced by the crossing of two saddles 
is foliation-preserving outside a neighborhood $U$ of the two critical points, and even outside a slightly smaller neighborhood $V$, where is is given by:
\begin{equation}\label{eq:defUandVcrossing}
    \begin{tikzpicture}[baseline = 0pt, scale = 1.34]
        \begin{scope}
        \clip (0,0) circle(1.2);
    \draw[gray!70] (1,1)--(-1,-1);
    \draw[gray!70] (1,-1)--(-1,1);
    \def\sadlines{8}
    \foreach \y in {5,...,\sadlines}{
    \draw[gray!70] ({-\y/(\sadlines+1)},1) .. controls (0,{1.6-1.6*\y/(\sadlines+1)}) .. ({\y/(\sadlines+1)},1);
    \draw[gray!70] ({-\y/(\sadlines+1)},-1) .. controls (0,{-1.6+1.6*\y/(\sadlines+1)}) .. ({\y/(\sadlines+1)},-1);
    \draw[gray!70] (-1,{-\y/(\sadlines+1)}) .. controls ({-1.6+1.6*\y/(\sadlines+1)},0) .. (-1,{\y/(\sadlines+1)});
    \draw[gray!70] (1,{-\y/(\sadlines+1)}) .. controls ({1.6-1.6*\y/(\sadlines+1)},0) .. (1,{\y/(\sadlines+1)});
    }
    \fill[white] (0,0) circle(0.62);
    \draw[gray!70] (21:0.62) to[out = -120, in = 120, looseness = 1] (-21:0.62);
    \draw[gray!70] (201:0.62) to[out = 60, in = -60, looseness = 1] (-201:0.62);
    \draw[gray!70] (33:0.62) to[out = -128, in = 52, looseness = 0.85] (-147:0.62);
    \draw[gray!70] (-33:0.62) to[out = 128, in = -52, looseness = 0.85] (147:0.62);
    \draw[gray!70] (45:0.62) to[out = -135, in = -45] (135:0.62);
    \draw[gray!70] (-45:0.62) to[out = 135, in = 45] (-135:0.62);
    \def\sadlines{8}
    \foreach \y in {7,...,\sadlines}{
    \draw[gray!70] ({-\y/(\sadlines+1)},1) .. controls (0,{1.6-1.6*\y/(\sadlines+1)}) .. ({\y/(\sadlines+1)},1);
    \draw[gray!70] ({-\y/(\sadlines+1)},-1) .. controls (0,{-1.6+1.6*\y/(\sadlines+1)}) .. ({\y/(\sadlines+1)},-1);
    } 
    \end{scope}
    \draw[gray, ->](-0.8,-1)--++(-0.2,0.2) node[midway, below left = -2pt]{\small $\vec{y}$};
    \draw[gray, ->](0.8,-1)--++(0.2,0.2) node[midway, below right = -2pt]{\small $\vec{y}$};
    \node[star,  star point ratio=2.25, fill=black, inner sep = 1pt] at (0,0) {};
    \draw[red!70!black] (0.62,0) arc(0:360:0.62);
\end{tikzpicture}
\hskip-5pt \sqcup \hskip-5pt
    \begin{tikzpicture}[baseline = 0pt, scale = 1.34]
        \begin{scope}[rotate = 90]
        \clip (0,0) circle(1.2);
    \draw[gray!70] (1,1)--(-1,-1);
    \draw[gray!70] (1,-1)--(-1,1);
    \def\sadlines{8}
    \foreach \y in {5,...,\sadlines}{
    \draw[gray!70] ({-\y/(\sadlines+1)},1) .. controls (0,{1.6-1.6*\y/(\sadlines+1)}) .. ({\y/(\sadlines+1)},1);
    \draw[gray!70] ({-\y/(\sadlines+1)},-1) .. controls (0,{-1.6+1.6*\y/(\sadlines+1)}) .. ({\y/(\sadlines+1)},-1);
    \draw[gray!70] (-1,{-\y/(\sadlines+1)}) .. controls ({-1.6+1.6*\y/(\sadlines+1)},0) .. (-1,{\y/(\sadlines+1)});
    \draw[gray!70] (1,{-\y/(\sadlines+1)}) .. controls ({1.6-1.6*\y/(\sadlines+1)},0) .. (1,{\y/(\sadlines+1)});
    }
    \fill[white] (0,0) circle(0.62);
    \draw[gray!70] (21:0.62) to[out = -120, in = 120, looseness = 1] (-21:0.62);
    \draw[gray!70] (201:0.62) to[out = 60, in = -60, looseness = 1] (-201:0.62);
    \draw[gray!70] (33:0.62) to[out = -128, in = 52, looseness = 0.85] (-147:0.62);
    \draw[gray!70] (-33:0.62) to[out = 128, in = -52, looseness = 0.85] (147:0.62);
    \draw[gray!70] (45:0.62) to[out = -135, in = -45] (135:0.62);
    \draw[gray!70] (-45:0.62) to[out = 135, in = 45] (-135:0.62);
    \def\sadlines{8}
    \foreach \y in {7,...,\sadlines}{
    \draw[gray!70] ({-\y/(\sadlines+1)},1) .. controls (0,{1.6-1.6*\y/(\sadlines+1)}) .. ({\y/(\sadlines+1)},1);
    \draw[gray!70] ({-\y/(\sadlines+1)},-1) .. controls (0,{-1.6+1.6*\y/(\sadlines+1)}) .. ({\y/(\sadlines+1)},-1);
    } 
    \end{scope}
    \draw[gray, ->](-0.8,-1)--++(-0.2,0.2) node[midway, below left = -2pt]{\small $\vec{y}$};
    \draw[gray, ->](0.8,-1)--++(0.2,0.2) node[midway, below right = -2pt]{\small $\vec{y}$};
    \node[star,  star point ratio=2.25, fill=black, inner sep = 1pt] at (0,0) {};
    \draw[red!70!black] (0.62,0) arc(0:360:0.62) node[below right]{$V$};
    \node[gray] at (1,0.9) {$U$};
\end{tikzpicture}
\overset{\Phi}{\longrightarrow}
    \begin{tikzpicture}[baseline = 0pt, scale = 1.34]
        \begin{scope}[rotate = 90]
        \clip (0,0) circle(1.2);
    \draw[gray!70] (1,1)--(-1,-1);
    \draw[gray!70] (1,-1)--(-1,1);
    \def\sadlines{8}
    \foreach \y in {5,...,\sadlines}{
    \draw[gray!70] ({-\y/(\sadlines+1)},1) .. controls (0,{1.6-1.6*\y/(\sadlines+1)}) .. ({\y/(\sadlines+1)},1);
    \draw[gray!70] ({-\y/(\sadlines+1)},-1) .. controls (0,{-1.6+1.6*\y/(\sadlines+1)}) .. ({\y/(\sadlines+1)},-1);
    \draw[gray!70] (-1,{-\y/(\sadlines+1)}) .. controls ({-1.6+1.6*\y/(\sadlines+1)},0) .. (-1,{\y/(\sadlines+1)});
    \draw[gray!70] (1,{-\y/(\sadlines+1)}) .. controls ({1.6-1.6*\y/(\sadlines+1)},0) .. (1,{\y/(\sadlines+1)});
    }
    \fill[white] (0,0) circle(0.62);
    \draw[gray!70] (21:0.62) to[out = -120, in = 120, looseness = 1] (-21:0.62);
    \draw[gray!70] (201:0.62) to[out = 60, in = -60, looseness = 1] (-201:0.62);
    \draw[gray!70] (33:0.62) to[out = -128, in = 52, looseness = 0.85] (-147:0.62);
    \draw[gray!70] (-33:0.62) to[out = 128, in = -52, looseness = 0.85] (147:0.62);
    \draw[gray!70] (45:0.62) to[out = -135, in = -45] (135:0.62);
    \draw[gray!70] (-45:0.62) to[out = 135, in = 45] (-135:0.62);
    \def\sadlines{8}
    \foreach \y in {7,...,\sadlines}{
    \draw[gray!70] ({-\y/(\sadlines+1)},1) .. controls (0,{1.6-1.6*\y/(\sadlines+1)}) .. ({\y/(\sadlines+1)},1);
    \draw[gray!70] ({-\y/(\sadlines+1)},-1) .. controls (0,{-1.6+1.6*\y/(\sadlines+1)}) .. ({\y/(\sadlines+1)},-1);
    } 
    \end{scope}
    \draw[gray, ->](-0.8,-1)--++(-0.2,0.2) node[midway, below left = -2pt]{\small $\vec{y}$};
    \draw[gray, ->](0.8,-1)--++(0.2,0.2) node[midway, below right = -2pt]{\small $\vec{y}$};
    \node[star,  star point ratio=2.25, fill=black, inner sep = 1pt] at (0,0) {};
    \draw[red!70!black] (0.62,0) arc(0:360:0.62);
\end{tikzpicture}
\hskip-5pt \sqcup \hskip-5pt
\begin{tikzpicture}[baseline = 0pt, scale = 1.34]
        \begin{scope}
        \clip (0,0) circle(1.2);
    \draw[gray!70] (1,1)--(-1,-1);
    \draw[gray!70] (1,-1)--(-1,1);
    \def\sadlines{8}
    \foreach \y in {5,...,\sadlines}{
    \draw[gray!70] ({-\y/(\sadlines+1)},1) .. controls (0,{1.6-1.6*\y/(\sadlines+1)}) .. ({\y/(\sadlines+1)},1);
    \draw[gray!70] ({-\y/(\sadlines+1)},-1) .. controls (0,{-1.6+1.6*\y/(\sadlines+1)}) .. ({\y/(\sadlines+1)},-1);
    \draw[gray!70] (-1,{-\y/(\sadlines+1)}) .. controls ({-1.6+1.6*\y/(\sadlines+1)},0) .. (-1,{\y/(\sadlines+1)});
    \draw[gray!70] (1,{-\y/(\sadlines+1)}) .. controls ({1.6-1.6*\y/(\sadlines+1)},0) .. (1,{\y/(\sadlines+1)});
    }
    \fill[white] (0,0) circle(0.62);
    \draw[gray!70] (21:0.62) to[out = -120, in = 120, looseness = 1] (-21:0.62);
    \draw[gray!70] (201:0.62) to[out = 60, in = -60, looseness = 1] (-201:0.62);
    \draw[gray!70] (33:0.62) to[out = -128, in = 52, looseness = 0.85] (-147:0.62);
    \draw[gray!70] (-33:0.62) to[out = 128, in = -52, looseness = 0.85] (147:0.62);
    \draw[gray!70] (45:0.62) to[out = -135, in = -45] (135:0.62);
    \draw[gray!70] (-45:0.62) to[out = 135, in = 45] (-135:0.62);
    \def\sadlines{8}
    \foreach \y in {7,...,\sadlines}{
    \draw[gray!70] ({-\y/(\sadlines+1)},1) .. controls (0,{1.6-1.6*\y/(\sadlines+1)}) .. ({\y/(\sadlines+1)},1);
    \draw[gray!70] ({-\y/(\sadlines+1)},-1) .. controls (0,{-1.6+1.6*\y/(\sadlines+1)}) .. ({\y/(\sadlines+1)},-1);
    } 
    \end{scope}
    \draw[gray, ->](-0.8,-1)--++(-0.2,0.2) node[midway, below left = -2pt]{\small $\vec{y}$};
    \draw[gray, ->](0.8,-1)--++(0.2,0.2) node[midway, below right = -2pt]{\small $\vec{y}$};
    \node[star,  star point ratio=2.25, fill=black, inner sep = 1pt] at (0,0) {};
    \draw[red!70!black] (0.62,0) arc(0:360:0.62) node[below right]{$V'$};
    \node[gray] at (1,0.9) {$U'$};
\end{tikzpicture}
\end{equation}

\begin{definition}
    A string net $T \subseteq W$ is said to be \textit{in crossing position} if $T\cap V$ is equal to $\Lambda \subseteq V$ where
    $$\Lambda :=     \begin{tikzpicture}[baseline = -5pt, scale = 1.5]
        \clip (0,0) circle(0.62);
    \draw[gray!70] (21:0.62) to[out = -120, in = 120, looseness = 1] (-21:0.62);
    \draw[gray!70] (201:0.62) to[out = 60, in = -60, looseness = 1] (-201:0.62);
    \draw[gray!70] (33:0.62) to[out = -128, in = 52, looseness = 0.85] (-147:0.62);
    \draw[gray!70] (-33:0.62) to[out = 128, in = -52, looseness = 0.85] (147:0.62);
    \draw[gray!70] (45:0.62) to[out = -135, in = -45] (135:0.62);
    \draw[gray!70] (-45:0.62) to[out = 135, in = 45] (-135:0.62);
    \def\sadlines{8}
    \foreach \y in {7,...,\sadlines}{
    \draw[gray!70] ({-\y/(\sadlines+1)},1) .. controls (0,{1.6-1.6*\y/(\sadlines+1)}) .. ({\y/(\sadlines+1)},1);
    \draw[gray!70] ({-\y/(\sadlines+1)},-1) .. controls (0,{-1.6+1.6*\y/(\sadlines+1)}) .. ({\y/(\sadlines+1)},-1);
    }
    \node[star,  star point ratio=2.25, fill=black, inner sep = 1pt] at (0,0) {};
    \draw[very thick] (0,-0.62) -- (0,0) node[midway, sloped]{\footnotesize $>$} node[midway, right]{$\alpha$};
\end{tikzpicture}
\sqcup
\begin{tikzpicture}[baseline = -5pt, scale = 1.5]
    \begin{scope}[rotate = 90]
        \clip (0,0) circle(0.62);
    \draw[gray!70] (21:0.62) to[out = -120, in = 120, looseness = 1] (-21:0.62);
    \draw[gray!70] (201:0.62) to[out = 60, in = -60, looseness = 1] (-201:0.62);
    \draw[gray!70] (33:0.62) to[out = -128, in = 52, looseness = 0.85] (-147:0.62);
    \draw[gray!70] (-33:0.62) to[out = 128, in = -52, looseness = 0.85] (147:0.62);
    \draw[gray!70] (45:0.62) to[out = -135, in = -45] (135:0.62);
    \draw[gray!70] (-45:0.62) to[out = 135, in = 45] (-135:0.62);
    \def\sadlines{8}
    \foreach \y in {7,...,\sadlines}{
    \draw[gray!70] ({-\y/(\sadlines+1)},1) .. controls (0,{1.6-1.6*\y/(\sadlines+1)}) .. ({\y/(\sadlines+1)},1);
    \draw[gray!70] ({-\y/(\sadlines+1)},-1) .. controls (0,{-1.6+1.6*\y/(\sadlines+1)}) .. ({\y/(\sadlines+1)},-1);
    }        
    \end{scope}
    \node[star,  star point ratio=2.25, fill=black, inner sep = 1pt] at (0,0) {};
    \draw[very thick] (0,-0.62) -- (0,0) node[midway, sloped]{\footnotesize $>$} node[midway, right]{$\alpha$};
\end{tikzpicture}$$
    In this case, we denote $T_{\smallsetminus V} = T\cap (W\smallsetminus V)$, so that $T = T_{\smallsetminus V}\cup \Lambda$.
\end{definition}
\begin{proposition}
    There exists a unique natural transformation
    $$\orSN_\II^\alpha(\XX(\sigma,\tau,\sigma',\tau')) :\SN_\II^\alpha(W) \to \SN_\II^\alpha(W')$$
    which maps a string net $T \subseteq W$ in crossing position to the string net $\Phi(T_{\smallsetminus V}) \cup \Lambda'  \subseteq W'$ where $\Lambda'\subseteq V'$ is the $\alpha$-colored lines shown below:
    $$T_{\smallsetminus V}\cup \left(\begin{tikzpicture}[baseline = -5pt, scale = 1.5]
        \clip (0,0) circle(0.62);
    \draw[gray!70] (21:0.62) to[out = -120, in = 120, looseness = 1] (-21:0.62);
    \draw[gray!70] (201:0.62) to[out = 60, in = -60, looseness = 1] (-201:0.62);
    \draw[gray!70] (33:0.62) to[out = -128, in = 52, looseness = 0.85] (-147:0.62);
    \draw[gray!70] (-33:0.62) to[out = 128, in = -52, looseness = 0.85] (147:0.62);
    \draw[gray!70] (45:0.62) to[out = -135, in = -45] (135:0.62);
    \draw[gray!70] (-45:0.62) to[out = 135, in = 45] (-135:0.62);
    \def\sadlines{8}
    \foreach \y in {7,...,\sadlines}{
    \draw[gray!70] ({-\y/(\sadlines+1)},1) .. controls (0,{1.6-1.6*\y/(\sadlines+1)}) .. ({\y/(\sadlines+1)},1);
    \draw[gray!70] ({-\y/(\sadlines+1)},-1) .. controls (0,{-1.6+1.6*\y/(\sadlines+1)}) .. ({\y/(\sadlines+1)},-1);
    }
    \node[star,  star point ratio=2.25, fill=black, inner sep = 1pt] at (0,0) {};
    \draw[very thick] (0,-0.62) -- (0,0) node[midway, sloped]{\footnotesize $>$} node[midway, right]{$\alpha$};
\end{tikzpicture}
\sqcup
\begin{tikzpicture}[baseline = -5pt, scale = 1.5]
    \begin{scope}[rotate = 90]
        \clip (0,0) circle(0.62);
    \draw[gray!70] (21:0.62) to[out = -120, in = 120, looseness = 1] (-21:0.62);
    \draw[gray!70] (201:0.62) to[out = 60, in = -60, looseness = 1] (-201:0.62);
    \draw[gray!70] (33:0.62) to[out = -128, in = 52, looseness = 0.85] (-147:0.62);
    \draw[gray!70] (-33:0.62) to[out = 128, in = -52, looseness = 0.85] (147:0.62);
    \draw[gray!70] (45:0.62) to[out = -135, in = -45] (135:0.62);
    \draw[gray!70] (-45:0.62) to[out = 135, in = 45] (-135:0.62);
    \def\sadlines{8}
    \foreach \y in {7,...,\sadlines}{
    \draw[gray!70] ({-\y/(\sadlines+1)},1) .. controls (0,{1.6-1.6*\y/(\sadlines+1)}) .. ({\y/(\sadlines+1)},1);
    \draw[gray!70] ({-\y/(\sadlines+1)},-1) .. controls (0,{-1.6+1.6*\y/(\sadlines+1)}) .. ({\y/(\sadlines+1)},-1);
    }        
    \end{scope}
    \node[star,  star point ratio=2.25, fill=black, inner sep = 1pt] at (0,0) {};
    \draw[very thick] (0,-0.62) -- (0,0) node[midway, sloped]{\footnotesize $>$} node[midway, right]{$\alpha$};
\end{tikzpicture}\right)
    \longmapsto \Phi(T_{\smallsetminus V})\cup \left( \begin{tikzpicture}[baseline = -5pt, scale = 1.5]
    \begin{scope}[rotate = 90]
        \clip (0,0) circle(0.62);
    \draw[gray!70] (21:0.62) to[out = -120, in = 120, looseness = 1] (-21:0.62);
    \draw[gray!70] (201:0.62) to[out = 60, in = -60, looseness = 1] (-201:0.62);
    \draw[gray!70] (33:0.62) to[out = -128, in = 52, looseness = 0.85] (-147:0.62);
    \draw[gray!70] (-33:0.62) to[out = 128, in = -52, looseness = 0.85] (147:0.62);
    \draw[gray!70] (45:0.62) to[out = -135, in = -45] (135:0.62);
    \draw[gray!70] (-45:0.62) to[out = 135, in = 45] (-135:0.62);
    \def\sadlines{8}
    \foreach \y in {7,...,\sadlines}{
    \draw[gray!70] ({-\y/(\sadlines+1)},1) .. controls (0,{1.6-1.6*\y/(\sadlines+1)}) .. ({\y/(\sadlines+1)},1);
    \draw[gray!70] ({-\y/(\sadlines+1)},-1) .. controls (0,{-1.6+1.6*\y/(\sadlines+1)}) .. ({\y/(\sadlines+1)},-1);
    }        
    \end{scope}
    \node[star,  star point ratio=2.25, fill=black, inner sep = 1pt] at (0,0) {};
    \draw[very thick] (0,-0.62) -- (0,0) node[midway, sloped]{\footnotesize $>$} node[midway, right]{$\alpha$};
\end{tikzpicture}
\sqcup
\begin{tikzpicture}[baseline = -5pt, scale = 1.5]
        \clip (0,0) circle(0.62);
    \draw[gray!70] (21:0.62) to[out = -120, in = 120, looseness = 1] (-21:0.62);
    \draw[gray!70] (201:0.62) to[out = 60, in = -60, looseness = 1] (-201:0.62);
    \draw[gray!70] (33:0.62) to[out = -128, in = 52, looseness = 0.85] (-147:0.62);
    \draw[gray!70] (-33:0.62) to[out = 128, in = -52, looseness = 0.85] (147:0.62);
    \draw[gray!70] (45:0.62) to[out = -135, in = -45] (135:0.62);
    \draw[gray!70] (-45:0.62) to[out = 135, in = 45] (-135:0.62);
    \def\sadlines{8}
    \foreach \y in {7,...,\sadlines}{
    \draw[gray!70] ({-\y/(\sadlines+1)},1) .. controls (0,{1.6-1.6*\y/(\sadlines+1)}) .. ({\y/(\sadlines+1)},1);
    \draw[gray!70] ({-\y/(\sadlines+1)},-1) .. controls (0,{-1.6+1.6*\y/(\sadlines+1)}) .. ({\y/(\sadlines+1)},-1);
    }
    \node[star,  star point ratio=2.25, fill=black, inner sep = 1pt] at (0,0) {};
    \draw[very thick] (0,-0.62) -- (0,0) node[midway, sloped]{\footnotesize $>$} node[midway, right]{$\alpha$};
\end{tikzpicture}\right)$$
\end{proposition}
\begin{proof}
Consider a string net $T \subseteq W$. Using \eqref{eq:saddle-rel} we can ensure that the $\alpha$-strands at the critical points come from below. Up to to isotopy and snake relations, we can ensure that $T$ is otherwise disjoint from $V$, hence is in crossing position. 

Note that snake relations might indeed be needed, for example the strand below cannot be isotoped away from $V$ without a snake relation, as it always needs to stay ``above" the foliation line from which it starts.
\begin{equation*}
    \begin{tikzpicture}[baseline = 0pt, scale = 1]
        \begin{scope}
        \clip (0,0) circle(1.2);
    \draw[gray!70] (1,1)--(-1,-1);
    \draw[gray!70] (1,-1)--(-1,1);
    \def\sadlines{8}
    \foreach \y in {5,...,\sadlines}{
    \draw[gray!70] ({-\y/(\sadlines+1)},1) .. controls (0,{1.6-1.6*\y/(\sadlines+1)}) .. ({\y/(\sadlines+1)},1);
    \draw[gray!70] ({-\y/(\sadlines+1)},-1) .. controls (0,{-1.6+1.6*\y/(\sadlines+1)}) .. ({\y/(\sadlines+1)},-1);
    \draw[gray!70] (-1,{-\y/(\sadlines+1)}) .. controls ({-1.6+1.6*\y/(\sadlines+1)},0) .. (-1,{\y/(\sadlines+1)});
    \draw[gray!70] (1,{-\y/(\sadlines+1)}) .. controls ({1.6-1.6*\y/(\sadlines+1)},0) .. (1,{\y/(\sadlines+1)});
    }
    \fill[white] (0,0) circle(0.62);
    \draw[gray!70] (21:0.62) to[out = -120, in = 120, looseness = 1] (-21:0.62);
    \draw[gray!70] (201:0.62) to[out = 60, in = -60, looseness = 1] (-201:0.62);
    \draw[gray!70] (33:0.62) to[out = -128, in = 52, looseness = 0.85] (-147:0.62);
    \draw[gray!70] (-33:0.62) to[out = 128, in = -52, looseness = 0.85] (147:0.62);
    \draw[gray!70] (45:0.62) to[out = -135, in = -45] (135:0.62);
    \draw[gray!70] (-45:0.62) to[out = 135, in = 45] (-135:0.62);
    \def\sadlines{8}
    \foreach \y in {7,...,\sadlines}{
    \draw[gray!70] ({-\y/(\sadlines+1)},1) .. controls (0,{1.6-1.6*\y/(\sadlines+1)}) .. ({\y/(\sadlines+1)},1);
    \draw[gray!70] ({-\y/(\sadlines+1)},-1) .. controls (0,{-1.6+1.6*\y/(\sadlines+1)}) .. ({\y/(\sadlines+1)},-1);
    } 
    \end{scope}
    \draw[gray, ->](-0.8,-1)--++(-0.2,0.2) node[midway, below left = -2pt]{\small $\vec{y}$};
    \draw[gray, ->](0.8,-1)--++(0.2,0.2) node[midway, below right = -2pt]{\small $\vec{y}$};
    \node[star,  star point ratio=2.25, fill=black, inner sep = 1pt] at (0,0) {};
    \draw[red!70!black] (0.62,0) arc(0:360:0.62);
    \draw[thick] (-132:1.2) to[out = 47, in = 180, out distance = 1.2cm]node[pos = 0.8, sloped]{\footnotesize $>$} (-20:1);
\end{tikzpicture}
\sim
    \begin{tikzpicture}[baseline = 0pt, scale = 1]
        \begin{scope}
        \clip (0,0) circle(1.2);
    \draw[gray!70] (1,1)--(-1,-1);
    \draw[gray!70] (1,-1)--(-1,1);
    \def\sadlines{8}
    \foreach \y in {5,...,\sadlines}{
    \draw[gray!70] ({-\y/(\sadlines+1)},1) .. controls (0,{1.6-1.6*\y/(\sadlines+1)}) .. ({\y/(\sadlines+1)},1);
    \draw[gray!70] ({-\y/(\sadlines+1)},-1) .. controls (0,{-1.6+1.6*\y/(\sadlines+1)}) .. ({\y/(\sadlines+1)},-1);
    \draw[gray!70] (-1,{-\y/(\sadlines+1)}) .. controls ({-1.6+1.6*\y/(\sadlines+1)},0) .. (-1,{\y/(\sadlines+1)});
    \draw[gray!70] (1,{-\y/(\sadlines+1)}) .. controls ({1.6-1.6*\y/(\sadlines+1)},0) .. (1,{\y/(\sadlines+1)});
    }
    \fill[white] (0,0) circle(0.62);
    \draw[gray!70] (21:0.62) to[out = -120, in = 120, looseness = 1] (-21:0.62);
    \draw[gray!70] (201:0.62) to[out = 60, in = -60, looseness = 1] (-201:0.62);
    \draw[gray!70] (33:0.62) to[out = -128, in = 52, looseness = 0.85] (-147:0.62);
    \draw[gray!70] (-33:0.62) to[out = 128, in = -52, looseness = 0.85] (147:0.62);
    \draw[gray!70] (45:0.62) to[out = -135, in = -45] (135:0.62);
    \draw[gray!70] (-45:0.62) to[out = 135, in = 45] (-135:0.62);
    \def\sadlines{8}
    \foreach \y in {7,...,\sadlines}{
    \draw[gray!70] ({-\y/(\sadlines+1)},1) .. controls (0,{1.6-1.6*\y/(\sadlines+1)}) .. ({\y/(\sadlines+1)},1);
    \draw[gray!70] ({-\y/(\sadlines+1)},-1) .. controls (0,{-1.6+1.6*\y/(\sadlines+1)}) .. ({\y/(\sadlines+1)},-1);
    } 
    \end{scope}
    \draw[gray, ->](-0.8,-1)--++(-0.2,0.2) node[midway, below left = -2pt]{\small $\vec{y}$};
    \draw[gray, ->](0.8,-1)--++(0.2,0.2) node[midway, below right = -2pt]{\small $\vec{y}$};
    \node[star,  star point ratio=2.25, fill=black, inner sep = 1pt] at (0,0) {};
    \draw[red!70!black] (0.62,0) arc(0:360:0.62);
    \draw[thick] (-132:1.2) to[out = 47, in = -136] (-135:0.7)  node{\small $\bullet$} to[out = -50, in = 140] (0,-0.73) node{\small $\bullet$} to[out = 20, in = 180]   node[pos = 0.8, sloped]{\footnotesize $>$} (-20:1);
\end{tikzpicture}
\end{equation*}

The result only depends on $T$ up to \eqref{eq:skein-slide-1} and does not depend on how we applied \eqref{eq:saddle-rel} because these can be applied on the right hand side too.
\end{proof}

\subsection{Relations and the main result}\label{subsec:relations+result}
We now prove that the assignment $\orSN_{\II}^{\alpha}$ on the generating $2$-morphisms 
of $\on{Sur}_{12\sim}^{\on{or}}$ given in Section~\ref{subsec:generators} determines a 
symmetric monoidal functor
\begin{align*}
    \on{FSur}_{12\sim}^{\on{or}} \to \Bimod^{hop}.
\end{align*}
This amounts to showing that $\orSN_{\II}^{\alpha}$ satisfies the relations of 
$\on{Sur}_{12\sim}^{\on{or}}$ given in Sections~3.2.1 and~3.3 of~\cite{Filippos}. 
There are three kinds of relations in $\on{Sur}_{12\sim}^{\on{or}}$, called the 
$\mathcal{F}^0$-, $\mathcal{F}^1$-, and $\mathcal{F}^2$-relations. 
We check all three below, recalling them when necessary, and thus complete the proof 
of Theorem~\ref{thm:main_thm}.

\subsubsection{$\mathcal{F}^0$-relations}
The $\mathcal{F}^0$-relations, appearing in Sections~3.2.1 and~3.3.1 
of~\cite{Filippos}, are induced by Morse-foliation-preserving isotopies between 
Morse-foliation-preserving diffeomorphisms, and hence are automatically satisfied 
by our assignments.

\subsubsection{$\mathcal{F}^1$-relations}
The $\mathcal{F}^1$-relations are those appearing in Section~3.3.2 of~\cite{Filippos}. 
Using the terminology there, these are the eye, merge-unmerge, Reidemeister~II, 
cusp-isotopy, crossing-isotopy, naturality-beak, naturality-Reidemeister~III, frame change-cusp, and frame change-crossing relations. We recall (some of) these below and verify that our proposed assignment satisfies them.
\begin{itemize}
    \item \textbf{The eye and merge-unmerge relations}

    These express the fact that the cusp-birth $\Cb(\sigma,\tau)$ and the cusp-death $\Cd(\sigma,\tau)$ are inverses of each other:
    \[
    \begin{tikzpicture}[scale=0.45]
    
    
        \begin{scope}[yshift=-10cm]
            \draw[ultra thick] (0,0) -- (0,4);
            \draw[ultra thick] (6,0) -- (6,4);
            \draw[red, ultra thick] (4.5,2) parabola (3.7,1.2);
            \draw[red, ultra thick] (4.5,2) parabola (3.7,2.8);
            \draw[red, ultra thick] (3,3) .. controls (3.5,3) and (3.6,2.9) .. (3.7,2.8);
            \draw[red, ultra thick] (3,1) .. controls (3.5,1) and (3.6,1.1) .. (3.7,1.2);
            \node[red] at (3, 3.4) {$\tau$};
            \node[red] at (3, 0.6) {$\sigma$};
            \draw[blue, ultra thick] (0,2) -- (1.5,2);
            \draw[blue, ultra thick] (4.5,2) -- (6,2);
            \node[blue] at (0.75, 2.6) {$\phi$};
            \node[blue] at (5.25, 2.6) {$\phi$};
            \draw[red, ultra thick] (1.5,2) parabola (2.3,1.2);
            \draw[red, ultra thick] (1.5,2) parabola (2.3,2.8);
            \draw[red, ultra thick] (3,3) .. controls (2.5,3) and (2.4,2.9) .. (2.3,2.8);
            \draw[red, ultra thick] (3,1) .. controls (2.5,1) and (2.4,1.1) .. (2.3,1.2);
        \end{scope}
    
        \node at (7,-8) {$=$};
    
        \begin{scope}[xshift=8cm, yshift=-10cm]
            \draw[ultra thick] (0,0) -- (0,4);
            \draw[ultra thick] (6,0) -- (6,4);
            \draw[blue, ultra thick] (0,2) -- (6,2);
            \node[blue] at (3, 2.5) {$\phi$};
        \end{scope}
    
        \node at (7, -11.5) {\small Eye};
    
    
        \begin{scope}[xshift=17cm, yshift=-10cm]
            \draw[ultra thick] (0,0) -- (0,4);
            \draw[ultra thick] (6,0) -- (6,4);
            \draw[red, ultra thick] (1.5,2) parabola (0.7,1.2);
            \draw[red, ultra thick] (1.5,2) parabola (0.7,2.8);
            \draw[red, ultra thick] (0,3) .. controls (0.5,3) and (0.6,2.9) .. (0.7,2.8);
            \draw[red, ultra thick] (0,1) .. controls (0.5,1) and (0.6,1.1) .. (0.7,1.2);
            \node[red] at (0.6, 0.6) {$\sigma$};
            \node[red] at (0.6, 3.4) {$\tau$};
            \draw[blue, ultra thick] (1.5,2) -- (4.5,2);
            \node[blue] at (3, 2.5) {$\phi$};
            \draw[red, ultra thick] (4.5,2) parabola (5.3,1.2);
            \draw[red, ultra thick] (4.5,2) parabola (5.3,2.8);
            \draw[red, ultra thick] (6,3) .. controls (5.5,3) and (5.4,2.9) .. (5.3,2.8);
            \draw[red, ultra thick] (6,1) .. controls (5.5,1) and (5.4,1.1) .. (5.3,1.2);
            \node[red] at (5.4, 0.6) {$\sigma$};
            \node[red] at (5.4, 3.4) {$\tau$};
        \end{scope}
    
        \node at (24,-8) {$=$};
    
        \begin{scope}[xshift=25cm, yshift=-10cm]
            \draw[ultra thick] (0,0) -- (0,4);
            \draw[ultra thick] (6,0) -- (6,4);
            \draw[red, ultra thick] (0,1) -- (6,1);
            \draw[red, ultra thick] (0,3) -- (6,3);
            \node[red] at (3, 0.6) {$\sigma$};
            \node[red] at (3, 3.4) {$\tau$};
        \end{scope}
    
        \node at (24, -11.5) {\small Merge-unmerge};
    
    \end{tikzpicture}
    \]
    $\orSN_{\II}^{\alpha}$ satisfies these relations because the natural 
    transformations $\orSN_\II^\alpha(\Cd(\sigma,\tau))$ and 
    $\orSN_\II^\alpha(\Cb(\sigma,\tau))$ are inverses of each other, as established 
    in Proposition~\ref{prop:cuspbirth_and_cupsdeath_are_inverses}.

    \item \textbf{The Reidemeister~II relations}

     These express the fact that the crossing $\XX(\sigma,\rho, \sigma', \rho')$ and the crossing $\XX(\rho',\sigma', \rho, \sigma)$ are inverses of each other:
    \[\begin{tikzpicture}[scale=0.5]

    \begin{scope}[yshift = -10cm]
    \draw[ultra thick] (0,0) -- (0,4);
    \draw[ultra thick] (6,0) -- (6,4);

    \draw[red, ultra thick] (0,1) -- (3,3);
    \draw[red, ultra thick] (3,3) -- (6,1);
    \draw[red, ultra thick] (0,3) -- (3,1);
    \draw[red, ultra thick] (3,1) -- (6,3);

    \node[red] at (0.6, 0.6) {$\sigma$};
    \node[red] at (0.6, 3.4) {$\rho$};

    \node[red] at (3, 0.6) {$\rho'$};
    \node[red] at (3, 3.4) {$\sigma'$};

    \node[red] at (5.4, 0.6) {$\sigma$};
    \node[red] at (5.4, 3.4) {$\rho$};

    \end{scope}

    \node at (7,-8) {$=$};

    \begin{scope}[xshift=8cm, yshift = -10cm]
    \draw[ultra thick] (0,0) -- (0,4);
    \draw[ultra thick] (6,0) -- (6,4);
    \draw[red, ultra thick] (0,1) -- (6,1);
    \draw[red, ultra thick] (0,3) -- (6,3);

    \node[red] at (3, 0.6) {$\sigma$};
    \node[red] at (3, 3.4) {$\rho$};
    
    \end{scope}

    \end{tikzpicture}\]
    $\orSN_{\II}^{\alpha}$ satisfies these relations because the natural 
    transformations $\orSN_\II^\alpha(\XX(\sigma,\rho, \sigma', \rho'))$ and 
    $\orSN_\II^\alpha(\XX(\rho',\sigma', \rho, \sigma))$ are inverses of each other, 
    which follows immediately from the construction of the assignment on crossings 
    in Section~\ref{subsubsec:crossings}.
    
    \item \textbf{The cusp-isotopy relations}

    These express the observation that isotopies keeping the surgery spheres in 
    cancelling position can be ``pushed through the cusp'':
    \[\begin{tikzpicture}[scale=0.5]

    \begin{scope}[yshift = -10cm]
    \draw[ultra thick] (0,0) -- (0,4);
    \draw[ultra thick] (6,0) -- (6,4);

     \draw[red, ultra thick] (4,2) parabola (3.2,1.2);
    \draw[red, ultra thick] (4,2) parabola (3.2,2.8);
    \draw[red, ultra thick] (2.5,3) .. controls (3,3) and (3.1,2.9) .. (3.2,2.8);
    \draw[red, ultra thick] (2.5,1) .. controls (3,1) and (3.1,1.1) .. (3.2,1.2);
    
    \node[red] at (3.1, 0.6) {$\tilde\sigma$};
    \node[red] at (3.1, 3.4) {$\tilde\tau$};

    \draw[red, ultra thick] (0, 1) -- (2.5, 1);
    \draw[red, ultra thick] (0, 3) -- (2.5, 3);

    \fill[color=red] (1.75,1) circle (1.5mm);
    \fill[color=red] (1.75,3) circle (1.5mm);

    \node[red] at (1.75, 0.4) {$\theta_{\sigma}$};
    \node[red] at (1.75, 3.6) {$\theta_{\tau}$};

    \node[red] at (0.6, 0.6) {$\sigma$};
    \node[red] at (0.6, 3.4) {$\tau$};

    \draw[blue, ultra thick] (4,2) -- (6,2);

    \node[blue] at (5, 2.5) {$\tilde\phi$};

    \node at (-0.4, 0.25) {$A$};
    \node at (-0.4, 2) {$B$};
    \node at (-0.4, 3.75) {$\Gamma$};

    \node at (6.4, 0.25) {$A$};
    \node at (6.4, 3.75) {$\Gamma$};

    \end{scope}

    \node at (7.4,-8) {$=$};

    \begin{scope}[xshift=8.8cm, yshift = -10cm]
    \draw[ultra thick] (0,0) -- (0,4);
    \draw[ultra thick] (6,0) -- (6,4);
    
     \draw[red, ultra thick] (1.5,2) parabola (0.7,1.2);
    \draw[red, ultra thick] (1.5,2) parabola (0.7,2.8);
    \draw[red, ultra thick] (0,3) .. controls (0.5,3) and (0.6,2.9) .. (0.7,2.8);
    \draw[red, ultra thick] (0,1) .. controls (0.5,1) and (0.6,1.1) .. (0.7,1.2);
    \node[red] at (0.6, 0.6) {$\sigma$};
    \node[red] at (0.6, 3.4) {$\tau$};

    \node at (-0.4, 0.25) {$A$};
    \node at (-0.4, 2) {$B$};
    \node at (-0.4, 3.75) {$\Gamma$};

    \node at (6.4, 0.25) {$A$};
    \node at (6.4, 3.75) {$\Gamma$};

    \draw[blue, ultra thick] (1.5, 2) -- (6, 2);
    \fill[color=blue] (3.75,2) circle (1.5mm);

    \node[blue] at (3.75, 2.5) {$\theta$};
    \node[blue] at (5, 2.5) {$\tilde\phi$};
    \node[blue] at (2.5, 2.5) {$\phi$};

    \end{scope}

    \end{tikzpicture}\]

    Here, $\theta_\sigma$ and $\theta_\tau$ are isotopies of surgery triples with 
    $\theta_\sigma(0) = \sigma$, $\theta_\sigma(1) = \sigma'$, $\theta_\tau(0) = \tau$, 
    and $\theta_\tau(1) = \tau'$, such that $\theta_\sigma(t)$ and $\theta_\tau(t)$ 
    remain in cancelling position for all $t$. We denote their cancelling diffeomorphism 
    at time $t$ by $\theta(t)$, so that in particular $\theta(0) = \phi$ and 
    $\theta(1) = \phi'$.

    Let $W := W(A, B, \sigma) \cup W(B, \Gamma, \tau)$ and 
    $\tilde{W} := W(A, B, \tilde{\sigma}) \cup W(B, \Gamma, \tilde{\tau})$, with 
    saturated double neighborhoods $U \subset W$ and $\tilde{U} \subset \tilde{W}$, 
    respectively. The isotopies $\theta_\sigma$ and $\theta_\tau$ induce a 
    foliation-preserving diffeomorphism $D\colon W \xrightarrow{\sim} \tilde{W}$ 
    sending $U$ to $\tilde{U}$. Denote by $\Phi^{-1}$ and $\tilde{\Phi}^{-1}$ the 
    diffeomorphisms assigned by $\BRF$ to $\Cd(\sigma, \tau)$ and 
    $\Cd(\tilde{\sigma}, \tilde{\tau})$, respectively, and by 
    $\tilde{D}\colon \BRF(\phi) \xrightarrow{\sim} \BRF(\tilde{\phi})$ the 
    diffeomorphism induced by the isotopy $\theta$. The left-hand side of the 
    relation then corresponds to $\tilde{\Phi}^{-1} \circ D$, and the right-hand 
    side to $\tilde{D} \circ \Phi^{-1}$. 
    
    We denote the maps on string nets associated to the left and right hand sides by
    $$F = \orSN^\alpha_\II(\Cd(\tilde\sigma,\tilde\tau)) \circ \orSN^\alpha_\II(D) \quad  \text{ and }\quad G = \orSN^\alpha_\II(\tilde{D}) \circ \orSN^\alpha_\II(\Cd(\sigma,\tau))\ .$$
    To verify the relation, we show that both maps transport any string 
    net $T$ on $W$ in canceling position with respect to $U$ to the same string 
    net on $\tilde{W}$. Since both $\tilde{\Phi}^{-1} \circ D$ and 
    $\tilde{D} \circ \Phi^{-1}$ send $U$ to $\tilde{U}$, it suffices to check this 
    separately on $W \setminus U$ and on $U$.

    On $W \setminus U$, both maps $F$ and $G$ are induced by applying the diffeomorphisms $\tilde{\Phi}^{-1} \circ D$ and $\tilde{D} \circ \Phi^{-1}$ which coincide.

    On $U$, since $D$ and $\tilde{D}$ are foliation-preserving and by the 
    definitions of how the cusp-death acts on string nets, the image of $T$ inside 
    $\tilde{U}$ in both cases consists of a single progressive strand from the lower 
    to the upper boundary of $\tilde{U}$, labeled by $\alpha$ or $\alpha^{-1}$ 
    depending on the type of cancellation. Since $\tilde{U}$ is a disk, this 
    uniquely determines the image of $T$ inside $\tilde{U}$ and completes the 
    verification that $\orSN_{\II}^{\alpha}$ satisfies the cusp-isotopy relations.

    \item \textbf{The remaining $\mathcal{F}^1$-relations}

    The remaining $\mathcal{F}^1$-relations --- namely the crossing-isotopy, 
    naturality-beak, naturality-Reidemeister~III, frame change-cusp, and frame 
    change-crossing relations --- are verified by an argument analogous to that of 
    the cusp-isotopy relation. In each case, the relation is obtained by pre- or 
    post-composing a cusp or crossing generating $2$-morphism with a 
    foliation-preserving diffeomorphism that sends double neighborhoods to double 
    neighborhoods. In the cusp-isotopy and crossing-isotopy cases, this diffeomorphism 
    is induced by diffeo isotopies and surgery isotopies; in the remaining 
    cases, it is induced by diffeo-surgery $2$-morphisms and frame change data. In all cases, the argument verifying that $\orSN_{\II}^{\alpha}$ satisfies the relation is essentially the same.
\end{itemize}

\subsubsection{$\mathcal{F}^2$-relations}
The $\mathcal{F}^2$-relations are those appearing in Section~3.3.3 of~\cite{Filippos}, 
namely the swallowtail, beak, and Reidemeister~III relations. We recall each relation 
by showing the diagrams it involves, omitting labels for readability. In each case, 
the verification that $\orSN^{\alpha}_{\II}$ satisfies the relation follows immediately 
from its construction; we carry out the verification in detail for the swallowtail 
relation and sketch it for the remaining two.

\begin{itemize}
    \item \textbf{The swallowtail relations}

    \[\begin{tikzpicture}[scale=0.5]

    \begin{scope}[yshift = -10cm]
    \draw[ultra thick] (0,0) -- (0,3);
    \draw[ultra thick] (6,0) -- (6,3);

    \draw[red, ultra thick] (4.5,2) parabola (3.7,1.7);
    \draw[red, ultra thick] (4.5,2) parabola (3.7,2.3);
    
    \draw[red, ultra thick] (4.5,2) parabola (3.7,2.3);
    
    \draw[red, ultra thick] (3,2.5).. controls (3.5,2.5) and (3.6,2.4)..(3.7,2.3);
    
    \draw[red, ultra thick] (3,2.5).. controls (2.5, 2.5) and (2.4, 2.4).. (2.3, 2.3);
    
    \draw[red, ultra thick] (1.5, 2) parabola (2.3, 2.3);
    \draw[red, ultra thick] (1.5, 2) parabola (2.3, 1.7);

    \draw[red, ultra thick] (2.3, 1.7).. controls (3.25, 1) and (3.5, 1).. (6,1);
    
    \draw[red, ultra thick] (3.7, 1.7).. controls (2.75, 1) and (2.5,1).. (0,1);


    \draw[blue, ultra thick] (1.5, 2) -- (0, 2);
    \draw[blue, ultra thick] (4.5, 2) -- (6, 2);

    \end{scope}

    \node at (7,-8) {$=$};

    \begin{scope}[xshift=8cm, yshift = -10cm]
    \draw[ultra thick] (0,0) -- (0,3);
    \draw[ultra thick] (6,0) -- (6,3);
    \draw[red, ultra thick] (0,1) -- (6,1);

    \draw[blue, ultra thick] (0, 2) -- (1.5, 1);
    \draw[blue, ultra thick] (6, 2) -- (4.5, 1);
    \end{scope}

    \end{tikzpicture}\]

    Each diagram above represents a family of relations, one for each valid choice 
    of labels. There are also analogous families of relations given by the vertical 
    reflections of the above diagrams.

    Consider first the swallowtail beginning with a zero-handle, introducing a one-handle and zero-handle via a cusp-birth, and then canceling the original zero-handle with the one-handle using a cusp-death. This swallowtail is supported in a region where it is given by the following composition:
    
    \begin{equation*}
        \begin{tikzpicture}[xscale = 2,yscale = 1, baseline = 1cm]
        
        \begin{scope}[yshift=0.7cm]
        \foreach \y in {0,...,10}{
        \draw[gray!70] (-0.5,1.8*\y/10) --++ (1,0);
        }
        \end{scope}
    
         \begin{scope}
         \draw[gray!70] (-0.5,-0.7) -- (0.5,-0.7);
        \clip (-0.5,-0.7) rectangle (0.5,1);
        \foreach \y in {0,...,5}{
        \draw[gray!70] (0,0) circle(\y/8);
        }
        \node[star,  star point ratio=2.25, fill=black, inner sep = 0.5pt] at (0,0) {};
        \draw (0,0) -- (0,-0.7) node[pos = 0.3, sloped]{\footnotesize $>$} node[pos = 0.3, left]{\footnotesize $\alpha$};
        \end{scope}
        \end{tikzpicture}
\hspace{20pt}\mapsto \hspace{20pt}
        \begin{tikzpicture}[xscale = 2,yscale = 1, baseline = 1cm]
    
        \begin{scope}[yshift=1.8cm]
         \draw[gray!70] (-0.5,0.7) -- (0.5,0.7);
            \clip (-0.5,-1) rectangle (0.5,1);
        \foreach \y in {0,...,5}{
        \draw[gray!70] (0,0) circle(\y/8);
        }
        \node[circle, fill=black, inner sep = 1pt] at (0,2) {};
        \node[star,  star point ratio=2.25, fill=black, inner sep = 0.5pt] at (0,0) {};
        \end{scope}

        \draw (0,1.8) -- (0,0.9) node[pos = 0.5, sloped]{\footnotesize $>$} node[pos = 0.5, left]{\footnotesize $\alpha$};

        \begin{scope}[yshift=0.5cm]
        \draw[gray!70] (-0.5,0.8) -- (0.5,0);
        \draw[gray!70] (-0.5,0) -- (0.5,0.8);
        \clip (-0.5,0) rectangle (0.5,0.8);
        \foreach \y in {2,...,3}{
        \draw[gray!70] (-0.7,0.4) circle(\y/7 and \y/10);
        \draw[gray!70] (0.7,0.4) circle(\y/7 and \y/10);
        }   
        \node[star,  star point ratio=2.25, fill=black, inner sep = 0.5pt] at (0,0.4) {};
        \end{scope}
    
         \begin{scope}
         \draw[gray!70] (-0.5,-0.7) -- (0.5,-0.7);
        \clip (-0.5,-0.7) rectangle (0.5,1);
        \foreach \y in {0,...,5}{
        \draw[gray!70] (0,0) circle(\y/8);
        }
        \node[star,  star point ratio=2.25, fill=black, inner sep = 0.5pt] at (0,0) {};
        \draw (0,0) -- (0,-0.7) node[pos = 0.3, sloped]{\footnotesize $>$} node[pos = 0.3, left]{\footnotesize $\alpha$};
        \end{scope}
        \end{tikzpicture}
\hspace{20pt}\mapsto \hspace{20pt}
        \begin{tikzpicture}[xscale = 2,yscale = 1, baseline = 1cm]
        
        \begin{scope}[yshift=-0.7cm]
        \foreach \y in {0,...,10}{
        \draw[gray!70] (-0.5,1.8*\y/10) --++ (1,0);
        }
        \end{scope}
    
         \begin{scope}[yshift=1.8cm]
         \draw[gray!70] (-0.5,0.7) -- (0.5,0.7);
        \clip (-0.5,-0.7) rectangle (0.5,1);
        \foreach \y in {0,...,5}{
        \draw[gray!70] (0,0) circle(\y/8);
        }
        \node[star,  star point ratio=2.25, fill=black, inner sep = 0.5pt] at (0,0) {};
        \end{scope}
        \draw (0,1.8) -- (0,-0.7) node[pos = 0.3, sloped]{\footnotesize $>$} node[pos = 0.3, left]{\footnotesize $\alpha$};
        \end{tikzpicture}
    \end{equation*}
    
    The skeins before and after the swallowtail isomorphism are clearly equivalent upon identifying the two foliations. Consider next the swallowtail beginning with a one-handle, introducing a zero-handle and one-handle via a cusp-birth, and then canceling the original one-handle with the zero-handle using a cusp-death. It is supported in a region where it is given by the following composition: 
    \begin{equation*}
        \begin{tikzpicture}[xscale = 2,yscale = 1, baseline = 0cm]
        \begin{scope}[yshift=0cm]
        \foreach \y in {0,...,9}{
        \begin{scope}
            \clip (-0.5,-0.7 + 1.8*\y/10) rectangle (0.5, 1.8*\y/10);
             \draw[gray!70] (0,1.8*\y/10) circle(5/8);
        \end{scope}
        }
        \clip (-0.5,-0.7 + 1.3) rectangle (0.5, 1.25);
        \draw[gray!70] (0,1.8) circle(5/8);
        \end{scope}
        \begin{scope}[yshift=-1.3cm]
        \draw[gray!70] (-0.5,0.8) -- (0.5,0);
        \draw[gray!70] (-0.5,0) -- (0.5,0.8);
        \clip (-0.5,0) rectangle (0.5,0.8);
        \foreach \y in {2,...,3}{
        \draw[gray!70] (-0.7,0.4) circle(\y/7 and \y/10);
        \draw[gray!70] (0.7,0.4) circle(\y/7 and \y/10);
        }   
        \node[star,  star point ratio=2.25, fill=black, inner sep = 0.5pt] at (0,0.4) {};
        \draw[gray!70] (0,-0.4) circle(0.55);
        \draw[gray!70] (0,-0.4) circle(0.45);
        \end{scope}
        \draw (0,-1.3) -- (0,-0.9) node[pos = 0.3, sloped]{\footnotesize $>$} node[pos = 0.5, left]{\footnotesize $\alpha$};
        \end{tikzpicture}
\hspace{15pt}\mapsto \hspace{15pt}
        \begin{tikzpicture}[xscale = 2,yscale = 1, baseline = 0cm]
        \begin{scope}[yshift=0.5cm]
        \clip (-0.5,-1.2) rectangle (0.5,0.8);
        \draw[gray!70] (-0.5,0.8) -- (0,0.4) ..controls (1,-0.4) and (0.22,-1.15).. (0,-1.15);
        \draw[gray!70] (0.5,0.8) -- (0,0.4) ..controls (-1,-0.4) and (-0.22,-1.15).. (0,-1.15);
        \foreach \y in {2,...,3}{
        \draw[gray!70] (-0.7,0.35) circle(\y/8);
        \draw[gray!70] (0.7,0.35) circle(\y/8);
        }   
        \node[star,  star point ratio=2.25, fill=black, inner sep = 0.5pt] at (0,0.4) {};
        \draw[gray!70] (0,1.2) circle(0.55);
        \draw[gray!70] (0,1.2) circle(0.45);
        \end{scope}
        \draw (0,0.9) -- (0,0) node[pos = 0.3, sloped]{\footnotesize $<$} node[pos = 0.3, left]{\footnotesize $\alpha$};
         \begin{scope}
        \clip (-0.5,-0.7) rectangle (0.5,1);
        \foreach \y in {0,...,4}{
        \draw[gray!70] (0,\y/60) circle(\y/10 and \y/7);
        }
        \node[star,  star point ratio=2.25, fill=black, inner sep = 0.5pt] at (0,0) {};
        \end{scope}
        \begin{scope}[yshift=-1.3cm]
        \draw[gray!70] (-0.5,0.8) -- (0.5,0);
        \draw[gray!70] (-0.5,0) -- (0.5,0.8);
        \clip (-0.5,0) rectangle (0.5,0.8);
        \foreach \y in {2,...,3}{
        \draw[gray!70] (-0.7,0.4) circle(\y/7 and \y/10);
        \draw[gray!70] (0.7,0.4) circle(\y/7 and \y/10);
        }   
        \node[star,  star point ratio=2.25, fill=black, inner sep = 0.5pt] at (0,0.4) {};
        \draw[gray!70] (0,-0.4) circle(0.55);
        \draw[gray!70] (0,-0.4) circle(0.45);
        \end{scope}
        \draw (0,-1.3) -- (0,-0.9) node[pos = 0.3, sloped]{\footnotesize $>$} node[pos = 0.5, left]{\footnotesize $\alpha$};
        \end{tikzpicture}
\hspace{15pt}\mapsto \hspace{15pt}
\begin{tikzpicture}[xscale = 2,yscale = 1, baseline = 0cm]
        \begin{scope}[yshift=-0.5cm, yscale = -1]
        \clip (-0.5,-1.2) rectangle (0.5,0.8);
        \draw[gray!70] (-0.5,0.8) -- (0,0.4) ..controls (1,-0.4) and (0.22,-1.15).. (0,-1.15);
        \draw[gray!70] (0.5,0.8) -- (0,0.4) ..controls (-1,-0.4) and (-0.22,-1.15).. (0,-1.15);
        \foreach \y in {2,...,3}{
        \draw[gray!70] (-0.7,0.35) circle(\y/8);
        \draw[gray!70] (0.7,0.35) circle(\y/8);
        }   
        \node[star,  star point ratio=2.25, fill=black, inner sep = 0.5pt] at (0,0.4) {};
        \draw[gray!70] (0,1.2) circle(0.55);
        \draw[gray!70] (0,1.2) circle(0.45);
        \end{scope}
        \draw (0,0.9) -- (0,0) node[pos = 0.3, sloped]{\footnotesize $<$} node[pos = 0.3, left]{\footnotesize $\alpha$};
         \begin{scope}[yscale = -1]
        \clip (-0.5,-0.7) rectangle (0.5,1);
        \foreach \y in {0,...,4}{
        \draw[gray!70] (0,\y/60) circle(\y/10 and \y/7);
        }
        \node[star,  star point ratio=2.25, fill=black, inner sep = 0.5pt] at (0,0) {};
        \end{scope}
        \begin{scope}[yshift=1.3cm, yscale = -1]
        \draw[gray!70] (-0.5,0.8) -- (0.5,0);
        \draw[gray!70] (-0.5,0) -- (0.5,0.8);
        \clip (-0.5,0) rectangle (0.5,0.8);
        \foreach \y in {2,...,3}{
        \draw[gray!70] (-0.7,0.4) circle(\y/7 and \y/10);
        \draw[gray!70] (0.7,0.4) circle(\y/7 and \y/10);
        }   
        \node[star,  star point ratio=2.25, fill=black, inner sep = 0.5pt] at (0,0.4) {};
        \draw[gray!70] (0,-0.4) circle(0.55);
        \draw[gray!70] (0,-0.4) circle(0.45);
        \end{scope}
        \draw (0,-1.3) -- (0,-0.9) node[pos = 0.3, sloped]{\footnotesize $>$} node[pos = 0.5, left]{\footnotesize $\alpha$};
\end{tikzpicture}
\hspace{15pt}\mapsto \hspace{15pt}
        \begin{tikzpicture}[xscale = 2,yscale = 1, baseline = 0cm]
        \begin{scope}[yshift=-1.8cm]
        \foreach \y in {1,...,10}{
        \begin{scope}
            \clip (-0.5, 1.8*\y/10) rectangle (0.5, 0.7 + 1.8*\y/10);
             \draw[gray!70] (0,1.8*\y/10) circle(5/8);
        \end{scope}
        }
        \clip (-0.5,0.55) rectangle (0.5, 1);
        \draw[gray!70] (0,0) circle(5/8);
        \end{scope}
    
        \begin{scope}[yshift=0.5cm]
        \draw[gray!70] (-0.5,0.8) -- (0.5,0);
        \draw[gray!70] (-0.5,0) -- (0.5,0.8);
        \clip (-0.5,0) rectangle (0.5,0.8);
        \foreach \y in {2,...,3}{
        \draw[gray!70] (-0.7,0.4) circle(\y/7 and \y/10);
        \draw[gray!70] (0.7,0.4) circle(\y/7 and \y/10);
        }    
        \node[star,  star point ratio=2.25, fill=black, inner sep = 0.5pt] at (0,0.4) {};
        \draw[gray!70] (0,1.2) circle(0.55);
        \draw[gray!70] (0,1.2) circle(0.45);
        \end{scope}
    
        \draw (0,-1.25) -- (0,0.9) node[pos = 0.3, sloped]{\footnotesize $>$} node[pos = 0.5, left]{\footnotesize $\alpha$};
        \end{tikzpicture}
    \end{equation*}
    
    The swallowtails involving the two-handle produce mirrored versions of the diagrams for the swallowtails involving zero-handles, with $\alpha$-strands replaced by $\alpha^{-1}$-strands.

    \item \textbf{The beak relations}

    \[\begin{tikzpicture}[xscale=0.5, yscale = 0.4]

    \begin{scope}[yshift = -10cm]    
    \draw[ultra thick] (0,0) -- (0,5);
    \draw[ultra thick] (5,0) -- (5,5);

    \draw[red, ultra thick] (3,2) parabola (2.2,1.2);
    \draw[red, ultra thick] (3,2) parabola (2.2,2.8);
    \draw[red, ultra thick] (0,3)-- (1.5,3);
    \draw[red, ultra thick] (1.5,3).. controls (2,3) and (2.1,2.9)..(2.2,2.8);
    \draw[red, ultra thick] (0,1)--(1.5,1);
    \draw[red, ultra thick] (1.5,1).. controls (2,1) and (2.1, 1.1)..(2.2,1.2);
    
    \draw[red, ultra thick] (0,4)to[out = 0, in = 180, out distance = 2cm, in distance = 5cm] (5,1);

    \draw[blue, ultra thick] (3,2) to[out = 0, in = 180](5, 2);

    \end{scope}

    \node at (6,-8) {$=$};

    \begin{scope}[xshift=8cm, yshift = -10cm]
    \draw[ultra thick] (0,0) -- (0,5);
    \draw[ultra thick] (5,0) -- (5,5);

    \draw[red, ultra thick] (2,2) parabola (1.2,1.2);
    \draw[red, ultra thick] (2,2) parabola (1.2,2.8);
    \draw[red, ultra thick] (0,3)-- (0.5,3);
    \draw[red, ultra thick] (0.5,3).. controls (1,3) and (1.1,2.9)..(1.2,2.8);
    \draw[red, ultra thick] (0,1)--(0.5,1);
    \draw[red, ultra thick] (0.5,1).. controls (1,1) and (1.1, 1.1)..(1.2,1.2);
    
    \draw[red, ultra thick] (0,4)to[out = 0, in = 180, out distance = 5cm] (5,1);

    \draw[blue, ultra thick] (2, 2)--(3.25, 2);
    \draw[blue, ultra thick] (3.25,2)to[out = 0, in = 180](5, 2);

    \end{scope}

    \end{tikzpicture}\]

    Each diagram above represents a family of relations, one for each valid choice of 
    labels, with analogous families given by the horizontal and vertical reflections 
    of the above diagrams. These relations arise from a path of functions starting with 
    three critical points, two of which cancel via a cusp while the third moves 
    independently, its critical value dropping below the critical level of the cusp 
    either before or after the cusp occurs --- this is precisely what distinguishes the 
    left- and right-hand sides. Starting with a string net in both cancelling and 
    crossing position and tracing through the definitions, one verifies that both sides 
    induce the same final string net.

    \item \textbf{The Reidemeister~III relations}

    \[\begin{tikzpicture}[scale=0.5]

    \begin{scope}[yshift = -10cm]
    \draw[ultra thick] (0,0) -- (0,4);
    \draw[ultra thick] (6,0) -- (6,4);

    \draw[red, ultra thick] (0,1) -- (6,3);
    \draw[red, ultra thick] (0,3) -- (6,1);

    \draw[red, ultra thick] (0,2).. controls (1.5,3) and (2.25,3).. (3,3);
    
    \draw[red, ultra thick] (3,3).. controls (3.75,3) and (4.5,3).. (6,2);
    \end{scope}

    \node at (7,-8) {$=$};

    \begin{scope}[xshift=8cm, yshift = -10cm]
    \draw[ultra thick] (0,0) -- (0,4);
    \draw[ultra thick] (6,0) -- (6,4);
    \draw[red, ultra thick] (0,1) -- (6,3);
    \draw[red, ultra thick] (0,3) -- (6,1);

    \draw[red, ultra thick] (0,2).. controls (1.5,1) and (2.25,1).. (3,1);
    
    \draw[red, ultra thick] (3,1).. controls (3.75,1) and (4.5,1).. (6,2);
    \end{scope}

    \end{tikzpicture}\]

    Each diagram above represents a family of relations, one for each valid choice of 
    labels. The verification that $\orSN^{\alpha}_{\II}$ satisfies 
    these relations is entirely analogous to that of the beak relations, with the 
    word ``cusp'' replaced by ``crossing'' throughout.

\end{itemize}


\section{Some properties of the twisted string net TQFT}\label{sec:Properties_S2}

\subsection{The 2-sphere and the distinguished invertible object}
We compute the value of our TQFT on the simplest closed surface, the 2-sphere. As we have shown, the vector space assigned to $S^2$ only depends on the orientation of $S^2$. However, in order to compute it, we must first find a Morse foliation on $S^2$.
\begin{definition}
    Let $S^2 \subseteq \mathbb R^3$ be the standard 2-sphere. The projection on the first coordinate is a Morse function with two critical points of index 0 and 2. We will still denote $S^2$ the foliated surface $S^2$ with foliation induced by this Morse function.

    Let $S^2\smallsetminus\ast$ be the foliated surface obtained from $S^2$ by removing (a small neighborhood of) a non-critical point. It is homeomorphic to a disk but its foliation has two critical points. It can be obtained as the gluing $S^2\smallsetminus\ast \simeq \D_0\underset{I}{\cup}\D_2$. 
\end{definition}
\begin{proposition}\label{prop:coendFormulaForS2}
Let $\II\subseteq \BB$ be a tensor ideal in a twisted-pivotal category. 
    The twisted string net module of the punctured 2-sphere (without boundary labels) is obtained as 
    $$\SN_\II^\alpha(S^2\smallsetminus \ast) \simeq \int^{x\in\II} \Hom_\BB(\alpha, x) \otimes \Hom_\BB(x,\alpha^{-1})$$
with isomorphism given by 
\begin{equation*}
    \begin{tikzpicture}[scale = 0.9, baseline = 0pt]
    \clip (0,0) circle(2);
    \def\lines{8}
    \foreach \y in {1,...,\lines}{
    \draw[gray!70] (0,\y/\lines) arc(-90:270:{20/(\y*\y)});
    }
    \draw[gray!70] (-2,0) --(2,0);
    \foreach \y in {-1,...,-\lines}{
    \draw[gray!70] (0,\y/\lines) arc(90:540:{20/(\y*\y)});
    }
    \node[star, star point ratio=2.25, fill=black, inner sep = 1pt] at (0,1.3) {};
    \node[star, star point ratio=2.25, fill=black, inner sep = 1pt] at (0,-1.3) {};
    \draw[very thick] (0,-1.3) -- (0,1.3) node[pos = 0.1, right]{$\alpha$} node[pos = 0.1, sloped]{$>$} node[pos = 0.3, rectangle, draw, fill=white]{$f$}node[pos = 0.5, right]{$x$} node[pos = 0.5, sloped]{$>$} node[pos = 0.7, rectangle, draw, fill=white]{$g$}node[pos = 0.9, right]{$\alpha^{-1}$} node[pos = 0.9, sloped]{$>$};
    \end{tikzpicture}
    \quad\mapsfrom \quad
    f \otimes g
\end{equation*}
    The twisted string net module of the 2-sphere is obtained as
    $$\SN_\II^\alpha(S^2) \simeq \int^{x\in\II} \Hom_\BB(\alpha, x) \otimes \Hom_\BB(x,\alpha^{-1}) \bigg/ \sim$$
    where $\sim$ is generated by
    \begin{equation}\label{eq:extraRelationUnpuncturedS2}
\begin{tikzpicture}[xscale = 1.5,yscale = 2, baseline = 0.9cm]
    \draw[very thick] (0,0) -- ++(0,0.5) node[pos = 0.3, sloped]{\small $>$}node[pos = 0.3, left]{\small $\alpha$} ;
    \draw[very thick] (-0.2, 1) -- ++(0,-0.5) node[pos = 0.3, sloped]{\small $<$}node[pos = 0.3, left]{\small $x$} ;
    \draw[very thick] (0.2, 1) -- ++(0,-0.5) node[pos = 0.3, sloped]{\small $<$}node[pos = 0.3, left]{\small $y$} ;
    \node[rectangle, draw, very thick, fill=white, minimum width = 0.9cm, inner sep = 2pt] at (0,0.5) {\small $f$};
\end{tikzpicture}
\ \otimes \ 
\begin{tikzpicture}[xscale = 1.5,yscale = 2, baseline = 0.9cm]
    \draw[very thick] (-0.2, 0) to[out=90, in=-90] node[pos = 0.3, sloped]{\small $>$}node[pos = 0.3, left]{\small $x$} (-0.2,0.5) ;
    \draw[very thick] (0.2, 0) to[out=90, in=-90] node[pos = 0.3, sloped]{\small $>$}node[pos = 0.3, left]{\small $y$} (0.2,0.5) ;
    \draw[very thick] (0,0.5) to[out = 90, in = -90] node[pos = 0.7, sloped]{\small $>$} node[pos = 0.7, right]{\small $\alpha^{-1}$}  (0,1);
    \node[rectangle, draw, very thick, fill=white, minimum width = 0.9cm, inner sep = 2pt] at (0,0.5) {\small $g$};
\end{tikzpicture}
\quad \sim \quad
\begin{tikzpicture}[xscale = 1.5,yscale = 2, baseline = 0.9cm]
    \draw[very thick] (0.2,0.5) to[out = 90, in = -135] node[pos = 0.6, sloped]{\small $>$}node[pos = 0.6, left]{\small $y$} (0.3, 0.8) node{\small $\bullet$} to[out = -45, in = 13, in distance  = 25pt] (-0.2, 0.18) node{\small $\bullet$} to[out = 155, in=-100] node[pos = 0.7, sloped]{\small $>$}node[pos = 0.7, left]{\small $y$} (-0.3, 1);
    \draw[very thick] (0.1, 1) to[out = -90, in=90] node[pos = 0.3, sloped]{\small $>$}node[pos = 0.3, left]{\small $x$} (-0.2,0.5);
    \draw[very thick] (0,0) -- ++(0,0.5) node[pos = 0.1, sloped]{\small $>$}node[pos = 0.1, right]{\small $\alpha$} node[pos = 0.45, rectangle, draw, fill=white, rotate = 15]{};
    \node[rectangle, draw, very thick, fill=white, minimum width = 0.9cm, inner sep = 2pt] at (0,0.5) {\small $f$};
\end{tikzpicture}
\ \otimes \ 
\begin{tikzpicture}[xscale = 1.5,yscale = -2, baseline = -1.1cm]
    \draw[very thick] (0.2,0.5) to[out = 90, in = -135] node[pos = 0.6, rotate = 100]{\small $>$}node[pos = 0.6, left]{\small $y$} (0.3, 0.8) node{\small $\bullet$} to[out = -45, in = 13, in distance  = 25pt] (-0.2, 0.18) node{\small $\bullet$} to[out = 155, in=-100] node[pos = 0.7, rotate = 100]{\small $>$}node[pos = 0.7, left]{\small $y$} (-0.3, 1);
    \draw[very thick] (0.1, 1) to[out = -90, in=90] node[pos = 0.3, rotate = 125]{\small $>$}node[pos = 0.3, left]{\small $x$} (-0.2,0.5);
    \draw[very thick] (0,0) -- ++(0,0.5) node[pos = 0.1, rotate = 90]{\small $>$}node[pos = 0.1, right]{\small $\alpha^{-1}$} node[pos = 0.45, rectangle, draw, fill=white, rotate = -15]{};
    \node[rectangle, draw, very thick, fill=white, minimum width = 0.9cm, inner sep = 2pt] at (0,0.5) {\small $g$};
\end{tikzpicture}
    \end{equation}
    for $x,y\in\II$, $f\in\Hom_\BB(\alpha, x\otimes y)$ and $g\in\Hom_\BB(x\otimes y, \alpha^{-1})$
\end{proposition}
\begin{proof}
For the punctured sphere, we apply Proposition \ref{prop:excision} on the decomposition $S^2\smallsetminus\ast = \D_0\underset{I}{\cup}\D_2$ to express its string net module as a coend
$$\SN_\II^\alpha(S^2\smallsetminus\ast) \simeq \int^{X \in \SN_\II^\alpha(I)} \SN_\II^\alpha(\D_0;X) \otimes\SN_\II^\alpha(\D_2;X)\ .$$
Now we identify $\SN_\II^\alpha(I) \simeq \II$ where $x\in\II$ can be seen as a single point in $I$ colored by $x$, and by Propositions \ref{prop:SNonD0} and \ref{prop:SNonD2} we have $\SN_\II^\alpha(\D_0,x)\simeq \Hom_\BB(\alpha, x)$ and $ \SN_\II^\alpha(\D_2,x)  \simeq \Hom_\BB(x,\alpha^{-1})$. It is easy to see in that these isomorphisms preserve the action of $\II$, which gives the desired formula. 

For the 2-sphere, we use excision on $S^2 = \D_0\underset{S^1}{\cup}\D_2$ to express $\SN_\II^\alpha(S^2)$ as the coend
$$\SN_\II^\alpha(S^2) \simeq \int^{X \in \SN_\II^\alpha(S^1)} \SN_\II^\alpha(\D_0;X) \otimes\SN_\II^\alpha(\D_2;X)\ .$$
As discussed in Section \ref{sec:functoriality}, the category $\SN_\II^\alpha(S^1)$ is spanned by objects of $\II$, and morphisms are spanned by morphisms of $\II$ together with the invertible morphisms $T_{x,y}$ of \eqref{eq:genMorphSNS1}. Therefore, $\SN_\II^\alpha(S^2)$ has the same generators as $\SN_\II^\alpha(S^2\smallsetminus \ast)$ but there are more relations, coming from these additional morphisms. Since they are invertible, the coend relations can be expressed as
$$(f,g) \sim (T_{x,y}\cdot f, T_{x,y}^{-1}\cdot g)\ .$$
It only remains to identify the action of $T_{x,y}$ on $\Hom_\BB(\alpha, x)$ and $\Hom_\BB(x,\alpha^{-1})$ by tracking down the isomorphism described in Propositions \ref{prop:SNonD0} and \ref{prop:SNonD2}, which gives the claimed depictions. 
\end{proof}
\begin{remark}\label{rmk:S2formulabothalphasincoming}
Using duality isomorphisms, we can also obtain 
$$\SN_\II^\alpha(S^2\smallsetminus \ast) \simeq \int^{x\in\II} \Hom_\BB(\alpha^{\otimes 2}, x) \otimes \Hom_\BB(x,\unit)$$
and $\SN_\II^\alpha(S^2)$ as a suitable quotient.
\end{remark}

Let us now suppose that $\BB$ is a tensor category with enough projectives in the sense of \cite{EGNO} and $\II$ is the tensor ideal $\Proj$ of projective objects. Every simple object $S$ has a unique projective cover $P_S\to S$ and injective hull $S\to I_S$. Injective and projective objects coincide using rigidity, and a projective cover $P_S$ is the injective hull of a unique simple $S'$ called the socle of $P_S$. 

We denote $D$ the socle of the projective cover of the unit, the unique simple that fits in 
$$D \hookrightarrow P_\unit \twoheadrightarrow \unit$$
If $\BB$ is moreover finite, then $D$ is the \textit{distinguished invertible object} of $\BB$ defined in \cite{ENOdistinguishedinvertible}. It comes equipped with the Radford trivialization of the quadruple dual $x^{****}\otimes D\overset{\Radford}{\simeq} D\otimes x$.

\begin{corollary}\label{cor:dimSNS2}
    Let $\BB$ be an $\alpha$-twisted-pivotal tensor category and $\II=\Proj$. Then:
    \begin{enumerate}
        \item The twisted string net module of the punctured two sphere is non-trivial if and only if $\alpha$ squares to $D$, more precisely:
    $$\operatorname{dim}\,\SN_\II^\alpha( S^2\smallsetminus\ast)=\left\{ \begin{array}{ll}
        1 & \text{if } \alpha^{\otimes 2} \simeq D \\
        0 & \text{if } \alpha^{\otimes 2} \not\simeq D
    \end{array}\right. \ .$$
    \item If $\BB$ is moreover finite, the twisted string net module of the two sphere is non-trivial if and only if $\alpha$ squares to the distinguished invertible and the trivializations of the quadruple duals agree, more precisely:
    $$\operatorname{dim}\,\SN_\II^\alpha( S^2)=\left\{ \begin{array}{ll}
        1 & \text{if } \alpha^{\otimes 2} \simeq D \text{ and } \apiv^2 = \Radford \\
        0 & \text{else}
    \end{array}\right.$$
    where the square of $\alpha$ is equipped with a trivialization of the quadruple dual 
    $$\apiv^2_x: x^{****}\otimes \alpha^{\otimes 2} \overset{\apiv_{x^{**}}}\simeq \alpha\otimes x^{**}\otimes \alpha  \overset{\apiv_{x}}\simeq \alpha^{\otimes 2} \otimes x\ .$$
    \item If $\alpha$ does not square to $D$ with its canonical trivialization of the quadruple dual, then the oriented categorified 2-TQFT $\orSN_\II^\alpha$ is 0 on $S^2$ and cannot be extended to a non-compact 3-TQFT. 
\end{enumerate}
\end{corollary}
\begin{proof} We begin with point 1. 
    The coend 
    $$\SN_\II^\alpha(S^2\smallsetminus \ast) \simeq \int^{x\in\II} \Hom_\BB(\alpha^{\otimes 2}, x) \otimes \Hom_\BB(x,\unit)$$
reduces to a coend where $x$ runs over isomorphism classes of indecomposable projectives.
The vector space $\Hom_\BB(x,\unit)$ is one-dimensional if $x = P_\unit$ is the projective cover of the unit, generated by $\varepsilon:P_\unit \to \unit$, and zero otherwise. 
The vector space $\Hom_\BB(\alpha^{\otimes 2},x)$ is one-dimensional if the socle of $x$ is $\alpha^{\otimes 2}$, generated by $\eta: \alpha^{\otimes 2} \to x$, and zero otherwise. 

If $\alpha^{\otimes 2} \not\simeq D$ is not the socle of $P_\unit$, we have written $\SN_\II^\alpha(S^2\smallsetminus \ast)$ as a quotient of the zero vector space, so it is zero.

If $\alpha^{\otimes 2} \simeq D$, we have written it as a quotient of a one-dimensional vector space, generated by $\varepsilon\otimes \eta$, and we are left to identify the coend relations. We follow closely \cite[Sections 2.3, 4.1 and 5.3]{GKPmtrace}. The endomorphisms of an indecomposable projective $\End(P_\unit) = \Id \oplus \text{Nilpotents}$ split as identity and nilpotent endomorphisms. The identity acts as the identity on both Hom spaces above, whereas nilpotent elements act as zero on both \cite[Lemma 4.3]{GKPmtrace}. Hence the coend relations are trivial, and $\SN_\II^\alpha(S^2\smallsetminus \ast)$ is one-dimensional as claimed.

Let us now turn to point 2. The twisted string net module of the 2-sphere is a quotient of the module of the punctured 2-sphere by the relation \eqref{eq:extraRelationUnpuncturedS2}. We need to show that this relation is trivial if and only if $\alpha^{\otimes 2} \simeq D \text{ and } \apiv^2 = \Radford$. The Radford isomorphism $\Radford$ can be described as follows, see \cite{FSSRadfordS4, ShimizuUnimodularFTC}.

First, the distinguished invertible object $D$ is the inverse of value of the right Nakayama functor on the unit \cite[Sec. 4.3]{FSSRadfordS4} and \cite[Lemma 5.1]{ShimizuUnimodularFTC}:
$$\Nakayama(\unit):= \int^{x \in \BB}\Hom_\BB(\unit,x)^*\otimes x \simeq D^{-1}$$
By this coend description, $D^{-1}$ comes with a canonical isomorphism
$$N_x: \Hom_\BB(x,D^{-1}) \tilde\to \Hom_\BB(\unit, x)^* \ , \quad \text{for } x\in\Proj$$
or in other words, it comes with a canonical non-degenerate pairing 
$$\tau_x: \Hom_\BB(\unit, x) \otimes \Hom_\BB(x,D^{-1}) \to \Bbbk\ $$
natural in $x \in \Proj$. Given an isomorphism $\alpha^2 \simeq D$, this induces a non-zero a linear map 
$$\tau: \SN_\II^\alpha(S^2\smallsetminus\ast) \to \Bbbk$$
which is an isomorphism by the discussion above.

Second, the Radford isomorphism $\delta$ is induced by the module structure of $\Nakayama$ by \cite[Remark 4.11]{ShimizuRibbonStrOnDrinfeld}. Following the proof of \cite[Thm 3.18]{FSSRadfordS4}, we see that the mate $$\overline{\Radford}_x: D^{-1}\otimes x^{****} \simeq  x\otimes D^{-1}$$ of $\Radford$ is the unique isomorphism inducing the canonical isomorphism on $\Hom$-spaces given by 
\begin{multline*}
    \Hom(y, D^{-1}\otimes x^{****}) \simeq \Hom(y\otimes x^{***},D^{-1}) \overset{N_{y\otimes x^{***}}}\simeq \Hom(\unit, y\otimes x^{***})^* \simeq \Hom(x^{**},y)^*\\
    \simeq \Hom(\unit, x^*\otimes y)^* \overset{N_{x^*\otimes y}^{-1}}{\simeq} \Hom(x^*\otimes y,D^{-1}) \simeq \Hom(y, x \otimes D^{-1})
\end{multline*}
for $y \in \Proj$, where unlabeled isomorphisms are induced by dualities.
Rewriting this in terms of the pairing $\tau$, and graphical calculus for dualities we obtain that $\overline\Radford$ is the unique isomorphism such that $\tau$ is invariant under the transformation
\begin{equation*}
\begin{tikzpicture}[xscale = 1.5,yscale = 2, baseline = 0.9cm]
    \draw[very thick] (-0.2, 1) -- ++(0,-0.5) node[pos = 0.3, sloped]{\small $<$}node[pos = 0.3, left]{\small $y$} ;
    \draw[very thick] (0.2, 1) -- ++(0,-0.5) node[pos = 0.3, sloped]{\small $<$}node[pos = 0.3, right]{\small $x^{***}$} ;
    \node[rectangle, draw, very thick, fill=white, minimum width = 0.9cm, inner sep = 2pt] at (0,0.5) {\small $f$};
\end{tikzpicture}
\ \otimes \ 
\begin{tikzpicture}[xscale = 1.5,yscale = 2, baseline = 0.9cm]
    \draw[very thick] (-0.2, 0) to[out=90, in=-90] node[pos = 0.3, sloped]{\small $>$}node[pos = 0.3, left]{\small $y$} (-0.2,0.5) ;
    \draw[very thick] (0.2, 0) to[out=90, in=-90] node[pos = 0.3, sloped]{\small $>$}node[pos = 0.3, right]{\small $x^{***}$} (0.2,0.5) ;
    \draw[very thick] (0,0.5) to[out = 90, in = -90] node[pos = 0.7, sloped]{\small $>$} node[pos = 0.7, right]{\small $D^{-1}$}  (0,1);
    \node[rectangle, draw, very thick, fill=white, minimum width = 0.9cm, inner sep = 2pt] at (0,0.5) {\small $g$};
\end{tikzpicture}
\quad \leadsto \quad
\begin{tikzpicture}[xscale = 1.5,yscale = 2, baseline = 0.9cm]
    \draw[very thick] (0.2,0.5) to[out = 90, in = -135] node[pos = 0.6, sloped]{\small $>$}node[pos = 0.6, left]{} (0.3, 0.8) node{\small $\bullet$} to[out = -45, in = 13, in distance  = 25pt] (-0.2, 0.18) node{\small $\bullet$} to[out = 155, in=-100] node[pos = 0.7, sloped]{\small $>$}node[pos = 0.7, left]{\small $x^*$} (-0.3, 1);
    \draw[very thick] (0.1, 1) to[out = -90, in=90] node[pos = 0.3, sloped]{\small $>$}node[pos = 0.3, left]{\small $y$} (-0.2,0.5);
    \node[rectangle, draw, very thick, fill=white, minimum width = 0.9cm, inner sep = 2pt] at (0,0.5) {\small $f$};
\end{tikzpicture}
\ \otimes \ 
\begin{tikzpicture}[xscale = 1.5,yscale = -2, baseline = -1.1cm]
    \draw[very thick] (0.2,0.5) to[out = 90, in = -135] node[pos = 0.6, rotate = 100]{\small $>$} (0.3, 0.88) node{\small $\bullet$} to[out = -65, in = 13, in distance  = 25pt] node[pos = 0.3, rotate = 90]{\small $>$}node[pos = 0.3, right]{$x^{****}$} (-0.35, 0.15) node{\small $\bullet$} to[out = 155, in=-100] node[pos = 0.7, rotate = 100]{\small $>$}node[pos = 0.7, left]{\small $x^*$} (-0.3, 1);
    \draw[very thick] (0.1, 1) to[out = -90, in=90] node[pos = 0.3, rotate = 125]{\small $>$}node[pos = 0.3, left]{\small $y$} (-0.2,0.5);
    \draw[very thick] (0,0) -- ++(0,0.5) node[pos = 0.1, rotate = 90]{\small $>$}node[pos = 0.05, right]{\small $D^{-1}$} node[pos = 0.55, rectangle, draw, fill=white, rotate = 15, minimum width = 0.6cm, inner sep = 2pt]{\footnotesize $\overline{\Radford}$};
    \node[rectangle, draw, very thick, fill=white, minimum width = 0.9cm, inner sep = 2pt] at (0,0.6) {\small $g$};
\end{tikzpicture}
    \end{equation*}
In other words, using that $\tau$ is an isomorphism, the extra relations \eqref{eq:extraRelationUnpuncturedS2} are trivial if and only if $\apiv^2 = \Radford$ as claimed.

Finally, for point 3, any non-compact 3-cobordim $M:\Sigma\to \Sigma'$ factors as $$\Sigma \overset{M\smallsetminus B^3}{\longrightarrow} \Sigma'\sqcup S^2 \overset{\Id\sqcup B^3}{\longrightarrow} \Sigma'$$ 
so its image under a TQFT extending $\orSN_\II^\alpha$ factors through $\orSN_\II^\alpha(\Sigma'\sqcup S^2)$ which is zero.
\end{proof}
\begin{remark}
    The condition $\apiv^2=\Radford$ for the non-vanishing of $\SN^\alpha_\II(S^2)$ is exactly the condition exhibited in \cite{ShimizuRibbonStrOnDrinfeld} for $(\alpha,\apiv)$ to induce a ribbon structure on $\ZZ(\BB)$. It is also equivalent to what we call sphericality of the twisted pivotal structure below.
\end{remark}

\subsection{The 2-sphere and twisted modified traces}
The following definition closely follows the notion of right $(\alpha, \alpha^{-1})$-modified trace of \cite{GKPmtrace}. However, we really insist on not choosing a pivotal structure on $\BB$, and hence amend the notion as necessary. We change the name to avoid confusion and match more consistently our conventions. Our notion is equivalent to $\big(\alpha^{-2}\otimes(-)^{**}\big)$-twisted modified traces compatible with the regular $\BB$-module structure of \cite[Def. 3.1]{ShibataShimizuModifiedTracesandNakayama}. In the spherical case, they agree with the twisted traces of \cite[Def. 2.4]{SchweigertWoikeModifiedTracesandNakayama}. 
\begin{definition} \label{def:modifiedtrace} Let $\II$ be a tensor ideal in an $\alpha$-twisted pivotal category $\BB$. 
    An \textit{$\alpha$-twisted modified trace} $\mt$ is a collection of linear forms 
    $$\mt_x: \Hom_\BB(\alpha\otimes x,\alpha^{-1}\otimes x^{**}) \to \Bbbk$$
    for $x \in \II$ satisfying cyclicity and the right partial trace property:
    \begin{equation}
        \mt_x \left(
\begin{tikzpicture}[xscale = 1.5,yscale = 2, baseline = 0.9cm]
    \draw[very thick] (-0.4,0) -- ++(0,1) node[pos = 0.3, sloped]{\small $>$}node[pos = 0.3, left]{\small $\alpha$}node[pos = 0.9, sloped]{\small $>$}node[pos = 0.9, left]{\small $\alpha^{-1}$};
    \draw[very thick] (0,0) -- (0,1)node[pos = 0.3, rectangle, draw, fill=white]{\small $g$} node[pos = 0.1, sloped]{\small $>$}node[pos = 0.1, right]{\small $x$} node[pos = 0.5, sloped]{\small $>$}node[pos = 0.5, right]{\small $y$} node[pos = 0.9, sloped]{\small $>$}node[pos = 0.9, right]{\small $x^{**}$} ;
    \node[rectangle, draw, very thick, fill=white, minimum width = 0.9cm, inner sep = 2pt] at (-0.2,0.7) {\small $f$};
\end{tikzpicture}
        \right)
        \ =\ 
        \mt_y \left(
\begin{tikzpicture}[xscale = 1.5,yscale = 2, baseline = 0.9cm]
    \draw[very thick] (-0.4,0) -- ++(0,1) node[pos = 0.1, sloped]{\small $>$}node[pos = 0.1, left]{\small $\alpha$}  node[pos = 0.7, sloped]{\small $>$}node[pos = 0.7, left]{\small $\alpha^{-1}$} ;
    \draw[very thick] (0,0) -- (0,1)node[pos = 0.7, rectangle, draw, fill=white]{\small $g^{**}$} node[pos = 0.1, sloped]{\small $>$}node[pos = 0.1, right]{\small $y$} node[pos = 0.5, sloped]{\small $>$}node[pos = 0.5, right]{\small $x^{**}$} node[pos = 0.95, sloped]{\small $>$}node[pos = 0.95, right]{\small $y^{**}$} ;
    \node[rectangle, draw, very thick, fill=white, minimum width = 0.9cm, inner sep = 2pt] at (-0.2,0.3) {\small $f$};
\end{tikzpicture}
        \right)
    \end{equation}
    and
    \begin{equation}
        \mt_{x \otimes y} \left(
\begin{tikzpicture}[xscale = 1.5,yscale = 2, baseline = 0.9cm]
    \draw[very thick] (-0.4,0) -- ++(0,1) node[pos = 0.2, sloped]{\small $>$}node[pos = 0.2, left]{\small $\alpha$}node[pos = 0.8, sloped]{\small $>$}node[pos = 0.8, left]{\small $\alpha^{-1}$};
    \draw[very thick] (-0.1,0) -- ++(0,1) node[pos = 0.2, sloped]{\small $>$}node[pos = 0.2, left]{\small $x$} node[pos = 0.9, sloped]{\small $>$}node[pos = 1, left=-8pt]{\small $x^{**}$} ;
    \draw[very thick] (0.1,0) -- ++(0,1) node[pos = 0.2, sloped]{\small $>$}node[pos = 0.2, right]{\small $y$} node[pos = 0.9, sloped]{\small $>$}node[pos = 0.9, right]{\small $y^{**}$} ;
    \node[rectangle, draw, very thick, fill=white, minimum width = 1.2cm, inner sep = 2pt] at (-0.1,0.5) {\small $f$};
\end{tikzpicture}
        \right)
        \ =\ 
        \mt_y \left(
\begin{tikzpicture}[xscale = 1.5,yscale = 2, baseline = 0.9cm]
    \draw[very thick] (-0.4,0) -- ++(0,1) node[pos = 0.2, sloped]{\small $>$}node[pos = 0.2, left]{\small $\alpha$}node[pos = 0.8, sloped]{\small $>$}node[pos = 0.8, left]{\small $\alpha^{-1}$};
    \draw[very thick] (-0.1,0) -- ++(0,1) node[pos = 0.2, sloped]{\small $>$}node[pos = 0.2, left]{\small $x$} node[pos = 0.9, sloped]{\small $>$}node[pos = 1, left=-8pt]{\small $x^{**}$} ;
    \draw[very thick] (0.3,0.2) to[out = 135, in = -135] node[pos = 0.2, sloped]{\small $<$}node[pos = 0.2, below left=-2pt]{\small $y$} node[pos = 0.8, sloped]{\small $>$}node[pos = 0.8, above left=-4pt]{\small $y^{**}$} (0.3,0.8)node{\small $\bullet$} to[out = -45, in = 45] (0.3,0.2) node{\small $\bullet$};
    \node[rectangle, draw, very thick, fill=white, minimum width = 1.2cm, inner sep = 2pt] at (-0.1,0.5) {\small $f$};
\end{tikzpicture}
        \right)
    \end{equation}
We denote $\mathcal{T_\alpha}$ the vector space of $\alpha$-twisted modified traces.

    An $\alpha$-twisted modified trace $\mt$ is called \textit{spherical} if for every $f\in \Hom_\BB(\alpha\otimes x\otimes y,\alpha^{-1}\otimes x^{**}\otimes y^{**})$, $x, y\in\II$, we have:
\begin{equation}
\mt_{x \otimes y} \left(
\begin{tikzpicture}[xscale = 1.5,yscale = 2, baseline = 0.9cm]
    \draw[very thick] (-0.4,0) -- ++(0,1) node[pos = 0.2, sloped]{\small $>$}node[pos = 0.2, left]{\small $\alpha$}node[pos = 0.8, sloped]{\small $>$}node[pos = 0.8, left]{\small $\alpha^{-1}$};
    \draw[very thick] (-0.1,0) -- ++(0,1) node[pos = 0.2, sloped]{\small $>$}node[pos = 0.2, left]{\small $x$} node[pos = 0.9, sloped]{\small $>$}node[pos = 1, left=-8pt]{\small $x^{**}$} ;
    \draw[very thick] (0.1,0) -- ++(0,1) node[pos = 0.2, sloped]{\small $>$}node[pos = 0.2, right]{\small $y$} node[pos = 0.9, sloped]{\small $>$}node[pos = 0.9, right]{\small $y^{**}$} ;
    \node[rectangle, draw, very thick, fill=white, minimum width = 1.2cm, inner sep = 2pt] at (-0.1,0.5) {\small $f$};
\end{tikzpicture}
        \right)
        \quad=\quad
        \mt_y \left(
\begin{tikzpicture}[xscale =1.5,yscale = 2, baseline = 0.9cm]
    \draw[very thick] (0.3,0) -- ++(0,1.1) node[pos = 0.2, sloped]{\small $>$}node[pos = 0.2, right]{\small $y$} node[pos = 0.9, sloped]{\small $>$}node[pos = 1, right=-2pt]{\small $y^{**}$} ;
    \draw[very thick] (0,0)to[out = 110, in= -90] node[pos = 0.1, sloped]{\small $<$}node[pos = 0.1, right]{\small $\alpha$} (-0.4, 0.5) to[out = 90, in=-90] node[pos = 0.8, sloped]{\small $>$}node[pos = 1, left=-11pt]{\small $\alpha^{-1}$} (0,1.1);
    \draw[very thick] (-0.6,0.1) to[out = 30, in = -90] node[pos = 0.48,rectangle, draw, fill=white, sloped, rotate = 10]{} node[pos = 0.2, sloped]{\small $>$}node[pos = 0.15, above = -2pt]{\small $x^{**}$} (0,0.5) to[out = 90, in = -25] node[pos = 0.9, sloped]{\small $<$}node[pos = 0.9, right]{\small $x$}node[pos = 0.48,rectangle, draw, fill=white, sloped, rotate = -10]{} (-0.6,1) node{\small $\bullet$} to[out = -155, in = 155] (-0.6,0.1) node{\small $\bullet$};    
    \node[rectangle, draw, very thick, fill=white, minimum width = 1.5cm, inner sep = 2pt] at (-0.1,0.5) {\small $f$};
\end{tikzpicture}
        \right)
    \end{equation}
    We denote $\SS_\alpha\subseteq \mathcal{T_\alpha}$ the space of spherical $\alpha$-twisted modified traces.

A \textit{twisted spherical category} is a twisted pivotal category $\BB$ equipped with a non-zero spherical twisted modified trace on $\II$.
\end{definition}

The following result is well-known for non-twisted admissible string net modules. It was first announced by Reutter--Walker, but see also \cite{CGPAdmissbleskein, RunkelSchweigertThamExciAdmSkeins}.
\begin{proposition}
    The $\alpha$-twisted string net module of the punctured 2-sphere is dual to the space of $\alpha$-twisted modified traces.
    $$\SN_\II^\alpha( S^2\smallsetminus\ast)^* \simeq \mathcal{T}_\alpha \ .$$ 

    The $\alpha$-twisted string net module of the 2-sphere is dual to the subspace of spherical $\alpha$-twisted modified traces
    $$\SN_\II^\alpha( S^2)^* \simeq \SS_\alpha  \ . $$
\end{proposition}
\begin{proof}
    Let $\varphi\in \SN_\II^\alpha( S^2\smallsetminus\ast)^*$ be a linear form, we define the associated modified trace $\mt^\varphi$ by setting 
\begin{equation*}
    \mt^\varphi_x(f):= \varphi \left(
    \begin{tikzpicture}[scale = 0.7, baseline = 0pt]
    \clip (0,0) circle(2);
    \def\lines{8}
    \foreach \y in {1,...,\lines}{
    \draw[gray!70] (0,\y/\lines) arc(-90:270:{20/(\y*\y)});
    }
    \draw[gray!70] (-2,0) --(2,0);
    \foreach \y in {-1,...,-\lines}{
    \draw[gray!70] (0,\y/\lines) arc(90:540:{20/(\y*\y)});
    }
    \node[star, star point ratio=2.25, fill=black, inner sep = 1pt] at (0,1.3) {};
    \node[star, star point ratio=2.25, fill=black, inner sep = 1pt] at (0,-1.3) {};
    \draw[very thick] (0.9,-0.9) to[out = 135, in = -135]node[pos = 0.1, left]{$x$} node[pos = 0.15, sloped]{$<$}node[pos = 0.9, above]{$x^{**}$} node[pos = 0.9, sloped]{$>$}(0.9,0.9)node{\small $\bullet$} to[out = -45, in = 45] (0.9,-0.9)node{\small $\bullet$};
    \draw[very thick] (0,-1.3) -- (0,1.3) node[pos = 0.2, left]{$\alpha$} node[pos = 0.2, sloped]{$>$} node[pos = 0.5, rectangle, draw, fill=white, minimum width = 0.9cm, right = -6pt]{$f$} node[pos = 0.8, left]{$\alpha^{-1}$} node[pos = 0.8, sloped]{$>$};
    \end{tikzpicture}
    \right)\ , \quad f\in \Hom_\BB(\alpha\otimes x, \alpha^{-1}\otimes x^{**})
\end{equation*}
It satisfies the cyclicity and partial trace properties as they hold at the level of string nets.

Now let $\mt$ be an $\alpha$-twisted modified trace. We define the linear form $\varphi_t\in \SN_\II^\alpha( S^2\smallsetminus\ast)^*$ as follows. We use Proposition \ref{prop:coendFormulaForS2} to present $\SN_\II^\alpha( S^2\smallsetminus\ast)$ as a quotient of the direct sum of $\Hom_\BB(\alpha, x) \otimes \Hom_\BB(x,\alpha^{-1})$ over every $x\in\II$, and we set 
\begin{equation*}
    \varphi_{\mt}(f \otimes g) := \mt_{{}^*x}
    \left(
\begin{tikzpicture}[xscale = 2.5,yscale = 2.5, baseline = 1.1cm]
    \draw[very thick] (0,0) to[out = 90, in = -135] node[pos = 0.5, rectangle, draw, fill=white, inner sep = 1pt, minimum width = 0.5cm]{\small $f$} node[pos = 0.1, sloped]{\small $>$}node[pos = 0.1, left]{\small $\alpha$}  node[pos = 0.9, sloped]{\small $>$}node[pos = 0.9, left]{\small $x$} (0.2,0.4) node{\small $\bullet$} to[out = -45, in = 90] node[pos = 0.5, sloped]{\small $<$} node[midway, right]{${}^*x$} (0.4,0) ;
    \draw[very thick] (0,1) to[out = -90, in = 135] node[pos = 0.5, rectangle, draw, fill=white, minimum width = 0.5cm]{\small $g$} node[pos = 0.1, sloped]{\small $<$}node[pos = 0.1, left]{\small $\alpha^{-1}$}  node[pos = 0.85, sloped]{\small $<$}node[pos = 0.9, left]{\small $x$} (0.2,0.6) node{\small $\bullet$} to[out = 45, in = -90]node[pos = 0.5, sloped]{\small $>$} node[midway, right]{$x^*$} (0.4,1) ;
\end{tikzpicture}
        \right) 
        \ , \quad f\in \Hom_\BB(\alpha, x) \ ,\ g \in \Hom_\BB(x,\alpha^{-1})
\end{equation*}
It satisfies the coend relations, i.e. is a well-defined linear form on $\SN_\II^\alpha( S^2\smallsetminus\ast)$, by cyclicity of modified traces.

It is straightforward to check that these two maps are inverse to each other using snake relations for one composition and cyclicity and partial trace relations for the other. 

One can check that under this isomorphism, the additional coend relations on $\SN_\II^\alpha( S^2)$ from Proposition \ref{prop:coendFormulaForS2} correspond to the sphericality condition of Definition \ref{def:modifiedtrace}. 
\end{proof}


\subsection{Understanding the Mapping Class Group action}
As a consequence of Theorem \ref{thm:main_thm}, we obtain representations of mapping class groups of surfaces. Let us spell out here what these representations are and how to compute them. We will investigate more explicit formulas and computations in future work.

\begin{definition}
    Let $\Sigma$ be a compact oriented surface and $\operatorname{MCG}(\Sigma)$ be the group of isotopy classes of diffeomorphisms of $\Sigma$ which are the identity on the boundary. Choose an arbitrary excellent Morse function $f$ on $\Sigma$ such that the boundary components are level sets, so we may view $\Sigma$ as a foliated surface.
    For any labeling $X \in \SN_\II(\partial \Sigma)$, we have an representation $$\rho_{\Sigma,X}:\operatorname{MCG}(\Sigma)\to \End(\SN_\II^\alpha(\Sigma;X))$$
    mapping a diffeomorphism $\Phi$ to the $X$-component of the natural transformation $\orSN_\II^\alpha(\Phi)$.
    Different choices for $f$ will induce conjugate representations.
\end{definition}
    Given a diffeomoprhism $\Phi \in \operatorname{MCG}(\Sigma)$, the action is computed as follows. First $\Phi$ induces a foliation-preserving diffeomoprhism $\Phi: (\Sigma,f) \to (\Sigma, f\circ \Phi)$ between $\Sigma$ with its chosen foliation and $\Sigma$ with foliation transported under $\Phi$. This diffeomorphism acts on twisted string nets in a natural way. Then, one can choose a generic path of functions between $f$ and $f\circ \Phi$, which will only have cusp and crossing singularities. At each of these singularities, one needs to apply one of the maps defined in Section \ref{subsec:generators}. 


\section{Examples}\label{sec:Examples}
We will discuss examples coming from Hopf algebras below, which allow for explicit computations and on which the literature is more abundant.

We give examples where there are no pivotal structures but there are twisted ones, showing that our construction allows to study a strictly larger class of examples. 

We also give examples of non-unimodular categories (for which pivotal structure may exist but which have no hope of extending to 3d because the string net module of the 2-sphere is trivial) which admit twisted pivotal structures for which the twisted string net modules of the 2-sphere is non-zero. We expect that for such examples, the twisted string net skein theory can be extended to a non-compact 3-TQFT.

\subsection{Pivotal tensor categories}
When the category $\BB$ actually comes equipped with a (non-twisted) pivotal structures, we have an alternative description of twisted pivotal structures in terms of the Drinfeld center.
\begin{lemma}
Let $\BB$ be a pivotal tensor category. Then $\alpha$-twisted pivotal structures on $\BB$ are equivalent to invertible elements of $\ZZ(\BB)$ with trivial twist.
\end{lemma}

\begin{proof}
Let $\BB$ be pivotal with pivotal isomorphism $j_{x}:x\to x^{\ast\ast}$. Then $\alpha$-twisted quasi-pivotal structures on $\BB$ are equivalent to invertible elements of $\ZZ(\BB)$ by the map $(\alpha,\apiv) \mapsto (\alpha, \gamma^{\apiv})$ where 
\begin{align*}
    \gamma^{\apiv}_{x} = \apiv_{x} \circ (j_{x} \otimes \Id_{\alpha}).
\end{align*}
Under this bijection, the twist relation for an $\alpha$-twisted pivotal structure becomes
\begin{equation}
    \begin{tikzpicture}[baseline=1cm]
        \draw (0,-0.5) -- (0,2.5) node[pos = 0.5, sloped]{$>$} node[pos = 0.5, left]{$\alpha^{\ast}$};
        \draw (1,-0.5) -- (1,0.5) node[pos = 0.5, sloped]{$>$} node[pos = 0.5, right]{$\alpha^{\ast\ast}$};
        \draw (1,0.5) -- (1,1.5) node[pos = 0.5, sloped]{$>$} node[pos = 0.5, right]{$\alpha$};
        \draw (1,1.5) -- (1,2.5) node[pos = 0.5, sloped]{$>$} node[pos = 0.5, right]{$\alpha$};
        \draw (2,-1) -- (2,1.5) node[pos = 0.5, sloped]{$>$} node[pos = 0.5, right]{$\alpha$};
        \draw (2,1.5) -- (2,3) node[pos = 0.5, sloped]{$>$} node[pos = 0.5, right]{$\alpha$};
        \node[rectangle, draw, fill=white, minimum width=1.5cm] at (0.5,-0.5) {$\rcoev$};
        \node[rectangle, draw, fill=white, minimum width=1.5cm] at (0.5,2.5) {$\rev$};
        \node[rectangle, draw, fill=white, minimum width=1.5cm] at (1.5,1.5) {$\gamma^{\apiv}_{\alpha}$};
        \node[rectangle, draw, fill=white, minimum width=0.5cm] at (1,0.5) {$j_{\alpha}^{-1}$};
    \end{tikzpicture}
    \qquad = \qquad
    \begin{tikzpicture}[baseline=1cm]
        \draw (0,-1) -- (0,3) node[pos = 0.5, sloped]{$>$} node[pos = 0.5, right]{$\alpha$};
    \end{tikzpicture}
\end{equation}
which is the condition that $(\alpha,\gamma^{\apiv}) \in \ZZ(\BB)$ has trivial twist.
\end{proof}

This allows us to classify twisted pivotal structures whenever the Drinfeld center in well-understood.

\begin{example}
Let $G$ be a finite group and let $\Vect_G$ denote the fusion category of
finite-dimensional $G$-graded $\Bbbk$-vector-spaces
\[
V = \bigoplus_{g \in G} V_g\ , \quad\text{ with tensor product } \quad (V \otimes W)_g
=
\bigoplus_{h\cdot k = g} V_h \otimes W_k.
\]
with strict associativity.

The Drinfeld center of $\Vect_{G}$ admits a concrete description as the category of $G$-graded vector spaces
together with a $G$-action compatible with conjugation. An object of $\ZZ(\Vect_G)$ is then a $G$-graded vector space
equipped with an action of $G$ such that
\[
h \cdot V_g \subseteq V_{h g h^{-1}}
\qquad \text{for all } g,h \in G.
\]

The simple objects of $\ZZ(\Vect_G)$ are classified by pairs $(C,\rho)$,
where $C$ is a conjugacy class in $G$ and $\rho$ is an irreducible
representation of the centralizer $Z_G(g)$ for any $g \in C$. The invertible objects of $\ZZ(\Vect_G)$ are classified by a pair $(z,\chi)$ where $z\in Z(G)$ is a central element in $G$ and $\chi:G \to k^{\times}$ is a character for $G$. The twist of $(z,\chi)$ is given by $\chi(z)\Id$, which is trivial if and only if $\chi(z) = 1$.

We can see that $\SN_{\Vect_G}^{(z,\chi)}(S^2 \smallsetminus \ast)$ is non-zero if and only if $z^2 = 1$ and $\SN_{\Vect_G}^{(z,\chi)}(S^2)$ is non-zero if and only if we also have $\chi^2 = 1$. 


\end{example}

\subsection{Pairs in involution}
The following definition can be found for example in \cite{HalbigGenTaft} (with appropriate inversions coming from different conventions for pivotal structures).
\begin{definition}
    Let $H$ be a finite-dimensional Hopf algebra. A \textit{pair in involution} in $H$ is a pair $(g,\alpha)$ of a group-like element $g\in H$ and an invertible module $\alpha:H\to \Bbbk$ satisfying 
    \begin{equation}\label{eq:pairininvolution}
        S^2h = \alpha(h_{(1)}) \alpha^{-1}(h_{(3)})\, g^{-1}\, h_{(2)}\, g
    \end{equation}
    for every $h\in H$. It is called \textit{modular} if $\alpha(g)=1$.
\end{definition}
Let $\BB = H{-}mod$ be the rigid monoidal category of finite-dimensional left $H$-modules.
Recall that duality in $\BB$ is given by duals as vector spaces with action twisted by $S$ for $V^*$ and by $S^{-1}$ for ${}^*V$. Evaluation and coevaluation are given by those in vector spaces. 
Below, we identify $V$ and $V^{**}$ as vector spaces and we write $v^{**} \in V^{**}$ corresponding to $v\in V$.
The following can essentially be extracted from \cite{HalbigGenTaft}.
\begin{proposition}\label{prop:piiistwistedpivotal}
    Twisted quasi-pivotal structures on $\BB$ are in correspondence with pairs in involutions in $H$ by 
    $$(g, \alpha) \longleftrightarrow \left[ \apiv^{g, \alpha}_V: \begin{array}{rcl}
        V^{**} \otimes \alpha & \to & \alpha\otimes V  \\
       v^{**} \otimes 1  & \mapsto & 1\otimes g\cdot v
    \end{array} \right]$$
    Moreover, $\apiv^{g,\alpha}$ is twisted pivotal if and only if $(g, \alpha)$ is modular. 
\end{proposition}
\begin{proof}
First, given a pair in involution $(g,\alpha)$, let us check that the definition of $\apiv^{g,\alpha}$ above is a twisted quasi-pivotal structure. It is a map of $H$-modules by \eqref{eq:pairininvolution}. It is natural because a map of $H$-modules satisfies $f(g\cdot h) = g\cdot f(h)$. It is compatible with the monoidal structure because $g$ is group-like. The twist is easy to compute as every evaluation and coevaluation is the standard one in vector spaces, and is indeed given by the scalar $\alpha^{-1}(g)$ times the identity. As $g$ is group-like, $\alpha^{-1}(g) = (\alpha(g))^{-1}$ is equal to 1 if and only if $\alpha(g)=1$.

    Now, let $\apiv$ be an $\alpha$-twisted quasi-pivotal structure. We evaluate $\apiv$ on the regular module $H$ and define $g$ as the image of the unit $1\in H$, identifying $\alpha\simeq\Bbbk$ and $H^{**}\simeq H$ as vector spaces:
    $$\begin{array}{rcl}
        \apiv_H: H^{**} \otimes \alpha & \to & \alpha\otimes H  \\
       1^{**} \otimes 1  & \mapsto & 1\otimes g
    \end{array} $$
For every $h\in H$, there is a unique $H$-module endomorphism $r_h:H\to H$ mapping $1_H$ to $h$. Naturality of $\apiv$ implies that 
$$\apiv_H(h^{**}\otimes 1) = \apiv_H \circ (r_h^{**}\otimes \Id_\alpha)(1^{**}\otimes 1) = (\Id_\alpha\otimes r_h)\circ \apiv_H(1^{**}\otimes 1)=1 \otimes gh \ .$$
Hence $\apiv_H$ is of the form claimed in the statement. The fact that this is a map of $H$-comodules implies precisely that $(g, \alpha)$ satisfies \eqref{eq:pairininvolution}. This is easier to see using the duality to turn $\apiv_H$ into a map $H^{**} \to \alpha \otimes H \otimes \alpha^{-1}$.
For any module $V$ and $v\in V$, there is also a unique map $H\to V$ sending 1 to $v$ and we conclude that $\apiv_V$ must be of the form claimed. Compatibility of $\apiv$ with the monoidal structure implies that $g$ is group-like. Hence $(g, \alpha)$ is a pair in involution.
\end{proof}

\subsection{The Taft algebra}
\begin{definition}
    Let $N\geq 1$ be an integer and $q \in \mathbb C^\times$ a primitive $N$-th root of unity. The Taft algebra is
    $$H_N = \left\langle g,x \vert g^N=1, x^N=0, gx=qxg \right\rangle$$
    with Hopf algebra structure
$$\Delta g = g\otimes g \ , \quad \Delta x = 1\otimes x + x \otimes g\ , \quad Sg = g^{-1} \ , \quad Sx = -xg^{-1}\ .$$
We denote $\BB = H_N{-}mod$ its category of finite-dimensional modules which is rigid monoidal and $\II \subseteq \BB$ the tensor ideal of projective objects.
\end{definition}
\noindent
Its simple modules are the invertible modules of weight $k  \in \Zz_{/N}$:
$$T_k = \mathbb C\langle t_k \rangle, \ g\cdot t_k = q^k t_k, \ x\cdot t_k = 0$$
Its group-like elements are the powers of $g$, $g^k$ for $k \in \Zz_{/N}$. 
Its indecomposable modules are the modules of highest weight $k \in \Zz_{/N}$ and order of $x$ being $\ell \leq N$. It has dimension $\ell$.
$$V_{k,\ell} = \mathbb C\langle v_k, v_{k+1},\dots,v_{k+\ell-1}\rangle,\ g\cdot v_i = q^i v_i, \ x\cdot v_i = v_{i+1}, \ x\cdot v_{k+\ell-1} = 0\ .$$
Its indecomposable projective modules are the highest-dimensional $P_k = V_{k,N}$.
The distinguished invertible object is the socle of $P_0$, namely 
$$D = T_{N-1}$$
The double dual $V_{k,\ell}^{**}$ is canonically isomorphic to $V_{k,\ell}$ as a vector space, but the action is now twisted by the square of the antipode, namely $g\cdot v_i^{**} = q^iv_i^{**}$ and $x\cdot v_i^{**} = qv_{i+1}^{**}$.
\begin{proposition}
    Let $c \in \Zz_{/N}$ and $\alpha = T_c$, then there exists a unique $T_c$-twisted quasi-pivotal structure on $\BB$ given by 
    $$\apiv_c:\begin{array}{rcl}
     V_{k,\ell}^{**} \otimes T_c &\to& T_c \otimes V_{k,\ell} \\
       v_i^{**} \otimes t_c &\mapsto  & q^{-(c+1)i} t_c\otimes v_i 
    \end{array}$$
    It is a $T_c$-twisted pivotal structure (i.e. satisfies the twist relation) if and only if $$c(c+1) \equiv 0 \ \ (\text{mod }N) \ .$$
\end{proposition}
\begin{proof}
By Proposition \ref{prop:piiistwistedpivotal}, any $\apiv$ must be of the form $ v_i^{**} \otimes t_c \mapsto   q^{di} t_c\otimes v_i $ for some $d \in \Zz_{/N}$. Now on the LHS, we have $$x \cdot (v_i^{**} \otimes t_c) = (x\cdot v_i^{**})\otimes (g\cdot t_c) = q^{c+1} v_{i+1}^{**} \otimes t_c$$
    whereas, on the RHS we have $x \cdot (t_c\otimes v_i) =t_c\otimes v_{i+1}$, so must have $c+1+d(i+1) \equiv di$, i.e. $d = -(c+1)$ as claimed.
The twist is given by the scalar $q^{c(c+1)}$ times the identity.
\end{proof}
\begin{proposition}
    Assume $N>1$, then for any choice of $\alpha$-twisted pivotal structure $\apiv$ on $\BB$, we have 
    $$\SN_\II^\alpha(S^2 \smallsetminus \ast) = 0 = \SN_\II^\alpha(S^2) \ .$$
\end{proposition}
\begin{proof}
Let $\alpha = T_c$ be any invertible object. By the proposition above, there is only one $\alpha$-twisted quasi-pivotal structure, and for it to be twisted pivotal we must have $c(c+1)\equiv 0$. Now, by Corollary \ref{cor:dimSNS2}, $\SN_\II^\alpha(S^2 \smallsetminus \ast)$ is zero unless $2c \equiv N-1 \equiv -1$. In particular, $c$ is invertible in $\Zz_{/N}$, so the equation above implies $c \equiv -1$, hence $2c \equiv -2 \equiv -1$ which contradicts our assumption. No $\alpha$-twisted pivotal structure can have $\alpha^{\otimes 2}\simeq D$, hence $\SN_\II^\alpha(S^2 \smallsetminus \ast)$ and $\SN_\II^\alpha(S^2)$ are zero.
\end{proof}
To see whether twisted pivotal structures can solve the problem of vanishing on $S^2$, we will need to consider more general examples.

\subsection{Generalized Taft algebras of type A1$\times$A1}
The following examples appear in \cite[Thm 3.14]{HalbigGenTaft}, where the author studies in detail modular pairs in involutions in these Hopf algebras, which correspond to twisted quasi-pivotal structures. We recall the main computations.
\begin{definition}
    The generalized Taft algebra of type A1$\times$A1 at level $N\geq 2$ and with parameters $\begin{pmatrix} a_1 & a_2 \\ b_1 &b_2\end{pmatrix}$ such that $a_1b_2+b_1a_2\equiv 0$ (mod $N$) and $a_1a_2 \not\equiv 0,\, b_1b_2\not\equiv 0$, is:
    \begin{equation*}
        H^N_{a_1,a_2,b_1,b_2} = \left\langle g,x,y \vert g^N = 1, x^n =0, y^m=0, gx = q^{a_2}xg, gy = q^{b_2}yg, xy = q^{a_1b_2}yx \right\rangle
    \end{equation*}
    where $n$ and $m$ are the orders of $a_1a_2$ and $b_1b_2$ in $\Zz_{/N}$ and $q$ is a primitive $N$-th root of unity. It has dimension $N\times n \times m$. It has a Hopf algebra structure given by 
    \begin{gather*}
        \Delta g = g\otimes g \ , \quad \Delta x = 1\otimes x + x \otimes g^{a_1}\ , \quad \Delta y = 1\otimes y + y \otimes g^{b_1} \ , \\ Sg = g^{-1} \ , \quad Sx = -xg^{-a_1} \ , \quad Sy = -yg^{-b_1}\ .
    \end{gather*}
We denote $\BB = H^N_{a_1,a_2,b_1,b_2}{-}mod$ its category of finite-dimensional modules which is rigid monoidal and $\II \subseteq \BB$ the tensor ideal of projective objects.
\end{definition}
Its simple modules are the invertible modules of weight $k\in \Zz_{/N}$:
$$T_k = \mathbb C\langle t_k \rangle, \ g\cdot t_k = q^k t_k, \ x\cdot t_k = 0, \ y\cdot t_k = 0$$
Its indecomposable modules are the modules of highest weight $k \in \Zz_{/N}$ and where $x$ and $y$ have order $\ell\leq n$ and $h\leq m$. It has dimension $\ell+h$.
$$V_{k,\ell, h} = \mathbb C\langle (x^iy^jv_{k})_{\substack{0\leq i\leq \ell-1 \\ 0 \leq j \leq h-1}}\rangle,\ g\cdot x^iy^jv_{k} = q^{k+a_2i+b_2j} x^iy^jv_{k}, \ x\cdot x^iy^jv_{k} = x^{i+1}y^jv_{k}, \ y\cdot x^iy^jv_{k} = q^{-ia_1b_2} x^{i}y^{j+1}v_{k}\ .$$
Its indecomposable projective modules are the highest-dimensional $P_k = V_{k,n,m}$.
The distinguished invertible object is the socle of $P_0$, namely 
$$D = T_{a_2(n-1)+b_2(m-1)}$$
On double duals, $x$ acts as $S^2x = q^{a_1a_2}x$ and $y$ as $S^2y = q^{b_1b_2}y$.
\begin{proposition}
Let $c\in \Zz_{/N}$, then $T_c$-twisted quasi-pivotal structures are of the form
 $$\apiv^{c,d}:\begin{array}{rcl}
     V^{**} \otimes T_c &\to& T_c \otimes V \\
       v^{**} \otimes t_c &\mapsto  &  t_c\otimes g^d\cdot v
    \end{array}$$
for some $d \in \Zz_{/N}$ satisfying
\begin{equation}\label{eq:cdgivesquasipiv}
ca_2+da_1+a_1a_2\equiv 0 \quad\text{and}\quad cb_2+db_1+ b_1b_2\equiv 0 \ .
\end{equation}
It is a twisted pivotal structure if and only if
\begin{equation}\label{eq:cdgivespiv}
cd \equiv 0 \ .
\end{equation}
\end{proposition} 
\begin{proof} Let $\alpha = T_c$ which we see as a map $\alpha:H\to \Bbbk$ with $\alpha(g) = q^c$ and $\alpha(x)=\alpha(y)=0$. By Proposition \ref{prop:piiistwistedpivotal}, any $T_c$-twisted quasi-pivotal structure has to be induced by a group-like $g^d$ for some $d \in \Zz_{/N}$, hence be of the form claimed. Moreover, $g^d$ must satisfy \eqref{eq:pairininvolution} for all $h\in H$. This equation is automatic for $h = g$, and gives the two relations of \eqref{eq:cdgivesquasipiv} for $h=x$ and $h=y$ respectively:
    \begin{gather*}
       q^{a_1a_2}x = S^2x = \alpha(x_{(1)}) \alpha^{-1}(x_{(3)})\, g^{-d}\, x_{(2)}\, g^d = 1\cdot q^{-ca_1} \cdot q^{-da_2}x \\
       q^{b_1b_2}y = S^2y = \alpha(y_{(1)}) \alpha^{-1}(y_{(3)})\, g^{-d}\, y_{(2)}\, g^d = 1\cdot q^{-cb_1} \cdot q^{-db_2}y
    \end{gather*}
These three elements generate the Hopf algebra under multiplication, and both sides of \eqref{eq:cdgivesquasipiv} are multiplicative, so it holds for any $h$.
The twist is given by $\alpha^{-1}(g^d) = q^{-cd}$ which is trivial if and only if $cd \equiv 0$.
\end{proof}
\begin{proposition}
Choose $c, d \in \Zz_{/N}$ satisfying \eqref{eq:cdgivesquasipiv} and \eqref{eq:cdgivespiv} and let $\alpha = T_c$ with twisted pivotal structure $\apiv^{c,d}$ as above, then the $\alpha$-twisted string net module $\SN_\II^\alpha(S^2\smallsetminus \ast)$ is nonzero if and only if
\begin{equation}
    2c \equiv a_2(n-1)+b_2(m-1) \label{eq:NZ1}
\end{equation}
and the $\alpha$-twisted string net module $\SN_\II^\alpha(S^2)$ is nonzero if and only if we moreover have
\begin{equation}
    2d \equiv a_1(n-1)+b_1(m-1)\ .\label{eq:NZ2}
\end{equation}
\end{proposition}
\begin{proof}
    By Corollary \ref{cor:dimSNS2} the string net module $\SN_\II^\alpha(S^2\smallsetminus \ast)$ is nonzero if and only if $T_c^2 = T_{2c}$ agrees with the distinguished invertible $D = T_{a_2(n-1)+b_2(m-1)}$, from which \eqref{eq:NZ1} follows. This is equivalent to $\alpha^2$ being the distinguished group-like element of $H^*$. The canonical trivialization of the quadruple dual is given by Radford's $S^4$ formula. The trivialization of the quadruple dual determined by $\apiv^{c,d}$ is given by
    \begin{align*}
    V^{****} \otimes D &\to D \otimes V \\
    v^{****}\otimes 1 &\mapsto 1\otimes g^{d}\cdot S^2(g^d)\cdot  v = 1\otimes g^{2d}\cdot v
    \end{align*}
    which agrees with the trivialization from Radford's $S^4$ formula if and only if $g^{2d}$ is the distinguished group like element of $H$ which is shown to be $g^{a_1(n-1)+b_1(m-1)}$ in \cite{HalbigGenTaft}, giving \eqref{eq:NZ2}.
\end{proof}

    It is easy to check these conditions of some equalities in $\Zz_{/N}$ on a computer. We thank Sebastian Halbig for providing us with a very long list of examples satisfying various requirements.
\begin{example}\label{ex:noPivotalStructure}
    Let $N=3$ and $\begin{pmatrix} a_1 & a_2 \\ b_1 &b_2\end{pmatrix}=\begin{pmatrix} 1 & 1 \\ 1 & 2 \end{pmatrix}$. Then there are no pivotal structures on $\BB$. However, for $(c,d) = (2,0)$, $\apiv^{c,d}$ is a twisted pivotal structure on $\BB$.

    Indeed pivotal structures correspond to twisted-pivotal structures for $\alpha = \unit$, i.e. for $c=0$. Then \eqref{eq:cdgivesquasipiv} implies that $d+1 \equiv d+2 \equiv 0$.
\end{example}
\begin{example}\label{ex:NonunimodTwSpherical}
Let $N=9$ and $\begin{pmatrix} a_1 & a_2 \\ b_1 &b_2\end{pmatrix}=\begin{pmatrix} 1 & 3 \\ 5 & 3 \end{pmatrix}$. Take $(c,d) = (6,6)$, then $\apiv^{c,d}$ is a twisted pivotal structure on $\BB$ and satisfies 
$$\SN_\II^\alpha(S^2)\neq 0$$ 
even though $\BB$ is not unimodular. Indeed, $c$ and $d$ satisfy all the equations above, and $D = T_3 \neq \unit$.
\end{example}
By a quick computer check, we see that at $N=9$, there are 24 values for $a_1,a_2,b_1,b_2$ that satisfy these properties. At $N=8$ there are 4, and moreover they support multiple twisted spherical structures, e.g. for $\begin{pmatrix} a_1 & a_2 \\ b_1 &b_2\end{pmatrix}=\begin{pmatrix} 2 & 2 \\ 6 & 2 \end{pmatrix}$ one can take any $(c,d) \in \{(2, 0), (2, 4), (6, 0), (6, 4)\}$.

At $N = 36$, there are 46800 Hopf algebras. 16092 of them have usual pivotal structures and 26928 twisted pivotal structures. 5130 are unimodular, 78 have spherical structures and 1044 have twisted spherical structures (888 of which are non-unimodular).

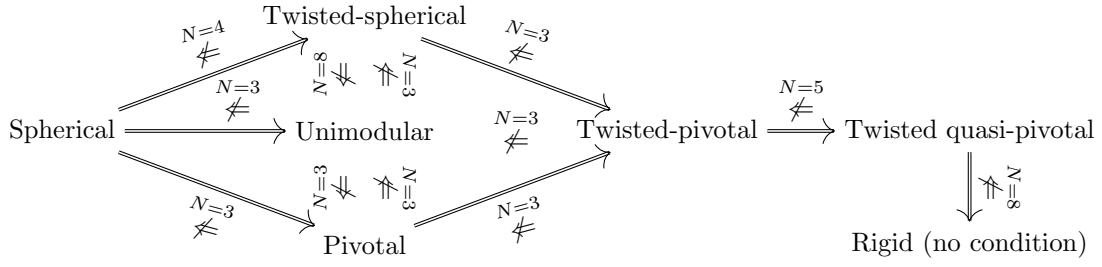
\begin{figure}
    \centering
\begin{tikzpicture}[xscale = 4, yscale = 1.5, rotate = 90]
    \node (S) at (0,1) {Spherical};
    \node (P) at (-1,0) {Pivotal};
    \node (U) at (0,0) {Unimodular};
    \node (TS) at (1,0) {Twisted-spherical};
    \node (TP) at (0,-1) {Twisted-pivotal};
    \node (Q) at (0,-2) {Twisted quasi-pivotal};
    \node (R) at (-1,-2) {Rigid (no condition)};
    \draw[->, double] (S) -- (P) node[pos = 0.5, below, rotate = -15]{$\overset{N=3}{\not\Leftarrow}$};
    \draw[->, double] (S) -- (U) node[pos = 0.7, above, rotate = 0]{$\overset{N=3}{\not\Leftarrow}$};
    \draw[->, double] (S) -- (TS) node[pos = 0.5, above, rotate = 15]{$\overset{N=4}{\not\Leftarrow}$};
    \draw[->, double] (TS) -- (TP) node[pos = 0.5, above, rotate = -15]{$\overset{N=3}{\not\Leftarrow}$};
    \draw[->, double] (P) -- (TP) node[pos = 0.5, below, rotate = 15]{$\overset{N=3}{\not\Leftarrow}$};
    \draw[->, double] (TP) -- (Q) node[pos = 0.5, above, rotate = 0]{$\overset{N=5}{\not\Leftarrow}$};
    \draw[->, double] (Q) -- (R) node[pos = 0.5, above, rotate = -90]{$\overset{N=8}{\not\Leftarrow}$};
    \node[rotate = 90, above] at (0.5,0) {$\overset{N=8}{\not\Leftarrow}$};
    \node[rotate = -90, above] at (0.5,0) {$\overset{N=3}{\not\Leftarrow}$};
    \node[rotate = 90, above] at (-0.5,0) {$\overset{N=3}{\not\Leftarrow}$};
    \node[rotate = -90, above] at (-0.5,0) {$\overset{N=3}{\not\Leftarrow}$};
    \node at (0,-0.5) {$\overset{N=3}{\not\Leftarrow}$};
\end{tikzpicture}
    \caption{Implications and non-implications of various properties and existence of different structures on a rigid category $\BB$. The labels $N=\, ?$ refer to the first value of $N$ such that a counter-example of the form above can be found. We do not know whether unimodular categories always admit twisted pivotal structures.}
    \label{fig:placeholder}
\end{figure}




\small
\bibliography{mybib}{}
\bibliographystyle{alpha}
\end{document}